\documentclass[11pt,reqno]{amsart}
\UseRawInputEncoding
\pdfoutput=1

\usepackage{graphicx}
\usepackage{amsmath,amssymb,amsthm,amsfonts}
\usepackage{enumerate}
\usepackage{tikz}
\usepackage{mathrsfs}
\usepackage{bbm}
\numberwithin{equation}{section}
\usepackage{hyperref}
\usepackage{marginnote}

\usepackage{color}

\newcommand{\R}{\mathbb{R}}

\newtheorem{theorem}{Theorem}[section]
\newtheorem{corollary}[theorem]{Corollary}
\newtheorem{lemma}[theorem]{Lemma}
\newtheorem{proposition}[theorem]{Proposition}
\newtheorem{remark}[theorem]{Remark}

\newtheorem{assumption}{Assumption}[section]

\def\v{\varepsilon}

\def\t{\theta}

\def\a{\alpha}

\def\di{\displaystyle}

\begin{document}
\title[Nonlinear Landau damping]{
Nonlinear Landau damping for the two-species screened Vlasov-Poisson system with large initial distributions
}

\author[Y. Wang]{Yi Wang}
\address[Y. Wang]{State Key Laboratory of Mathematical Sciences and Institute of Applied Mathematics, Academy of Mathematics and Systems Science, Chinese Academy of Sciences, Beijing 100190, P.R.~China; School of Mathematical Sciences, University of Chinese Academy of Sciences,Beijing 100049, P.R.~China}
\email{wangyi@amss.ac.cn}

\author[M.-X. Xiao]{Meixia Xiao}
\address[M.-X. Xiao]{Research Center of Nonlinear Science, School of Mathematics  and Statistics, Wuhan Textile University, 430200, P.R.~China }
\email{xiao\_meixia@163.com}

\author[H. Xiong]{Hang Xiong}
\address[H. Xiong]{{School of Mathematics and Computational Science, Xiangtan University \& National Center for Applied Mathematics in Hunan, Xiangtan 411105, P.R.~China; Institute of Applied Mathematics, AMSS, CAS, Beijing 100190, P.R.~China}}
\email{hbear0810@amss.ac.cn}

\date{\today}

\begin{abstract}
We investigate nonlinear Landau damping for the two-species screened Vlasov-Poisson system with large initial distributions on the phase space $\mathbb{R}^d \times \mathbb{R}^d$ (where $d \geq 3$). Under a structural quasi-neutrality condition, we establish the existence and uniqueness of global strong solutions to the two-species system with arbitrarily large initial distributions. Furthermore, we prove the time-asymptotic stability of Penrose-stable equilibria and establish the optimal decay rate $t^{-d}$ for the net charge density, thereby verifying the nonlinear Landau damping effect for the two-species screened Vlasov-Poisson system in the whole space. To the best of our knowledge, this represents the first result on Landau damping for the two-species Vlasov-Poisson system with large initial distributions that are significantly far from equilibrium.
\end{abstract}

%\thanks{}
\maketitle
\tableofcontents

\thispagestyle{empty}

%Keywords: Vlasov equation, global existence, large amplitude solution
%
%\
%
%AMS: 35Q35, 35B65, 76N10
%\end{abstract}

%\tableofcontents

%%%%%%%%%%%%%%%%%%%%%%%%%%%%%%%%%%%%%%%%%%%%%%%%%%%%%%%%%%%%%%%%%%%%%%%%%%%%%%%%%%%%%%%%%%%%%%%%%%
\section{Introduction}
\subsection{Landau Damping}
The classical Vlasov-Poisson system, first introduced by Vlasov \cite{Vlasov}  in 1938, is a fundamental kinetic model for collision-less plasmas, describing the transport of charged particles coupled with a self-consistent electric field that encodes the long-range Coulomb interactions. The well-posedness of the classical Vlasov-Poisson system has been well-established, even for large initial data,
due to its collision-less dynamics (see, e.g., \cite{Glassey000,Lions001,Pfaffelmoser001,Rein000} and the references therein). Physically, a fundamental phenomenon in collision-less plasma is the Landau damping, first proposed by Landau \cite{Landau000} in 1946, which describes the transfer of the energy from macroscopic waves to microscopic particles motion through the particle-wave interactions. Mathematically, this mechanism leads to the decay of the macroscopic density or the electric field in time despite the absence of collisions or any explicit dissipative terms.

In \cite{Landau000} Landau formally analyzed the damping effect for the linearized Vlasov-Poisson system near a spatially homogeneous Maxwellian but without the rigorous mathematical proof. In 2011, Mouhot and Villani \cite{Mouhot000} first proved the nonlinear Landau damping for the Vlasov-Poisson system on the spatial torus, for small analytic or near-analytic Gevrey perturbations of  general spatially homogeneous equilibria satisfying the Penrose stability condition  (see also the further results \cite{Bedrossian001,Grenier000}). However, the corresponding result in the whole space is much more complicated and remains largely open up to now. The main difficulty comes from that the Penrose stability condition does not hold at the low-frequency part of the whole space for the general spatially homogeneous equilibria, in particular the global Maxwellian type Gaussian distribution, while the Penrose condition is always satisfied in the frequency space of the spatial torus for general equilibria.
On the other hand, Bedrossian, Masmoudi and Mouhot \cite{Bedrossian006} proved the linear stability of global Maxwellian equilibrium in the whole space and showed that the macroscopic density or the electric field behaves like a Klein-Gordon-type oscillation at the low-frequency part and regularly at the high-frequency part. Furthermore, Nguyen \cite{Nguyen2026} proved the linear stability of general radially spatially homogeneous equilibria with the form $\mu(\frac{1}{2}|v|^2)$ having the connected support and including the non-monotone profiles, and found that the decay rate in the Landau damping at the low-frequency part is highly sensitive to the decay rate of the equilibria tails at large velocities: faster the tails decay, weaker the damping becomes.

Very recently, for Poisson-type equilibria with polynomial velocity decay, Ionescu, Pausader, Wang, and Widmayer \cite{Ionescu000} rigorously established nonlinear Landau damping to the Vlasov-Poisson system in the whole space, and then Nguyen, Wei, and Zhang \cite{Nguyen2024} provided an alternative simplified proof. However, for general spatially homogeneous equilibria including the  Maxwellian with the exponential velocity decay, the nonlinear Landau damping for the Vlasov-Poisson system in the whole space is still a very challenging open problem;  and one  can refer to the related literature {\cite{Bedrossian006,Glassey002,Glassey001,HanKwan000,Ionescu001}} and the references therein.

To address the low-frequency obstruction of the Coulomb interactions in the whole space, Bedross-ian, Masmoudi, and Mouhot \cite{Bedrossian000} proposed a screened Vlasov-Poisson system (also named as Vlasov-Yukawa system), with the standard Poisson kernel replaced by the screened Coulomb interactions
\begin{equation}\label{coulomb}
E(t,x)=\nabla_x (I-\Delta_x)^{-1}\rho(t,x).
\end{equation}
The above screened Coulomb interactions recover the strict ellipticity at the low frequency part, which circumvent the difficulty that the Penrose stability condition does not hold in the whole space for the general spatially homogeneous equilibria. And the nonlinear Landau damping is rigorously proved to the screened Vlasov-Poisson system in the whole space $\mathbb{R}^3$ for general equilibria under the finite Sobolev regularity assumptions in \cite{Bedrossian000}. Later this result %in %\cite{Bedrossian000}
was improved to obtaining the optimal decay rate  up to a logarithmic correction via Lagrangian and point-wise Green function method in \cite{HanKwanD2021} (see also \cite{Nguyen2020}), and  then the logarithmic correction was finally removed in \cite{Huang002, Huang2024}.

It should be emphasized that all the previous existing results on the nonlinear Landau damping for Vlasov-Poisson system, including both classical and screened, need the smallness
conditions on the initial perturbations around the equilibria. To the best of our knowledge, there is no result on the nonlinear Landau damping for Vlasov-Poisson system with large initial distributions that are significantly far from equilibrium.
The goal of the present paper is to verify the nonlinear Landau damping for screened two-species Vlasov-Poisson system in the whole space $\mathbb{R}^d (d\geq 3)$ for large initial distributions far from the equilibria under a structural quasi-neutrality condition.

The two-species screened
Vlasov-Poisson system (cf. \cite{Bedrossian000}) is read as
\begin{equation}\label{VKG0}
\left\{
\begin{array}{l}
\partial_{t}F^{+}+v\cdot\nabla_{x}F^++E(t,x)\cdot\nabla_{v}F^+=0,\\[1mm]
\partial_{t}F^{-}+v\cdot\nabla_{x}F^-- E(t,x)\cdot\nabla_{v}F^-=0,\\[1mm]
E(t,x)=\nabla_x(I-\Delta_x)^{-1}\rho_{F}(t,x),
\end{array}
\right.
\end{equation}
where $F^\pm:=F^\pm(t,x,v)$ denote the density distributions of ions (``$+$") and electrons (``$-$") respectively, at the time $t>0$ and in the phase-space $(x,v)\in\mathbb{R}^d\times\mathbb{R}^d$,
and the self-consistent field $E(t,x)$  is generated through the screened Coulomb  interactions \eqref{coulomb} by the net charge density $\rho_{F}:=\rho_{F}^{+}-\rho_{F}^{-}$ with
$\di \rho_{F}^\pm(t,x):=\int_{\mathbb{R}^d} F^\pm(t,x,v)\,dv.$
Note that the two-species Vlasov-Poisson system \eqref{VKG0} is a fundamental physical model due to the quasi-neutrality of the plasma in nature. For the physical background of the two-species system, one can refer to the related literature \cite{Baumann2023, Duan2024, Glassey09, Krall73}.

\subsection{Main Result}
Throughout the paper, the spatially homogeneous equilibrium $\mu=\mu(v)$ is always normalized as the unit density
\[
\int_{\mathbb{R}^d}\mu(v)\, dv=1,
\]
and assumed to satisfy the following regularity and structural conditions.

\medskip
\begin{assumption}[Regularity condition]\label{ass:regularity}
\begin{equation*}
\|\langle v\rangle^N\nabla_v\mu\|_{W^{3, \infty}(\mathbb{R}^d)}+\|\langle v\rangle^{d+5}\nabla_v\mu\|_{W^{2d+7, 1}(\mathbb{R}^d)}<\infty, \ \ \  \forall N>d.
\end{equation*}
\end{assumption}

\begin{assumption}[Penrose stability condition]\label{ass:penrose}
There exists $\kappa>0$ such that
\[
\inf_{\substack{\xi\in\mathbb{R}^d, \\ \lambda\in \mathbb{C}\ {\rm with}\ {\rm Im} \lambda<0}}\Big|1
+2\int_0^{\infty}e^{-i\lambda s}\frac{i\xi}{1+|\xi|^2}\cdot
\widehat{\nabla_v\mu}(s\xi)
\,ds\Big|\geq \kappa,
\]
where $\widehat{\nabla_v\mu}$
denotes the Fourier transform of $\nabla_v\mu$ and ${\rm Im} \lambda$ is the imaginary part of $\lambda\in \mathbb{C}$.
\end{assumption}

\vspace{0.5mm}

Denote the perturbations around the equilibrium $\mu(v)$ as
\[
f^\pm(t,x,v):=F^\pm(t,x,v)-\mu(v).
\]
Then the perturbations $f^\pm:=f^\pm(t,x,v)$ satisfy the system
\begin{equation}\label{VKG}
\left\{
\begin{array}{l}
\partial_{t}f^{+}+v\cdot\nabla_{x}f^++E(t,x)\cdot\nabla_{v}f^+
=- E(t,x)\cdot\nabla_{v}\mu,\\[2mm]
\partial_{t}f^{-}+v\cdot\nabla_{x}f^-- E(t,x)\cdot\nabla_{v}f^-
= E(t,x)\cdot\nabla_{v}\mu,\\[2mm]
E(t,x)=\nabla_x(I-\Delta_x)^{-1}\rho(t, x),
\end{array}
\right.
\end{equation}
with $\rho(t,x):=\rho^+(t,x)-\rho^-(t,x)=\rho_F(t,x)$ and
$
\di \rho^\pm(t,x):=\di \int_{\mathbb{R}^d} f^\pm(t,x,v)\, dv,
$ and the initial data
\[
f^{\pm}(0,x,v)=f_{\mathrm{in}}^{\pm}(x,v).
\]

Let $a\in(0,1)$ and $1\le p\le\infty$. We fix a homogeneous
Littlewood--Paley decomposition $\{\dot\Delta_j\}_{j\in\mathbb Z}$ on
$\mathbb R^d$. More precisely, if $\varphi$ is a smooth cutoff supported
in an annulus and chosen so that
$\sum_{j\in\mathbb Z}\varphi(2^{-j}\xi)=1, \xi\in\mathbb R^d\setminus\{0\}$,
then
\[
\mathcal F(\dot\Delta_j\phi)(\xi)
=
\varphi(2^{-j}\xi)\widehat\phi(\xi),
\qquad j\in\mathbb Z.
\]
We refer to \cite[Chapter~2]{BCD} for the standard construction. We define
the homogeneous Besov and Triebel--Lizorkin seminorms by
\begin{equation}\label{semi-norm}
\|\phi\|_{\dot B^a_{p,\infty}}
:=
\sup_{j\in\mathbb Z}
2^{ja}\|\dot\Delta_j\phi\|_{L^p(\mathbb R^d)},
\qquad
\|\phi\|_{\dot F^a_{p,\infty}}
:=
\Bigl\|
\sup_{j\in\mathbb Z}
2^{ja}|\dot\Delta_j\phi|
\Bigr\|_{L^p(\mathbb R^d)}.
\end{equation}
In the endpoint case $p=\infty$, these two seminorms coincide:
\begin{equation}\label{FTL0}
\|\phi\|_{\dot F^a_{\infty,\infty}}
=
\|\phi\|_{\dot B^a_{\infty,\infty}}.
\end{equation}
We shall also use the standard difference-quotient characterization of
$\dot B^a_{p,\infty}$:
\begin{equation}\label{Besov}
\|\phi\|_{\dot B^a_{p,\infty}}
\sim
\sup_{\alpha}
\frac{\|\delta_\alpha\phi\|_{L^p(\mathbb R^d)}}{|\alpha|^a},
\qquad
\delta_\alpha\phi(x):=\phi(x)-\phi(x-\alpha).
\end{equation}
%Throughout the paper, $\dot F^a_{p,\infty}$ is  in the dyadic
%sense \eqref{semi-norm}. The Besov seminorm  is
%defined by \eqref{semi-norm} and used mainly through the equivalent
%difference-quotient characterization \eqref{Besov}.

Throughout this paper, the space $\dot F^a_{p,\infty}$ is understood in the dyadic sense, with its associated seminorms defined in \eqref{semi-norm}. The Besov seminorm $\dot B^a_{p,\infty}$​, defined in \eqref{semi-norm}, can also be employed through its equivalent difference-quotient characterization  \eqref{Besov}.

For a function $h=h(x, v)$, we define
\[
D^ah:=D_1^ah+D_2^ah,
\]
where
\begin{equation}\label{Dhdef}
D_1^ah(x,v):=\sup\limits_{z}\frac{|h(x,v)-h(x-z,v)|}{|z|^a},  \  \
D_2^ah(x,v):=\sup\limits_{z}\frac{|h(x,v)-h(x,v-z)|}{|z|^a}.
\end{equation}
For $m\in \mathbb{R}$ and $j\in \mathbb{N}$, we define
\[
\|g\|_{W^{j,\infty}_m(\mathbb{R}^d\times \mathbb{R}^d)}
:=
\sum_{\ell=0}^{j}
\|\langle v\rangle^m \nabla_{x,v}^{\ell} g\|_{L^\infty_{x,v}},
\]
where $\nabla_{x,v}^{\ell} g$ denotes the collection of all derivatives of order $\ell$ with respect to $(x,v)$.
For the initial data, we denote
\begin{align*}
||| f_{\mathrm{in}}^\pm|||_1:=&\|f_{\mathrm{in}}^\pm\|_{W^{1,1}(\mathbb{R}^d\times \mathbb{R}^d)}+\|f_{\mathrm{in}}^\pm\|_{W^{2,\infty}_k(\mathbb{R}^d\times \mathbb{R}^d)}\\
&\quad +
\|\langle v\rangle^k\nabla_{x,v}f^\pm_{\mathrm{in}}\|_{ L_x^1L_v^\infty}+\|\langle v\rangle^k\nabla_{x,v}D^af^\pm_{\mathrm{in}}\|_{ L_x^1L_v^\infty}+\|\langle v\rangle^k\nabla_{x,v}^2f^\pm_{\mathrm{in}}\|_{ L_x^1L_v^\infty},
\end{align*}
and  $f_{\mathrm{in}}:=f^+_{\mathrm{in}}-f^-_{\mathrm{in}}$ and then denote
\[
|||f_{\mathrm{in}}|||_2:=\|f_{\mathrm{in}}\|_{W^{1,1}(\mathbb{R}^d\times \mathbb{R}^d)}+\|f_{\mathrm{in}}\|_{ L_v^1L_x^\infty\cap L_x^1L_v^\infty}+\|\nabla_{x,v}f_{\mathrm{in}}\|_{ L_v^1L_x^\infty\cap L_x^1L_v^\infty}+\|D^af_{\mathrm{in}}\|_{L_x^1L_v^1\cap L_x^1L_v^\infty}.
\]

%\vspace{1mm}

Throughout the paper, we denote $C$ as a generic positive constant whose value may change from line to line.
We use $C(\cdot)$ to denote the constant $C$ depending on the parameters in $(\cdot)$.
Denote
$\langle v\rangle:=\sqrt{1+|v|^2}$ for $v\in\mathbb{R}^d$.
We write $A \lesssim B$ if $A \leq C B$ for a generic positive constant $C$.
The notation $\nabla_{x,v}$ stands for the derivatives with respect to $(x,v)$.
Finally, $\ast_{(t,x)}$ and $\ast_x$ denote the convolution in $(t,x)$ and in $x$, respectively. Now we can state the main result.

\begin{theorem}\label{Theorem1}
Let $d \geq 3$, and suppose both Assumption \ref{ass:regularity} and \ref{ass:penrose} hold. For any fixed $M_0 \geq 1$, $k>d$, and  $a \in \big(\frac{\gamma}{d-\gamma-2}, 1 \big)$ with $\gamma>0$
being suitably  small constant,
if the initial perturbations satisfy
\begin{equation}\label{large-data}
\max\{|||f_{\mathrm{in}}^+|||_1, \ |||f_{\mathrm{in}}^-|||_1\}
\leq M_0,
\end{equation}
then there exists a small constant $\varepsilon_0>0$, depending on $M_0$, $a$, $d$, $k$, $\gamma$, $\mu$, such that if the following quasi-neutrality condition holds,
\begin{equation}\label{qn}
|||f_{\mathrm{in}}^+-f^-_{\mathrm{in}}|||_2\leq \varepsilon_0,
\end{equation}
then the perturbation system \eqref{VKG} admits a unique global strong solution
\[
f^\pm \in C\big([0,\infty); W^{1,1}(\mathbb{R}^d\times\mathbb{R}^d)
\cap W^{1,\infty}_{k-1}(\mathbb{R}^d\times\mathbb{R}^d)\big)
\cap L^\infty\big([0,\infty); W^{2,\infty}_k(\mathbb{R}^d\times\mathbb{R}^d)\big),
\]
satisfying the time-decay rates
\[
\|\rho(t,\cdot)\|_{L^1_x}
+\langle t\rangle^d \|\rho(t,\cdot)\|_{L^\infty_x}
+\langle t\rangle^a \|\rho(t,\cdot)\|_{\dot B^a_{1,\infty}}
+\langle t\rangle^{d+a} \|\rho(t,\cdot)\|_{\dot B^a_{\infty,\infty}}
\lesssim \varepsilon_0,
\]
and
\[
\langle t\rangle \|\nabla_x \rho(t,\cdot)\|_{L^1_x}
+\langle t\rangle^{d+1} \|\nabla_x \rho(t,\cdot)\|_{L^\infty_x}
\lesssim \varepsilon_0\ln(1+t).
\]
\end{theorem}

\vspace{0.5mm}

\begin{corollary}\label{tuilun}
Under the assumptions of Theorem~\ref{Theorem1}, the limits
\[
Y_\infty^\pm(x,v):=\lim_{t\to\infty}Y_{0,t}^\pm(x+tv,v),
\qquad
W_\infty^\pm(x,v):=\lim_{t\to\infty}W_{0,t}^\pm(x+tv,v),
\]
exist in \(L^\infty(\mathbb{R}^d_x\times\mathbb{R}^d_v)\) and satisfy
\[
\|Y^\pm_\infty\|_{L^\infty_{x,v}}+\|W^\pm_\infty\|_{L^\infty_{x,v}}
\lesssim \varepsilon_0,
\]
with $Y_{0,t}^\pm$ and $W_{0,t}^\pm $ being defined in \eqref{character001}.
Denote
\[
f^\pm_{in,\infty}(x,v)
:=f^\pm_{\mathrm{in}}\bigl(x+Y^\pm_\infty(x,v),\,v+W^\pm_\infty(x,v)\bigr)
\mp\bigl(\mu\bigl(v+W^\pm_\infty(x,v)\bigr)-\mu(v)\bigr).
\]
Then for all $t \ge 0$,
\begin{equation}\label{scatter}
\|f^\pm(t,x+tv,v)-f^\pm_{in,\infty}(x,v)\|_{L^\infty_{x,v}}
\lesssim
\varepsilon_0\frac{\ln(1+t)}{\langle t\rangle^{d-1}}.
\end{equation}
\end{corollary}

\begin{remark}\label{Remark1.2}
Theorem~\ref{Theorem1} implies that the two-species screened Vlasov-Poisson system \eqref{VKG} exhibits the following time-decay rates
$$
\|\rho(t,\cdot)\|_{L^\infty_x}\lesssim\langle t\rangle^{-d},  \  \
\|E(t, \cdot)\|_{L^\infty_x}
\lesssim \langle t\rangle^{-(d+1)}\ln(1+t),
$$
which verifies the nonlinear Landau damping for arbitrarily large initial distributions with the quasi-neutrality condition \eqref{qn}.
In other words, despite the absence of the collision effects,  the phase mixing yields the time decay of the charge density $\rho$ and hence the self-consistent electric field $E$, and the solutions $f^\pm$ to the screened Vlasov-Poisson system \eqref{VKG} scatter asymptotically to the free transport states as in \eqref{scatter}.
\end{remark}

\begin{remark}\label{Remark1.3}
The initial distributions in \eqref{large-data} can be arbitrarily large with the quasi-neutrality condition \eqref{qn}, which can be described in Figure~\ref{fig:initial-density} as follows. Nevertheless, Theorem~\ref{Theorem1} implies the optimal dispersive estimate for each species density
\begin{equation}\label{dispersive-e}
\|\rho^\pm(t,\cdot)\|_{L^\infty_x}=\|\rho_{F}^{\pm}(t,\cdot)-1\|_{L^\infty_x}\lesssim \langle t\rangle^{-d},
\quad \forall t\ge 0.
\end{equation}
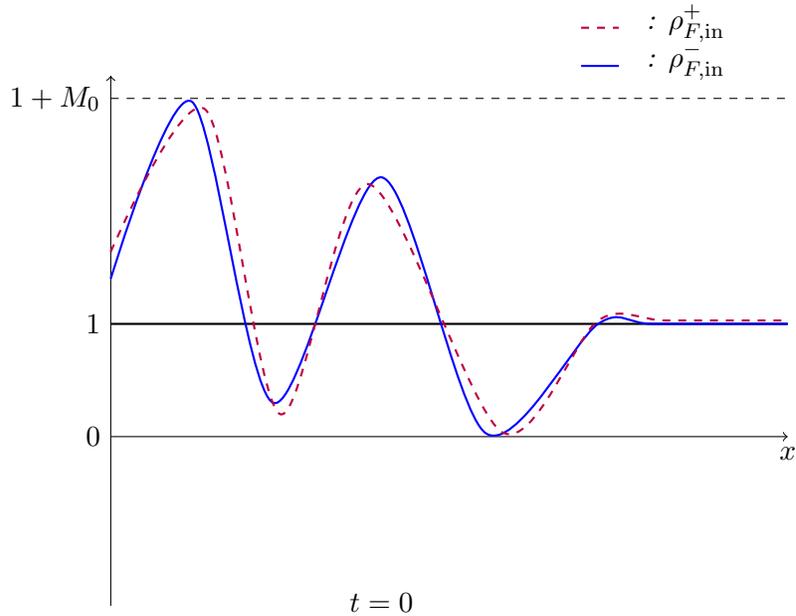
\begin{figure}[ht]
   \begin{tikzpicture}[xscale=1.8,yscale=1.48]
					\draw [->] (0,0)--(5,0);
					\draw [->] (0,0)--(0,3.2);
                    \draw [-] (0,0)--(0,-1.5);
					%\draw [dashed] (0,-1)--(5,-1);
					\draw [dashed] (0,3)--(5,3);
                    \draw [thick] (0,1)--(5,1);
					\node [left] at (0,1) {$1$};
                    %\node [left] at (0,-1) {$-1$};
					\node [left] at (0,0) {$0$};
					\node [below] at (5,0) {$x$};
					\node [left] at (0,3) {$1+M_0$};
					
\draw[blue, thick] plot[smooth] coordinates {(0,1.4)(0.6,2.97)(1.2,0.3)(2,2.3) (2.77,0.03) (3.6,1)(4,1)(5,1)};

\draw[purple, thick,dashed] plot[smooth] coordinates {(0,1.64) (0.7,2.9) (1.25,0.2)(1.9,2.24) (2.89,0.04)(3.6,1.03)(4.1,1.03)(5,1.03)};
					\node [below] at (2,-1.3) {$t=0$};
 \node at (4,3.5) {
                \begin{tabular}{cl}
                    \tikz{\draw[purple, thick,dashed] (3,4) -- (3.5,4);} & :  $\rho_{F, \mathrm{in}}^+$ \\
                    \tikz{\draw[blue, thick] (3,4) -- (3.5,4);} & :   $\rho_{F, \mathrm{in}}^-$ \\
                \end{tabular}
            };

				\end{tikzpicture}
			\caption{$\rho_{F, \mathrm{in}}^{\pm}:=\di \int_{\mathbb{R}^d} F^\pm(0,x,v)\, dv$  can be arbitrarily large with the quasi-neutrality condition}
\label{fig:initial-density}
\end{figure}
\end{remark}
\subsection{Key Challenges and Proof Strategy}
When the large initial distributions  \eqref{large-data}  are allowed, the key mathematical challenge for verifying nonlinear Landau damping to \eqref{VKG0} or \eqref{VKG} stems from the nonlinear terms, affecting both global well-posedness and the long-time stability.

As emphasized by Mouhot and Villani~\cite{Mouhot000}, nonlinear Landau damping can not be expected for arbitrarily large perturbations, which is also shown in the numerical study of Zhou, Guo, and Shu~\cite{ZhouGuoShu2001}. In this work, we introduce the quasi-neutrality condition \eqref{qn} as a pivotal stabilization mechanism to recover the macroscopic phase mixing. This condition constrains the net
charge density to keep the smallness, thereby ensuring the screened electric field $E(t,x)$ is suitably
weak to suppress unsteady perturbations and stabilize the characteristic flow. In this way, we have
\[
\text{quasi-neutrality} \Longrightarrow\ \text{weak screened field} \Longrightarrow\ \text{stable characteristics}
 \Longrightarrow\ \text{macroscopic phase mixing}.
\]
However, quasi-neutrality alone is insufficient to complete the proof, more accurate bounds on  $\nabla_v f^\pm$ must additionally be derived to close the a priori assumptions,
which is done by the bootstrap argument.

Based on the local-in-time existence of strong solution in Proposition~\ref{Local solutions}, we formulate a priori assumptions in \eqref{ziyouliu02}, that is,
\begin{equation*}
\left\{
\begin{aligned}
\|\rho(t)\|_{L^\infty_x}
&\le \varepsilon \langle t\rangle^{-d+\gamma},\\
\|\rho(t)\|_{\dot B^a_{\infty,\infty}}
&\le \varepsilon \langle t\rangle^{-(d+a)+\gamma},\\
\|\rho(t)\|_{\dot B^a_{1,\infty}}
&\le \varepsilon \langle t\rangle^{-a+\gamma},\\
\|\nabla_x \rho(t)\|_{L^\infty_x}
&\le \varepsilon \langle t\rangle^{-(d+1)+\gamma},\\
\|\nabla_x \rho(t)\|_{L^1_x}
&\le \varepsilon \langle t\rangle^{-1+\gamma}.
\end{aligned}
\right.
\end{equation*}
Then we can prove the global existence together with the optimal decay estimates for the net charge density.

To close the bootstrap assumptions in \eqref{ziyouliu02}, we combine the uniform estimates in finite time intervals with the sharp decay estimates for the large time. Precisely,
in finite time intervals, quasi-neutrality condition implies the smallness of the net charge, and together with Proposition~\ref{xiaoshijian008} yields
\[
\|\rho(t)\|_{L^1_x}+\|\rho(t)\|_{L^\infty_x}+\|\rho(t)\|_{\dot{B}^a_{1,\infty}}
+\|\rho(t)\|_{\dot{B}^a_{\infty,\infty}}
+\|\nabla_x\rho(t)\|_{L^1_x}+\|\nabla_x\rho(t)\|_{L^\infty_x}
\le M_0\,C(t)\,\varepsilon_0.
\]
For the large time, motivated by the methods developed in
\cite{Bedrossian000,HanKwanD2021,Landau000},  we rewrite the charge density along the nonlinear characteristics and represent it as a
convolution against a time-space kernel $\mathcal G$ (defined in \eqref{fangcheng000}),
\[
\rho=-\mathcal{G}\ast_{(t,x)}\left[(I^+-I^-)-(R^+ +R^-)\right]
+\left(I^+-I^-\right)-\left(R^+ +R^-\right),
\]
with $I^\pm$ and $R^\pm$ defined in \eqref{midufenjie001}. The problem is thus reduced to sharp bounds on $\mathcal G$ and on the associated convolution operator
(Lemmas~\ref{kernel}--\ref{juanji001}). To couple fractional regularity with dispersive decay, we work in Besov norms
and obtain, in particular, the estimate
\[
\langle t\rangle^d\|(\mathcal{G}\ast_{(t,x)}\hbar)(t)\|_{L^\infty_x}
\lesssim
\sup_{0\le s\le t}\left(\|\hbar(s)\|_{L^1_x}
+\langle s\rangle^{d+a}\|\hbar(s)\|_{\dot{B}_{\infty,\infty}^a}\right),
\]
which is proved in Lemma~\ref{juanji001}. This estimate constitutes the main input for the optimal decay of $\rho$.

By \eqref{eq:dq-1-local}, we can use the Triebel--Lizorkin-type seminorm for the electric field $\|E(t)\|_{\dot F^{a}_{1,\infty}}$ defined in \eqref{semi-norm}  to control the estimates of $R^{\pm}$. By \eqref{F03},
 $\|E(t)\|_{\dot F^{a}_{1,\infty}}$ can be interpolated by $\|E(t)\|_{W^{1,1}(\mathbb{R}^d)}$, which is already well-estimated in \eqref{ziyouliu01}.

% , the main obstruction comes from the low frequencies part,
%which is overcome by  and \eqref{eq:dq-1-local} in Lemmas~\ref{field02} and~ \ref{lem:difference-F}.

%naturally arise by \ref  Tof ,

% To overcome this low-frequency obstruction, we use \eqref{eq:dq-1-local}

%For this reason, \eqref{F03} assumes that $E\in W^{1,1}(\mathbb{R}^d)$, which removes the low-frequency obstruction and gives
%$$
%\|E\|_{\dot F^{a}_{1,\infty}}
%\lesssim
%\|E\|_{L^1_x}^{\,1-a}\|\nabla_x E\|_{L^1_x}^{\,a}.
%$$
%Likewise, \eqref{eq:dq-1-local} is restricted to $|\alpha|\le1$, and yields
%\[
%\Big\|\sup_{|\alpha|\le1}\frac{|\delta_\alpha \psi|}{|\alpha|^a}\Big\|_{L^1(\mathbb{R}^d)}
%\lesssim \|\psi\|_{\dot F^{a}_{1,\infty}(\mathbb{R}^d)},
%\]
%thereby discarding the large-scale difference-quotient contribution associated with the same low-frequency phenomenon.
Then by Propositions~\ref{Besovestimates} and~\ref{daoshu009}, we have
\begin{equation*}
\left\{
\begin{aligned}
\|\rho(t)\|_{L^\infty_x}
&\le C(\varepsilon_0+\varepsilon) \langle t\rangle^{-d},\\
\|\rho(t)\|_{\dot B^a_{\infty,\infty}}
&\le C(\varepsilon_0+\varepsilon) \langle t\rangle^{-(d+a)},\\
\|\rho(t)\|_{\dot B^a_{1,\infty}}
&\le C(\varepsilon_0+\varepsilon) \langle t\rangle^{-a},\\
\|\nabla_x\rho(t)\|_{L^\infty_x}
&\le C(\varepsilon_0+\varepsilon) \langle t\rangle^{-(d+1)}\ln(1+t),\\
\|\nabla_x\rho(t)\|_{L^1_x}
&\le C(\varepsilon_0+\varepsilon) \langle t\rangle^{-1}\ln(1+t),
\end{aligned}
\right.
\end{equation*}
which restore the optimal decay estimates for the net charge density and its spatial derivatives.

%It should be emphasized that our treatment for the initial data and the organization of the estimates differ substantially from those in earlier works under the quasi-neutrality assumption,
%Since the one-sided densities may be large, one cannot close separate estimates for the two species. We therefore organize the analysis around the net charge
% $\rho=\rho^+-\rho^-$, and rewrite the leading terms in difference form, setting $I:=I^+-I^-$. This allows the quasi-neutral cancellation to be exploited directly in the control of the electric field and the characteristic flow; see Section~\ref{diwujie} for further details.

Finally, suitable combinations of the estimates  in the finite-time and large-time imply
\begin{equation*}
\left\{
\begin{aligned}
\|\rho(t)\|_{L^\infty_x}
&\le \tfrac12 \varepsilon \langle t\rangle^{-d+\gamma},\\
\|\rho(t)\|_{\dot B^a_{\infty,\infty}}
&\le \tfrac12 \varepsilon \langle t\rangle^{-(d+a)+\gamma},\\
\|\rho(t)\|_{\dot B^a_{1,\infty}}
&\le \tfrac12 \varepsilon \langle t\rangle^{-a+\gamma},\\
\|\nabla_x \rho(t)\|_{L^\infty_x}
&\le \tfrac12 \varepsilon \langle t\rangle^{-(d+1)+\gamma},\\
\|\nabla_x \rho(t)\|_{L^1_x}
&\le \tfrac12 \varepsilon \langle t\rangle^{-1+\gamma},
\end{aligned}
\right.
\end{equation*}
which can close the bootstrap assumptions and obtain our main result in Theorem \ref{Theorem1}.

%The details can be seen in Section .

\subsection{Preliminary Lemmas}\label{dierjie}
In this subsection, we list some properties of the operator
$T:=\nabla_x(I-\Delta_x)^{-1}$
in Sobolev, Besov and Triebel--Lizorkin-type spaces, which will be frequently used in the sequel. %In particular, Lemmas~\ref{field02} and~\ref{lem:difference-F} provide the basic link between the dyadic \(\dot F^a_{p,\infty}\) bounds and the difference-quotient control needed in the Besov estimates of the subsequent sections.

\begin{lemma}\label{field01}
Let $p\in\{1,\infty\}$ and assume that
$\psi\in W^{1,p}(\mathbb{R}^d)$.
Then there exists a constant $C>0$ such that
$$
\|T\psi\|_{L^p(\mathbb{R}^d)}
+\|\nabla_x T\psi\|_{L^p(\mathbb{R}^d)}
\le
C\|\nabla_x\psi\|_{L^p(\mathbb{R}^d)}.
$$
Moreover, if $\psi\in W^{2,\infty}(\mathbb{R}^d)$, then
$$
\|\nabla_x^2 T\psi\|_{L^\infty(\mathbb{R}^d)}
\le
C\|\nabla_x^2\psi\|_{L^\infty(\mathbb{R}^d)}.
$$
\end{lemma}

\begin{lemma}\label{field02}
Let $a\in(0,1)$ and $p\in\{1,\infty\}$.

\medskip
\noindent (i)
Assume that $\psi\in \dot B^{a}_{p,\infty}(\mathbb{R}^d)$. Then there exists $C>0$ such that
\begin{equation}\label{field02:Besov1}
\|T\psi\|_{\dot B^{a}_{p,\infty}}
\le C\|\psi\|_{\dot B^{a}_{p,\infty}}.
\end{equation}

\medskip
\noindent (ii)
If $\psi\in \dot B^{a}_{\infty,\infty}(\mathbb{R}^d)$, then
\begin{equation*}
\|T\psi\|_{\dot F^{a}_{\infty,\infty}}
\le C\|\psi\|_{\dot B^{a}_{\infty,\infty}}.
\end{equation*}

\medskip
\noindent (iii)
If $T\psi\in W^{1,1}(\mathbb{R}^d)$, then
\begin{equation}\label{F03}
\|T\psi\|_{\dot F^{a}_{1,\infty}}
\lesssim
\|T\psi\|_{L^1}^{\,1-a}\|\nabla T\psi\|_{L^1}^{\,a}.
\end{equation}
\end{lemma}

\begin{lemma}\label{lem:difference-F}
Let $a\in(0,1)$.

\medskip
\noindent{(i)} If $\varphi^\ast\in \dot F^a_{\infty,\infty}(\mathbb{R}^d)$, then
\begin{equation}\label{eq:dq-infty}
\Big\|\sup_{\alpha}\frac{|\delta_\alpha \varphi^\ast|}{|\alpha|^a}\Big\|_{L^\infty}
\lesssim \|\varphi^\ast\|_{\dot F^a_{\infty,\infty}}.
\end{equation}

\medskip
\noindent{ (ii)} If $\varphi^\ast\in \dot F^{a}_{1,\infty}(\mathbb{R}^d)$, then
\begin{equation}\label{eq:dq-1-local}
\Big\|\sup_{|\alpha|\le1}\frac{|\delta_\alpha \varphi^\ast|}{|\alpha|^a}\Big\|_{L^1}
\lesssim \|\varphi^\ast\|_{\dot F^{a}_{1,\infty}}.
\end{equation}
\end{lemma}

We also use the following interpolation inequality, which controls the fractional  seminorm in terms of first-order Sobolev norms.

\begin{lemma}\label{chazhi00}
Let $a\in(0,1)$ and $p\in\{1,\infty\}$. If $g\in W^{1,p}(\mathbb{R}^d)$, then
\[
\|g\|_{\dot B^a_{p,\infty}}
\lesssim
\|g\|_{L^p}^{1-a}\|\nabla g\|_{L^p}^a,
\]
where the implicit constant depends only on $a$ and $d$.
\end{lemma}
\begin{proof}
We estimate the Besov difference quotient by separating large and small increments.

\noindent \textit{Case 1. $|\alpha| \geq R$.}
Using the triangle inequality, we have
\[
\Big\|\frac{g(\cdot)-g(\cdot-\alpha)}{|\alpha|^a}\Big\|_{L^p}
\le 2|\alpha|^{-a}\|g\|_{L^p}
\le 2R^{-a}\|g\|_{L^p}.
\]

\noindent \textit{Case 2. $|\alpha| < R$.}
By the mean value theorem,
\[
\Big\| \frac{g(\cdot)-g(\cdot-\alpha)}{|\alpha|^a} \Big\|_{L^p} \leq R^{1-a} \|\nabla g\|_{L^p}.
\]

Combining the two bounds, we obtain
\[
\|g\|_{\dot B^a_{p,\infty}}
\lesssim
R^{-a}\|g\|_{L^p}+R^{1-a}\|\nabla g\|_{L^p}.
\]
Optimizing in $R$ yields the desired estimate.
\end{proof}

%\subsection{Organization of the Paper}
The remainder of this paper is organized as follows.
%In Section~\ref{dierjie}, we collect several preliminary estimates for the electric field and its derivatives, together with the interpolation inequalities that will be used throughout the paper.
Section~\ref{disanjie} is devoted to the local well-posedness theory for the system \eqref{VKG} and the formulation of the bootstrap framework.
In Section~\ref{disijie}, we derive the estimates for the density and its spatial derivatives on any finite time intervals.
Section~\ref{diwujie} is concerned with the large-time analysis, based on Fourier analysis and pointwise estimates for the associated kernels.
Finally, in Section~\ref{diliujie}, we combine the preceding estimates to close the bootstrap argument and complete the proof of Theorem~\ref{Theorem1} and Corollary \ref{tuilun}.
In Appendix~\ref{diqijie} we present the proofs of several technical lemmas and propositions.

%%%%%%%%%%%%%%%%%%%%%%%%%%%%%%%%%%%%%%%%%%%%%%%%%%%%%
%\section{Preliminaries}

%%%%%%%%%%%%%%%%%%%%%%%%%%%%%%%%%%%%%%%%%%%%%%%%%%
\setcounter{equation}{0}
\section{Bootstrap Framework and Characteristic Estimates}\label{disanjie}
\subsection{Bootstrap Assumptions}
We begin with a local well-posedness result for the perturbation system \eqref{VKG}.
In addition to local existence and uniqueness in the strong-solution class, we show that
the smallness of the initial net perturbation is propagated on the local existence interval.
This provides the starting point for the bootstrap argument used in the proof of Theorem~\ref{Theorem1}.
For completeness, the proof is given in Appendix~\ref{diqijie}.

\begin{proposition}[Local well-posedness]\label{Local solutions}
Assume that the initial perturbations satisfy
\[
\|f_{\mathrm{in}}^\pm\|_{W^{1,1}(\mathbb{R}^d\times\mathbb{R}^d)}+\| f_{\mathrm{in}}^\pm\|_{W^{2,\infty}_k(\mathbb{R}^d\times\mathbb{R}^d)}\leq M_0, \quad k>d, \quad M_0\geq1.
\]
Then there exists a sufficiently small constant $c_*>0$, depending only on
$d$, $k$, and $\mu$, such that
\[
t_1:=\min\Bigl\{1,\frac{c_*}{C_0(M_0+1)}\Bigr\}>0,
\]
and the system \eqref{VKG} admits a unique strong solution
\[
f^\pm\in \mathcal{C}([0,t_1]; W^{1,1}(\mathbb{R}^d\times \mathbb{R}^d)
\cap W_{k-1}^{1,\infty}(\mathbb{R}^d\times \mathbb{R}^d))
\cap L^\infty([0,t_1]; W_k^{2,\infty}(\mathbb{R}^d\times \mathbb{R}^d)).
\]

Moreover, there exists a small constant $\varepsilon_0=\varepsilon_0(M_0,d,k,\mu)>0$ such that, if
\[
\|f_{\mathrm{in}}\|_{{L_v^1L_x^1}}+
\|f_{\mathrm{in}}\|_{{L_v^1L_x^\infty}}+
\|\nabla_xf_{\mathrm{in}}\|_{{L_v^1L_x^1}}
+\|\nabla_xf_{\mathrm{in}}\|_{{L_v^1L_x^\infty}}\leq \varepsilon_0,
\]
then
\[
\sup\limits_{0\leq t\leq t_1}\left(\|\rho(t)\|_{{L^1_x}}
+\|\rho(t)\|_{{L^\infty_x}}+\|\nabla_x\rho(t)\|_{{L^1_x}}
+\|\nabla_x\rho(t)\|_{{L^\infty_x}}\right)\leq 2\varepsilon_0.
\]
\end{proposition}

Proposition~\ref{Local solutions} yields a solution on the initial time interval $[0,t_1]$, which provides the starting point for the continuation argument.
To extend the solution beyond $t_1$, we introduce the following bootstrap assumptions.
Let $T>t_1$ be the maximal time such that the solution exists on $[0,T)$ and, for every $t\in[0,T)$, the charge density satisfies
\begin{equation}\label{ziyouliu02}
\left\{
\begin{aligned}
\|\rho(t)\|_{L^\infty_x}
&\le \varepsilon \langle t\rangle^{-d+\gamma},\\
\|\rho(t)\|_{\dot B^a_{\infty,\infty}}
&\le \varepsilon \langle t\rangle^{-(d+a)+\gamma},\\
\|\rho(t)\|_{\dot B^a_{1,\infty}}
&\le \varepsilon \langle t\rangle^{-a+\gamma},\\
\|\nabla_x \rho(t)\|_{L^\infty_x}
&\le \varepsilon \langle t\rangle^{-(d+1)+\gamma},\\
\|\nabla_x \rho(t)\|_{L^1_x}
&\le \varepsilon \langle t\rangle^{-1+\gamma}.
\end{aligned}
\right.
\end{equation}
Here $\varepsilon>0$ is taken sufficiently small. We choose $a \in (\tfrac{\gamma}{\,d-\gamma-2\,},1)$, and fix $\gamma>0$ sufficiently small so that all estimates in the subsequent analysis can be closed.

Combining the bootstrap bounds on $\nabla_x\rho$ with Lemma~\ref{field01}, and then applying Lemma~\ref{chazhi00} to $E(t)$, we obtain
\begin{equation}\label{ziyouliu01}
\left\{
\begin{aligned}
\|E(t)\|_{L^\infty_x}
+\|E(t)\|_{\dot B^a_{\infty,\infty}}
+\|\nabla_x E(t)\|_{L^\infty_x}
&\lesssim \varepsilon \langle t\rangle^{-(d+1)+\gamma},\\[1mm]
\|E(t)\|_{L^1_x}
+\|E(t)\|_{\dot B^a_{1,\infty}}
+\|\nabla_x E(t)\|_{L^1_x}
&\lesssim \varepsilon \langle t\rangle^{-1+\gamma}.
\end{aligned}
\right.
\end{equation}
In particular,
\[
E\in L^\infty([0,T);W^{1,\infty}(\mathbb{R}^d)).
\]
A corresponding $W^{2,\infty}$ bound for $E$, established later from Lemma~\ref{youjie010} and the density estimates, will also be used in the characteristic analysis.

\subsection{Estimates on the Characteristic Flow}\label{sec3.2}
For $0\le s\le t$, let $(\mathcal{X}^{\pm},\mathcal{V}^{\pm})$ be the characteristic flow associated with \eqref{VKG}, defined by
\begin{equation}\label{character1}
\left\{
  \begin{array}{ll}
    \dot{\mathcal{X}}^{\pm}(s;t,x,v)=\mathcal{V}^{\pm}(s;t,x,v), & \mathcal{X}^{\pm}(t;t,x,v)=x,  \\[1mm]
    \dot{\mathcal{V}}^{\pm}(s;t,x,v)=\pm E(s,\mathcal{X}^{\pm}(s;t,x,v)), & \mathcal{V}^{\pm}(t;t,x,v)=v.
  \end{array}
\right.
\end{equation}
For convenience, we abbreviate
$$
(\mathcal{X}^{\pm}(s),\mathcal{V}^{\pm}(s))
:=(\mathcal{X}^{\pm}(s;t,x,v), \mathcal{V}^{\pm}(s;t,x,v)).
$$
Integrating \eqref{character1}, we obtain for all $0\le s\le t$,
\begin{align}\label{z4}
\left\{
  \begin{array}{ll}
  \mathcal{X}^{\pm}(s)=x-(t-s)v+Y_{s,t}^{\pm}(x-tv,v), \\[1mm]
    \mathcal{V}^{\pm}(s)=v+W_{s,t}^{\pm}(x-tv,v),
  \end{array}
\right.
\end{align}
where $(Y_{s,t}^{\pm},W_{s,t}^{\pm})$ satisfy
\begin{align}\label{character001}
\left\{
  \begin{array}{ll}
  Y_{s,t}^\pm(x,v)=\pm\di \int_{s}^{t}(\tau-s)E\left(\tau,x+\tau v+Y_{\tau,t}^\pm(x,v)\right)\, d\tau,\\
    W_{s,t}^\pm(x,v)=\mp\di \int_{s}^{t}E\left(\tau,x+\tau v+Y_{\tau,t}^\pm(x,v)\right)\, d\tau.
  \end{array}
\right.
\end{align}

The rest of this subsection is devoted to quantitative bounds on $Y_{s,t}^{\pm}$ and $W_{s,t}^{\pm}$.
These estimates measure the deviation of the characteristic flow from free transport and play a central role in the bootstrap argument developed later.

\begin{lemma}\label{estimates001}
Assume that the field estimates \eqref{ziyouliu01} hold on $[0,T)$.
Then for all $0\le s\le t<T$,
the functions $Y^\pm_{s,t}$ and $W^\pm_{s,t}$ defined in \eqref{z4}
satisfy
\begin{align}\label{tezhengY}
&\sup_{0\leq s\leq t}\langle s\rangle^{d-\gamma-1}\Big(\|Y_{s,t}^\pm\|_{L^\infty_{x,v}}
+\|\nabla_x Y_{s,t}^\pm\|_{L^\infty_{x,v}}+\sup_{\alpha}\frac{\|\delta_{\alpha}^x	Y_{s,t}^\pm\|_{L^\infty_{x,v}}}{|\alpha|^a}\Big)\nonumber\\
&\quad\quad +\sup_{0\leq s\leq t}\Big(\langle s\rangle^{d-\gamma-2}\|\nabla_v Y_{s,t}^\pm\|_{L^\infty_{x,v}}+\langle s\rangle^{d-a-\gamma-1}\sup_{\alpha}\frac{\|\delta_{\alpha}^v	Y_{s,t}^\pm\|_{L^\infty_{x,v}}}{|\alpha|^a}\Big)
\lesssim \varepsilon,
\end{align}
and
\begin{align}\label{tezhengW}
&\sup_{0\leq s\leq t}\langle s\rangle^{d-\gamma}\Big(\|W_{s,t}^\pm\|_{L^\infty_{x,v}}
+\|\nabla_x W_{s,t}^\pm\|_{L^\infty_{x,v}}+\sup_{\alpha}\frac{\|\delta_{\alpha}^x	W_{s,t}^\pm\|_{L^\infty_{x,v}}}{|\alpha|^a}\Big)\nonumber\\
&\quad\quad +\sup_{0\leq s\leq t}\Big(\langle s\rangle^{d-\gamma-1}\|\nabla_v W_{s,t}^\pm\|_{L^\infty_{x,v}}+\langle s\rangle^{d-a-\gamma}\sup_{\alpha}\frac{\|\delta_{\alpha}^v	W_{s,t}^\pm\|_{L^\infty_{x,v}}}{|\alpha|^a}\Big)
\lesssim \varepsilon.
\end{align}
Here
\[
\delta_{\alpha}^xY_{s,t}^\pm(x,v)=Y_{s,t}^\pm(x,v)-Y_{s,t}^\pm(x-\alpha,v), \  \ \delta_{\alpha}^vY_{s,t}^\pm(x,v)=Y_{s,t}^\pm(x,v)-Y_{s,t}^\pm(x,v-\alpha),
\]
and similarly for $W_{s,t}^{\pm}$.

Moreover, for all $0\le s\le t$ and for $x,v\in \mathbb{R}^d$,
the maps
\[
x\mapsto \mathcal X^\pm(s;t,x,v),
\qquad
v\mapsto \mathcal V^\pm(s;t,x,v),
\]
are $\mathcal C^1$ diffeomorphisms.
\end{lemma}

\begin{proof}
We first estimate $Y_{s,t}^{\pm}$. From~\eqref{character001} and \eqref{ziyouliu01}, we obtain
\[
\|Y_{s,t}^\pm\|_{L^\infty_{x,v}}
\leq\int_s^t(\tau-s)\|E(\tau)\|_{L^\infty_x}\, d\tau\lesssim
\varepsilon \int_s^t \frac{\tau-s}{\langle \tau\rangle^{d+1-\gamma}}\,d\tau
\lesssim\frac{\varepsilon}{\langle s\rangle^{d-\gamma-1}}.
\]

Differentiating \eqref{character001} with respect to $x_i$, and using the chain rule, we obtain
\[
\|\partial_{x_i}Y_{s,t}^{\pm}\|_{L^\infty_{x,v}}
\leq \int_s^t(\tau-s)\|\nabla_xE(\tau)\|_{L^\infty_x}
\Big(1+\|\partial_{x_i}Y_{\tau,t}^{\pm}\|_{L^\infty_{x,v}}\Big)\, d\tau.
\]
It follows from~\eqref{ziyouliu01} that
\[
\|\partial_{x_i}Y_{s,t}^{\pm}\|_{L^\infty_{x,v}}
\lesssim  \varepsilon\int_s^t \frac{\tau-s}{\langle \tau\rangle^{d-\gamma+1}}\,d\tau
+\varepsilon\int_s^t\frac{\tau-s}{\langle \tau\rangle^{d-\gamma+1}}
\|\partial_{x_i}Y_{\tau,t}^{\pm}\|_{L^\infty_{x,v}}\, d\tau.
\]
Multiplying by $\langle s\rangle^{d-\gamma-1}$, and applying Gronwall's inequality, we have
\[
\sup\limits_{0\leq s\leq t}\langle s\rangle^{d-\gamma-1}
\|\nabla_xY_{s,t}^\pm\|_{L^\infty_{x,v}}\lesssim\varepsilon.
\]

Differentiating \eqref{character001} with respect to $v_i$ yields
\[
\|\partial_{v_i}Y_{s,t}^{\pm}\|_{L^\infty_{x,v}}
\leq \int_s^t(\tau-s)\|\nabla_xE(\tau)\|_{L^\infty_x}
\Big(\tau+\|\partial_{v_i}Y_{\tau,t}^{\pm}\|_{L^\infty_{x,v}}\Big)\, d\tau.
\]
Using~\eqref{ziyouliu01}, it holds that
\[
\|\partial_{v_i}Y_{s,t}^{\pm}\|_{L^\infty_{x,v}}
\lesssim \varepsilon\int_s^t \frac{(\tau-s)\tau}{\langle \tau\rangle^{d-\gamma+1}}\,d\tau
+\varepsilon\int_s^t\frac{\tau-s}{\langle \tau\rangle^{d-\gamma+1}}
\|\partial_{v_i}Y_{\tau,t}^{\pm}\|_{L^\infty_{x,v}}\, d\tau.
\]
By Gronwall's inequality, we have
\[
\sup\limits_{0\leq s\leq t}\langle s\rangle^{d-\gamma-2}
\|\nabla_vY_{s,t}^\pm\|_{L^\infty_{x,v}}\lesssim\varepsilon.
\]

Applying the mean-value inequality and the bound on $\nabla_x Y_{\tau,t}^\pm$, we find
\begin{align*}
\sup_{\alpha}\frac{\|\delta_{\alpha}^x	Y_{s,t}^\pm\|_{L^\infty_{x,v}}}{|\alpha|^a}
\leq&  \int_{s}^{t}(\tau-s)\| E(\tau)\|_{\dot{B}^a_{\infty,\infty}} \sup_{\alpha}\frac{(|\alpha|+\|\delta_\alpha^x Y_{\tau,t}^\pm\|_{L^\infty_{x,v}})^a}{|\alpha|^a}\, d\tau\\
\lesssim&\ \varepsilon\int_{s}^{t}\frac{(\tau-s)}{\langle \tau\rangle^{d-\gamma+1}} \Big(1+\|\nabla_xY_{\tau,t}^\pm\|_{L^\infty_{x,v}}\Big)^a\, d\tau \\ \lesssim&\frac{\varepsilon}{\langle s\rangle^{d-\gamma-1}}.
\end{align*}
Similarly,
\begin{align*}
\sup_{\alpha}\frac{\|\delta_{\alpha}^vY_{s,t}^\pm\|_{L^\infty_{x,v}}}{|\alpha|^a}
\leq& \int_{s}^{t}(\tau-s)\| E(\tau)\|_{\dot{B}^a_{\infty,\infty}} \sup_{\alpha}\frac{(|\alpha\tau|+\|\delta_\alpha^v Y_{\tau,t}^\pm\|_{L^\infty_{x,v}})^a}{|\alpha|^a}\,d\tau\\
\lesssim&\ \varepsilon\int_{s}^{t}\frac{(\tau-s)}{\langle \tau\rangle^{d-\gamma+1}} \left(\tau+\|\nabla_vY_{\tau,t}^\pm\|_{L^\infty_{x,v}}\right)^{a}\,d\tau\\
\lesssim&\frac{\varepsilon}{\langle s\rangle^{d-a-\gamma-1}}.
\end{align*}

The estimates for $W_{s,t}^\pm$ follow in the same way from the second identity in \eqref{character001}.
Finally, since
\[
\nabla_x\mathcal X^\pm
=I+\nabla_xY^\pm_{s,t},
\qquad
\nabla_v\mathcal V^\pm
=I+\nabla_vW^\pm_{s,t},
\]
and
\[
\|\nabla_xY^\pm_{s,t}\|_{L^\infty}
+\|\nabla_vW^\pm_{s,t}\|_{L^\infty}
\le \varepsilon,
\]
for $\varepsilon$ sufficiently small,
both maps are global
$\mathcal C^1$ diffeomorphisms.
\end{proof}

Specializing Lemma~\ref{estimates001} to the case \(s=0\) and applying the triangle inequality, we obtain the following corollary.

\begin{corollary}
Under the assumptions of Lemma~\ref{estimates001}, for all \(t\ge0\),
\begin{align}\label{cor:estimates001-difference}
&\|W_{0,t}^+-W_{0,t}^-\|_{L^\infty_{x,v}}
+\|Y_{0,t}^+-Y_{0,t}^-\|_{L^\infty_{x,v}}
+\|\nabla_v(W_{0,t}^+-W_{0,t}^-)\|_{L^\infty_{x,v}}\nonumber\\
&\quad+\|\nabla_v(Y_{0,t}^+-Y_{0,t}^-)\|_{L^\infty_{x,v}}+\sup_{\alpha}\frac{\|\delta_\alpha^v(W_{0,t}^+-W_{0,t}^-)\|_{L^\infty_{x,v}}}{|\alpha|^a}
+\sup_{\alpha}\frac{\|\delta_\alpha^v(Y_{0,t}^+-Y_{0,t}^-)\|_{L^\infty_{x,v}}}{|\alpha|^a}
\lesssim \varepsilon.
\end{align}
\end{corollary}

\begin{lemma}\label{Lemma3.4-YKB}
Under the assumptions of Lemma~\ref{estimates001}, let
$$
\Xi_{\mathcal{X},\theta}(0) := \theta \mathcal{X}^+(0)
+ (1-\theta)\mathcal{X}^-(0),  \quad \theta\in[0,1].
$$
For each fixed $x\in\mathbb R^d$, the maps
$v\mapsto \mathcal X^\pm(0)$
are global $\mathcal C^1$ diffeomorphisms. For each fixed
$(v,\theta)\in\mathbb R^d\times[0,1]$ and $(x,\theta)\in\mathbb R^d\times[0,1]$,
the maps
\[
x\mapsto \Xi_{\mathcal X,\theta}(0),
\qquad
v\mapsto \Xi_{\mathcal X,\theta}(0),
\]
are global $\mathcal C^1$ diffeomorphisms.

In addition, one has
\begin{equation*}
\big|\det\big(\nabla_x \mathcal{X}^+(0)\big)\big|^{-1}+\big|\det\!\big(\nabla_x\Xi_{\mathcal{X},\theta}(0)
\big)\big|^{-1}
\lesssim 1,
\end{equation*}
\begin{equation*}
\big|\det\big(\nabla_v \mathcal{X}^+(0)\big)\big|^{-1}+\big|\det\!\big(\nabla_v\Xi_{\mathcal{X},\theta}(0)
\big)\big|^{-1}
\lesssim t^{-d}.
\end{equation*}
Moreover, for all $0\le s< t$,
\begin{equation}\label{yabi01}
\big|\det\big(\nabla_x\mathcal X^\pm(s;t,\,x-\alpha,\,v-\frac{\alpha}{t})\big)\big|^{-1}
\lesssim 1,
\end{equation}
and for all $0\le s\leq\frac{t}{2}$ with $t>1$,
\begin{equation}\label{yabi02}
\big|\det\big(\nabla_v\mathcal X^\pm(s;t,\,x-\alpha,\,v-\frac{\alpha}{t})\big)\big|^{-1}
\lesssim (t-s)^{-d}.
\end{equation}
\end{lemma}

\begin{proof}
We first consider the Jacobian with respect to $x$. From \eqref{z4},
$$
\nabla_x\mathcal{X}^\pm(s)
=
I+\nabla_xY^\pm_{s,t}(x-tv,v).
$$
Hence, by \eqref{tezhengY},
\[
\|\nabla_x\mathcal X^\pm(s)-I\|_{L^\infty_{x,v}}
\lesssim \varepsilon,
\qquad 0\le s<t.
\]
Similarly,
\[
\|\nabla_x\Xi_{\mathcal X,\theta}(0)-I\|_{L^\infty_{x,v}}
\lesssim \varepsilon.
\]
If $\varepsilon>0$ is sufficiently small, then for each fixed $(v,\theta)$ the map
$x\mapsto \Xi_{\mathcal X,\theta}(0)$ is a  $\mathcal C^1$ diffeomorphism.
Moreover,
\[
\mathcal X^\pm\bigl(s;t,x-\alpha,v-\frac{\alpha}{t}\bigr)
=
x-(t-s)v-\frac{s}{t}\alpha
+
Y^\pm_{s,t}\bigl(x-tv,v-\frac{\alpha}{t}\bigr),
\]
thus
\[
\bigl\|
\nabla_x\mathcal X^\pm\bigl(s;t,x-\alpha,v-\frac{\alpha}{t}\bigr)-I
\bigr\|_{L^\infty_{x,v}}
\lesssim \varepsilon.
\]
Therefore,
\[
\big|\det\big(\nabla_x \mathcal{X}^+(0)\big)\big|^{-1}
+\big|\det\!\big(\nabla_x\Xi_{\mathcal{X},\theta}(0)\big)\big|^{-1}
+\big|\det\big(\nabla_x\mathcal X^\pm(s;t,\,x-\alpha,\,v-\frac{\alpha}{t})\big)\big|^{-1}
\lesssim 1.
\]

We next consider the Jacobian with respect to $v$.
Differentiating \eqref{z4}, we obtain
\[
\nabla_v\mathcal X^\pm(s)+(t-s)I
=
-t\nabla_xY^\pm_{s,t}(x-tv,v)+\nabla_vY^\pm_{s,t}(x-tv,v).
\]
By \eqref{tezhengY}, for $0\le s\le \frac{t}{2}$ and $t>1$,
\[
\|\nabla_v\mathcal X^\pm(s)+(t-s)I\|_{L^\infty_{x,v}}
\le
t\|\nabla_xY^\pm_{s,t}\|_{L^\infty_{x,v}}
+\|\nabla_vY^\pm_{s,t}\|_{L^\infty_{x,v}}
\lesssim \varepsilon (t-s).
\]
In particular,
\[
\|\nabla_v\mathcal X^\pm(0)+tI\|_{L^\infty_{x,v}}
\lesssim \varepsilon t,
\qquad t\ge1.
\]

It remains to treat the regime $0<t<1$. Set
\[
A^\pm(s):=\nabla_v\mathcal X^\pm(s).
\]
Differentiating the characteristic system, we find
\[
\frac{d^2}{ds^2}A^\pm(s)
=
\pm \nabla_xE(s,\mathcal X^\pm(s))A^\pm(s),
\qquad
A^\pm(t)=0,\quad \dot A^\pm(t)=I.
\]
Integrating twice gives
\[
A^\pm(s)
=
-(t-s)I
\pm\int_s^t\int_\tau^t
\nabla_xE(\sigma,\mathcal X^\pm(\sigma))
A^\pm(\sigma)\,d\sigma d\tau .
\]
Using \eqref{ziyouliu01}, we have
\[
\|A^\pm(s)\|_{L^\infty}\lesssim t-s,
\qquad
\|A^\pm(s)+(t-s)I\|_{L^\infty}
\lesssim \varepsilon (t-s)^3,
\qquad 0\le s\le t<1.
\]
In particular,
\[
\|\nabla_v\mathcal X^\pm(0)+tI\|_{L^\infty}
\lesssim \varepsilon t^3.
\]

Combining the cases $t\ge 1$ and $0<t<1$, we obtain
\[
\|\nabla_v\mathcal X^\pm(0)+tI\|_{L^\infty_{x,v}}
\lesssim \varepsilon t,\quad
\|\nabla_v\Xi_{\mathcal X,\theta}(0)+tI\|_{L^\infty_{x,v}}
\lesssim \varepsilon t,\qquad t>0.
\]
Hence,
for each fixed $x\in\mathbb R^d$, the map
$v\mapsto \mathcal X^\pm(0)$ and for each fixed $(x,\theta)\in\mathbb R^d\times[0,1]$, the map
$v\mapsto \Xi_{\mathcal X,\theta}(0)$
are $\mathcal C^1$ diffeomorphisms provided $\varepsilon>0$ is sufficiently small.
Hence
\[
\bigl|\det(\nabla_v\mathcal X^+(0))\bigr|^{-1}
+\bigl|\det(\nabla_v\Xi_{\mathcal X,\theta}(0))\bigr|^{-1}
\lesssim t^{-d}.
\]

Finally,
\[
\nabla_v\mathcal X^\pm\bigl(s;t,x-\alpha,v-\frac{\alpha}{t}\bigr)
=-(t-s)I-t\nabla_xY^\pm_{s,t}
+\nabla_vY^\pm_{s,t},
\]
and hence, for $0\le s\le \frac{t}{2}$ and $t>1$,
\[
\bigl\|
\nabla_v\mathcal X^\pm\bigl(s;t,x-\alpha,v-\frac{\alpha}{t}\bigr)+(t-s)I
\bigr\|_{L^\infty_{x,v}}
\lesssim \varepsilon (t-s),
\]
which yields \eqref{yabi02}.
This completes the proof.
\end{proof}

The next lemma introduces the inverse velocity map associated with the characteristic flow.
We follow here the same Lagrangian construction as in \cite[Proposition 5.1]{HanKwanD2021}.

\begin{lemma}\label{proposition of psi and phi}
There exists $\varepsilon_1\in(0,1)$ such that the following holds for all $0<\varepsilon\le \varepsilon_1$. For all $0\le s\le t<\infty$ and  $x,v\in \mathbb{R}^d$, there is a
$\mathcal{C}^1$ map $(x,v) \mapsto \Psi^\pm_{s,t}(x,v)$ such that
\[
\mathcal{X}^\pm(s;t,x,\Psi_{s,t}^\pm(x,v))=x-(t-s)v.
\]
Moreover, for each fixed $x$, the map $v \mapsto \Psi_{s,t}^\pm(x,v)$ is a $\mathcal{C}^1$ diffeomorphism.  In addition,
\begin{equation}\label{fangchonh0890}
 \sup_{0 \leq s \leq t} \Big( \langle s \rangle^{d-\gamma} \| \Psi_{s,t}^\pm-v \|_{L^\infty_{x,v}}
+\langle s\rangle^{d-\gamma}\|\nabla_x \Psi_{s,t}^\pm \|_{L^\infty_{x,v}} +\langle s \rangle^{d-\gamma-1} \| \nabla_v (\Psi_{s,t}^\pm-v) \|_{L^\infty_{x,v}} \Big)\lesssim \varepsilon.
\end{equation}
\end{lemma}

\begin{proof}
The case $s=t$ is trivial: one may simply take $\Psi_{t,t}^\pm(x,v)=v$, then
\[
\mathcal{X}^\pm(t;t,x,\Psi_{t,t}^\pm(x,v))=x.
\]

We now consider the case $s<t$. Define
\[
\Phi^\pm_{s,t}(x,v)=\mp\frac{1}{t-s}\int_{s}^{t}(\tau-s)E\left(\tau,x-(t-\tau)v
+Y_{\tau,t}^\pm(x-vt,v)\right)\, d\tau.
\]
Then, by~\eqref{z4},
\begin{equation}\label{tezheng0}
\mathcal{X}^\pm(s)=x-(t-s)(v+\Phi_{s,t}^\pm(x,v)).
\end{equation}
We first estimate $\Phi_{s,t}^\pm$. Using \eqref{ziyouliu01}, we obtain
\[
\|\Phi_{s,t}^{\pm}\|_{L^\infty_{x,v}}
\le
\int_s^t \|E(\tau)\|_{L^\infty}\,d\tau
\lesssim
\frac{\varepsilon}{\langle s\rangle^{d-\gamma}}.
\]
Similarly,
\[
\|\nabla_x \Phi_{s,t}^{\pm}\|_{L^\infty_{x,v}}
\le
\frac{1}{t-s}\int_s^t (\tau-s)\|\nabla_x E(\tau)\|_{L^\infty}
\bigl(1+\|\nabla_x Y_{\tau,t}^{\pm}\|_{L^\infty_{x,v}}\bigr)\,d\tau
\lesssim
\frac{\varepsilon}{\langle s\rangle^{d-\gamma}}.
\]
For the velocity gradient, differentiating the phase in $v$ gives
\[
\nabla_v\!\Bigl(x-(t-\tau)v+Y_{\tau,t}^\pm(x-tv,v)\Bigr)
=
-(t-\tau)I+(\nabla_v-t\nabla_x)Y_{\tau,t}^\pm(x-tv,v).
\]
Writing
\[
(\nabla_v-t\nabla_x)Y_{\tau,t}^\pm
=
(\nabla_v-\tau\nabla_x)Y_{\tau,t}^\pm-(t-\tau)\nabla_xY_{\tau,t}^\pm,
\]
we infer
\begin{align*}
&\|\nabla_v \Phi_{s,t}^{\pm}\|_{L^\infty_{x,v}}\\
&\le \frac{1}{t-s}\int_s^t (\tau-s)\|\nabla_xE(\tau)\|_{L^\infty}
\Big((t-\tau)+\|(\nabla_v-\tau\nabla_x)Y_{\tau,t}^{\pm}\|_{L^\infty_{x,v}}
+(t-\tau)\|\nabla_xY_{\tau,t}^{\pm}\|_{L^\infty_{x,v}}\Big)\,d\tau.
\end{align*}
Since
\[
\|(\nabla_v-\tau\nabla_x)Y_{\tau,t}^\pm\|_{L^\infty_{x,v}}
\le
\|\nabla_vY_{\tau,t}^\pm\|_{L^\infty_{x,v}}
+\tau\|\nabla_xY_{\tau,t}^\pm\|_{L^\infty_{x,v}}
\lesssim \tau+1,
\]
by \eqref{tezhengY}, it follows that
$$
\|\nabla_v \Phi_{s,t}^{\pm}\|_{L^\infty_{x,v}}
\lesssim \varepsilon \int_s^t \frac{1}{\langle \tau\rangle^{d-\gamma+1}}
\left[(t-\tau)+(\tau+1)+(t-\tau)\right]\,d\tau
\lesssim \frac{\varepsilon}{\langle s\rangle^{d-\gamma-1}}.
$$

For each fixed $x\in\mathbb{R}^d$ and consider the map
$\mathcal{P}_x(v):=v+\Phi_{s,t}^\pm(x,v)$.
If $\varepsilon>0$ is sufficiently small, then $\|\nabla_v\Phi_{s,t}^\pm\|_{L^\infty_{x,v}}$ is small, hence $\nabla_v \mathcal{P}_x$ is uniformly invertible.
Therefore, for each fixed $x\in\mathbb R^d$, the map
$v\mapsto P_x(v)$ is a $C^1$ diffeomorphism of $\mathbb R^d$.

We therefore define $\Psi_{s,t}^\pm(x,v)$ as the unique solution to
\begin{equation}\label{phi}
v=\Psi_{s,t}^\pm(x,v)+\Phi_{s,t}^\pm(x, \Psi_{s,t}^\pm(x,v)).
\end{equation}
Substituting \eqref{phi} into \eqref{tezheng0} yields
\[
\mathcal{X}^\pm\left(s; t, x, \Psi_{s,t}^\pm(x,v)\right)
=x-(t-s)v,
\]
which is the desired identity.

Finally, differentiating \eqref{phi} in $x$ and $v$ and using the bounds on $\Phi_{s,t}^\pm$ and its derivatives,
we obtain \eqref{fangchonh0890} from standard implicit-function estimates. This completes the proof.
\end{proof}

%%%%%%%%%%%%%%%%%%%%%%%%%%%%%%%%%%%%%%%
\section{Finite-Time Estimates}\label{disijie}
In this section, we establish quantitative a priori estimates on finite time
intervals for the distribution functions and the associated charge density.
Under the bootstrap assumptions \eqref{ziyouliu02}, we first propagate
weighted $W^{2,\infty}$ bounds for $f^\pm$, and then use these bounds to control the corresponding macroscopic density.

\subsection{Finite-Time A Priori Estimates}

We begin with the propagation of weighted $W^{2,\infty}$ bounds for the distribution functions.

\begin{lemma}\label{youjie010}
Let $k > d$  and assume that the bootstrap bound \eqref{ziyouliu02} holds. If
\[
\|f_{\mathrm{in}}^\pm\|_{W^{2,\infty}_k(\mathbb{R}^d\times\mathbb{R}^d)}\leq M_0,
\]
then, for every $t \ge 0$,
\begin{equation}\label{lemma-4.1}
\|f^\pm(t)\|_{W^{2,\infty}_k(\mathbb{R}^d\times\mathbb{R}^d)}\leq M_0C(t),
\end{equation}
where $C(t)>0$ is a nondecreasing function.
\end{lemma}

The next proposition transfers these microscopic bounds to the macroscopic density.

\begin{proposition}\label{xiaoshijian008}
Let $0<\varepsilon_0<1$ and assume that the hypotheses of Lemma~\ref{youjie010} hold.
If
\[
\|f_{\mathrm{in}}\|_{L_v^1L_x^1}+\|f_{\mathrm{in}}\|_{ L_v^1L_x^\infty}+\|\nabla_xf_{\mathrm{in}}\|_{L_v^1L_x^1}
+\|\nabla_xf_{\mathrm{in}}\|_{ L_v^1L_x^\infty}\leq\varepsilon_0,
\]
then, for every $t \geq 0$,
\[
\|\rho(t)\|_{L^1_x}
+\|\rho(t)\|_{L^\infty_x}+\|\rho(t)\|_{\dot B^a_{1,\infty}}
+\|\rho(t)\|_{\dot B^a_{\infty,\infty}}+\|\nabla_x\rho(t)\|_{L^1_x}
+\|\nabla_x\rho(t)\|_{L^\infty_x}
\le M_0 C(t)\varepsilon_0,
\]
where $C(t)>0$ is a nondecreasing function.
\end{proposition}

\subsection{Proof of Lemma \ref{youjie010}}
Set
\[
h^\pm:=\langle v\rangle^{k} f^{\pm},
\quad
h_{x}^\pm:=\langle v\rangle^{k}\nabla_{x}f^{\pm},
\quad
h_{v}^\pm:=\langle v\rangle^{k}\nabla_{v}f^{\pm},
\]
and similarly
\[
h_{xv}^\pm:=\langle v\rangle^k\nabla_x\nabla_v f^\pm,\qquad
h_{vv}^\pm:=\langle v\rangle^k\nabla_v^2 f^\pm,\qquad
h_{xx}^\pm:=\langle v\rangle^k\nabla_x^2 f^\pm.
\]

\medskip
\textit{Step 1. Zeroth-order estimate.}
Multiplying \eqref{VKG} by $\langle v \rangle^k$ yields
\[
\left[\partial_t+v\cdot\nabla_x\pm E\cdot\nabla_v \right](\langle v\rangle^kf^\pm)
=\mp\langle v\rangle^k E\cdot\nabla_v\mu\pm (E\cdot\nabla_v\langle v\rangle^k)f^\pm.
\]
Integrating along the characteristics \eqref{character1}, we obtain
\begin{align*}
|h^\pm (t,x,v)|
\leq|&\langle \mathcal{V}^\pm(0)\rangle^kf_{\mathrm{in}}^\pm(\mathcal{X}^\pm(0),\mathcal{V}^\pm(0))|\\
&+\int^t_0|E(s,\mathcal{X}^\pm(s))|\big|\langle \mathcal{V}^\pm(s)\rangle^k\nabla_v\mu(\mathcal{V}^\pm(s))\big|\,ds\\
&+\int_0^t |E(s,\mathcal{X}^\pm(s))|\,\big|\big(\nabla_v\langle \mathcal{V}^\pm(s)\rangle^k\big)\, f^\pm(s,\mathcal{X}^\pm(s),\mathcal{V}^\pm(s))\big|\,ds.
\end{align*}
Since
\[
|\nabla_v\langle v\rangle^k|\lesssim \langle v\rangle^{k-1}\lesssim \langle v\rangle^k,
\]
and Assumption~\ref{ass:regularity}, we have
\[
\|h^\pm(t)\|_{L_{x,v}^\infty}
\lesssim\ \|\langle v\rangle^kf_{\mathrm{in}}^\pm\|_{L_{x,v}^\infty}
+\int^t_0\|E(s)\|_{L^\infty_x}ds
+\int^t_0\|E(s)\|_{L^\infty_x}\|h^\pm(s)\|_{L_{x,v}^\infty}\,ds.
\]
%Since $d\ge 3$ and $\gamma>0$ is sufficiently small,
%\[
%\int_0^t \|E(s)\|_{L^\infty_x}\,ds
%\lesssim
%\varepsilon\int_0^t \langle s\rangle^{-(d+1)+\gamma}\,ds
%\lesssim 1.
%\]
It follows from \eqref{ziyouliu01} that
\[
\|h^\pm(t)\|_{L^\infty_{x,v}}
\lesssim
M_0+\varepsilon\int_0^t \langle s\rangle^{-(d+1)+\gamma}
\|h^\pm(s)\|_{L^\infty_{x,v}}\,ds.
\]
By Gronwall's inequality,
\begin{equation}\label{lemme4.1-3}
\|h^\pm(t)\|_{L^\infty_{x,v}}\lesssim M_0.
\end{equation}

\medskip
\textit{Step 2. First-order derivatives.}
Applying $\partial_{v_i}$ to \eqref{VKG}, we obtain
\begin{equation}\label{vlasov010}
\left[\partial_t+v\cdot\nabla_x\pm E\cdot\nabla_v\right](\partial_{v_i}f^\pm)=-\partial_{x_i}f^\pm
\mp E\cdot\nabla_v\partial_{v_i}\mu.
\end{equation}
Multiplying by $\langle v \rangle^k$ gives
\[
\left[\partial_t+v\cdot\nabla_x\pm E\cdot\nabla_v\right](\langle v\rangle^k\partial_{v_i}f^\pm)
=-\langle v\rangle^k\partial_{x_i}f^\pm\mp \langle v\rangle^k E\cdot\nabla_v\partial_{v_i}\mu\pm E\cdot\nabla_v\langle v\rangle^k\partial_{v_i}f^\pm.
\]
Integrating along characteristics, we obtain
\begin{align*}
\|h^\pm_v(t)\|_{L_{x,v}^\infty}
\lesssim&\ \|\langle v\rangle^k\nabla_vf_{\mathrm{in}}^\pm\|_{L_{x,v}^\infty}+\int^t_0\|\langle v\rangle^k\nabla_xf^\pm(s)\|_{L_{x,v}^\infty}\,ds\\
&+\int^t_0\|E(s)\|_{L^\infty_x}\,ds+\int^t_0\|E(s)\|_{L^\infty_x}
\|h^\pm_v(s)\|_{L_{x,v}^\infty}\,ds.
\end{align*}
Using \eqref{ziyouliu01}, we infer
\begin{equation*}
\|h^\pm_v(t)\|_{L_{x,v}^\infty}
\lesssim\ M_0+\int^t_0\|h^\pm_x(s)\|_{L_{x,v}^\infty}\,ds
+\varepsilon\int^t_0\langle s\rangle^{-(d+1)+\gamma}\|h^\pm_v(s)\|_{L_{x,v}^\infty}\,ds.
\end{equation*}

Next, applying $\partial_{x_j}$ to \eqref{VKG}, multiplying by $\langle v\rangle^k$, we obtain
\begin{equation*}
\left[\partial_{t}+v\cdot\nabla_{x}\pm E\cdot\nabla_{v}\right]\big(\langle v\rangle^k\partial_{x_j}f^\pm\big)
=\mp\partial_{x_j}E\cdot \langle v\rangle^k\nabla_v f^\pm\mp \partial_{x_j}E\cdot\langle v\rangle^k\nabla_{v}\mu\pm
E\cdot\nabla_v\langle v\rangle^k\partial_{x_j}f^\pm.
\end{equation*}
Integrating along characteristics and using \eqref{ziyouliu01} together with Step~1, we get
\begin{align*}
\|h^\pm_x(t)\|_{L_{x,v}^\infty}
\lesssim&\  \|\langle v\rangle^k\nabla_xf_{\mathrm{in}}^\pm\|_{L_{x,v}^\infty}+\int^t_0\|\nabla_xE(s)
\|_{L^\infty_x}\|h^\pm_v(s)\|_{L_{x,v}^\infty}\,ds
\\
&+\int^t_0\|\nabla_xE(s)\|_{L^\infty_x}ds+\int^t_0\|E(s)\|_{L^\infty_x}
\|h^\pm_x(s)\|_{L^\infty_{x,v}}\,ds\\
\lesssim&\  M_0+\varepsilon\int^t_0\langle s\rangle^{-(d+1)+\gamma}\left(\|h^\pm_v(s)\|_{L_{x,v}^\infty}
+\|h^\pm_x(s)\|_{L^\infty_{x,v}}\right)\,ds.
\end{align*}
Combining the preceding two estimates, we have
\[
\|h^\pm_v(t)\|_{L^\infty_{x,v}}+\|h^\pm_x(t)\|_{L^\infty_{x,v}}\lesssim M_0+\int_0^t\bigl(1+\varepsilon\langle s\rangle^{-(d+1)+\gamma}\bigr)
(\|h^\pm_v(s)\|_{L^\infty_{x,v}}+\|h^\pm_x(s)\|_{L^\infty_{x,v}})\,ds .
\]
Hence, by Gronwall's inequality, it holds that
\begin{equation}\label{lemme4.1-4}
\|h^\pm_v(t)\|_{L^\infty_{x,v}}+\|h^\pm_x(t)\|_{L^\infty_{x,v}}
\le M_0 C(t).
\end{equation}

\medskip
\textit{Step 3. Second-order derivatives.}
Applying $\partial_{x_j}\partial_{v_i}$ to \eqref{VKG}, multiplying by $\langle v\rangle^k$, we obtain
\begin{align*}
&\left[\partial_t+v\cdot\nabla_x\pm E\cdot\nabla_v\right](\langle v\rangle^k\partial_{x_j}\partial_{v_i}f^\pm)\\
&\qquad=
-\langle v\rangle^k\partial_{x_j}\partial_{x_i}f^\pm
\mp \partial_{x_j}E\cdot \langle v\rangle^k\nabla_v(\partial_{v_i}f^\pm)
\mp \partial_{x_j}E\cdot \langle v\rangle^k\nabla_v\partial_{v_i}\mu\\
&\qquad\quad
\pm (\partial_{x_j}E\cdot\nabla_v\langle v\rangle^k)\,\partial_{v_i}f^\pm
\pm (E\cdot\nabla_v\langle v\rangle^k)\,\partial_{x_j}\partial_{v_i}f^\pm .
\end{align*}
Integrating along characteristics and using \eqref{ziyouliu01} together with Step~2, we obtain
\begin{align*}
\|h^\pm_{xv}(t)\|_{L_{x,v}^\infty}
\lesssim&\ \|\langle v\rangle^k\nabla_{x}\nabla_{v}f_{\mathrm{in}}^\pm\|_{L_{x,v}^\infty}
+\int^t_0\|h^\pm_{xv}(s)\|_{L_{x,v}^\infty}\|E(s)\|_{L^\infty_x}\,ds\\
&+\int^t_0\|\nabla_xE(s)\|_{L^\infty_x}\|h^\pm_{vv}(s)\|_{L_{x,v}^\infty}\,ds
+\int^t_0\|h^\pm_{xx}(s)\|_{L_{x,v}^\infty}\,ds
+\int^t_0\|\nabla_xE(s)\|_{L^\infty_x}\,ds.
\end{align*}
Next, after applying $\partial_{v_j}\partial_{v_i}$ to \eqref{VKG}, multiplying by $\langle v\rangle^k$, and using Assumption~\ref{ass:regularity}, we obtain
\begin{align*}
\|h^\pm_{vv}(t)\|_{L_{x,v}^\infty}
\lesssim&\ \|\langle v\rangle^k\nabla_v^2f_{\mathrm{in}}^\pm\|_{L_{x,v}^\infty}
+\int^t_0\|E(s)\|_{L^\infty_x}\,ds\\
&+\int^t_0\|h^\pm_{xv}(s)\|_{L_{x,v}^\infty}\,ds
+\int^t_0\|E(s)\|_{L^\infty_x}\|h^\pm_{vv}(s)\|_{L_{x,v}^\infty}\,ds.
\end{align*}
Similarly, applying two spatial derivatives to \eqref{VKG}, multiplying by $\langle v\rangle^k$, and integrating along characteristics, we get
\begin{align*}
\|h^\pm_{xx}(t)\|_{L_{x,v}^\infty}
\lesssim&\ \|\langle v\rangle^k\nabla_x^2f_{\mathrm{in}}^\pm\|_{L_{x,v}^\infty}
+\int^t_0\|\nabla^2_xE(s)\|_{L^\infty_x}\,ds
+\int^t_0\|\nabla^2_xE(s)\|_{L^\infty_x}\|h^\pm_v(s)\|_{L_{x,v}^\infty}\,ds\\
&+\int^t_0\|\nabla_xE(s)\|_{L^\infty_x}\|h^\pm_{xv}(s)\|_{L_{x,v}^\infty}\,ds
+\int^t_0\|E(s)\|_{L^\infty_x}\|h^\pm_{xx}(s)\|_{L_{x,v}^\infty}\,ds.
\end{align*}
Define
\[
H^\pm(t):= \|h^\pm_{xv}(t)\|_{L_{x,v}^\infty}+\|h^\pm_{vv}(t)\|_{L_{x,v}^\infty}+ \|h^\pm_{xx}(t)\|_{L_{x,v}^\infty}, \  \ \widetilde{H}(t):= H^+(t)+H^-(t).
\]
By Lemma~\ref{field01}, Step~2, and \eqref{ziyouliu02}, it holds that
\begin{align*}
\widetilde{H}(t)
\lesssim&M_0
+\int_0^t \widetilde H(s)\,ds
+\int_0^t \|\nabla_x^2 E(s)\|_{L^\infty_x}\,ds
+\int_0^t \|\nabla_x^2 E(s)\|_{L^\infty_x}\,M_0 C(s)\,ds \nonumber\\
&+\varepsilon\int_0^t \langle s\rangle^{-(d+1)+\gamma}\widetilde H(s)\,ds.
\end{align*}

Since
\[
\rho(t,x)=\int_{\mathbb R^d}\bigl(f^+(t,x,v)-f^-(t,x,v)\bigr)\,dv,
\]
and $k>d$, we have
$$
\|\nabla_x^2\rho(t)\|_{L^\infty_x}
\le
\int_{\mathbb R^d}\langle v\rangle^{-k}\,dv\,
\Big(\|h^+_{xx}(t)\|_{L^\infty_{x,v}}+\|h^-_{xx}(t)\|_{L^\infty_{x,v}}\Big)\lesssim \widetilde H(t).
$$
Hence Lemma~\ref{field01} yields
\[
\|\nabla_x^2 E(t)\|_{L^\infty_x}
\lesssim \|\nabla_x^2 \rho(t)\|_{L^\infty_x}
\lesssim \widetilde H(t).
\]
Therefore
\[
\widetilde{H}(t)\lesssim M_0+\int^t_0\left(1+M_0C(s)\right) \widetilde{H}(s)\,ds.
\]
By Gronwall's inequality,
\begin{equation}\label{lemme4.1-5}
\widetilde{H}(t) \leq M_0C(t).
\end{equation}
Combining \eqref{lemme4.1-3} with \eqref{lemme4.1-4} and \eqref{lemme4.1-5}, we conclude \eqref{lemma-4.1}.
This completes the proof.

\subsection{Proof of Proposition \ref{xiaoshijian008}}
The starting point is the standard Duhamel representation
\begin{align}\label{midu010}
\rho(t,x)=&\int_{\mathbb{R}^d}f_{\mathrm{in}}(x-tv,v)\,dv\nonumber\\
&-\int_0^t\int_{\mathbb{R}^d}E\left(s,x-v(t-s)\right)\cdot
\nabla_v(f^++f^-)\left(s,x-v(t-s),v\right)\,dv\,ds\nonumber\\
&-2\int_0^t\int_{\mathbb{R}^d}E\left(s,x-v(t-s)\right)\cdot\nabla_v\mu(v)
\,dv\,ds.
\end{align}

\medskip
\textit{$L^1$-estimate for $\rho$.}
Taking the $L^1_x$ norm in \eqref{midu010}, using Minkowski's inequality, the translation invariance of the $L^1_x$ norm, $\|\nabla_v\mu\|_{L^1_v}<\infty$, and
$
\|E(s)\|_{L^1_x}\lesssim \|\rho(s)\|_{L^1_x},
$
we obtain
\begin{align}\label{midu001}
\|\rho(t)\|_{L^1_x}
\leq&\ \|f_{\mathrm{in}}\|_{L^1_vL^1_x}
+2\int_0^t\|\nabla_v\mu\|_{L^1_v}\,\|E(s)\|_{L^1_x}\,ds \nonumber\\
&+\int_0^t \|E(s)\|_{L^1_x}
\int_{\R^d}\big\|\nabla_v(f^++f^-)(s,\cdot,v)\big\|_{L^\infty_x}\,dv\,ds \nonumber\\
\leq&\ \|f_{\mathrm{in}}\|_{L^1_vL^1_x}
+C\int_0^t\big(1+M_0C(s)\big)\,\|\rho(s)\|_{L^1_x}\,ds .
\end{align}
In the last step we used Lemma~\ref{youjie010} and the velocity weight $k>d$,
\[
\int_{\R^d}\|\nabla_v f^\pm(s,\cdot,v)\|_{L^\infty_x}\,dv
\le \|\langle v\rangle^k\nabla_v f^\pm(s)\|_{L^\infty_{x,v}}
\int_{\R^d}\langle v\rangle^{-k}\,dv
\lesssim M_0C(s).
\]

\medskip
\textit{$L^\infty$-estimate for $\rho$.}
Taking the $L^\infty_{x}$ norm in \eqref{midu010} and
using Lemma~\ref{youjie010}, we have
\begin{align}\label{midu002}
\|\rho(t)\|_{L^\infty_x}
\leq&\ \|f_{\mathrm{in}}\|_{L^1_vL^\infty_x}
+2\int_0^t\|\nabla_v\mu\|_{L^1_v}\,\|E(s)\|_{L^\infty_x}\,ds \nonumber\\
&+\int_0^t \|E(s)\|_{L^\infty_x}
\int_{\R^d}\big\|\nabla_v(f^++f^-)(s,\cdot,v)\big\|_{L^\infty_x}\,dv\,ds \nonumber\\
\leq&\ \|f_{\mathrm{in}}\|_{L^1_vL^\infty_x}
+C\int_0^t\big(1+M_0C(s)\big)\,\|\rho(s)\|_{L^\infty_x}\,ds.
\end{align}

\medskip
\textit{Estimates for $\nabla_x\rho$.}
Differentiating \eqref{midu010} with respect to $x_i$ yields
\begin{align*}
|\partial_{x_i}\rho(t,x)|
\leq&\int_{\mathbb{R}^d}|\partial_{x_i}
f_{\mathrm{in}}(x-tv,v)|\,dv\\
&+2\int_0^t\int_{\mathbb{R}^d}|\partial_{x_i}E(s,x-v(t-s))|
|\nabla_v\mu(v)|\,dv\,ds\\
&+\int_0^t\int_{\mathbb{R}^d}|\partial_{x_i}E(s,x-v(t-s))|
\left|\nabla_v(f^++f^-)\left(s,x-v(t-s),v\right)\right|\,dv\,ds\\
&+\int_0^t\int_{\mathbb{R}^d}|E(s,x-v(t-s))|
\left|\partial_{x_i}\nabla_v(f^++f^-)\left(s,x-v(t-s),v\right)\right|\,dv\,ds.
\end{align*}
For $p \in \{1, \infty\}$, taking the $L^p$ norm and applying Minkowski's inequality, we obtain
\begin{align*}
\|\nabla_x\rho(t)\|_{L^p_x}
\leq&\ \|\nabla_x f_{\mathrm{in}}\|_{L^1_vL^p_x}
+2\int_0^t\|\nabla_v\mu\|_{L^1_v}\,\|\nabla_xE(s)\|_{L^p_x}\,ds \nonumber\\
&+\int_0^t \|\nabla_xE(s)\|_{L^p_x}
\int_{\R^d}\big\|\nabla_v(f^++f^-)(s,\cdot,v)\big\|_{L^\infty_x}\,dv\,ds \nonumber\\
&+\int_0^t \|E(s)\|_{L^p_x}
\int_{\R^d}\big\|\nabla_x\nabla_v(f^++f^-)(s,\cdot,v)\big\|_{L^\infty_x}\,dv\,ds.
\end{align*}
Using Lemma~\ref{youjie010} and $k>d$ as above,
\[
\int_{\R^d}\|\nabla_x\nabla_v f^\pm(s,\cdot,v)\|_{L^\infty_x}\,dv
\lesssim \|\langle v\rangle^k\nabla_x\nabla_v f^\pm(s)\|_{L^\infty_{x,v}}
\lesssim M_0C(s).
\]
Therefore,
\begin{align*}
\|\nabla_x\rho(t)\|_{L^p_x}
\leq&\ \|\nabla_x f_{\mathrm{in}}\|_{L^1_vL^p_x}
+C\int_0^t\bigl(1+M_0C(s)\bigr)\|\nabla_xE(s)\|_{L^p_x}\,ds \\
&+C\int_0^t M_0C(s)\|E(s)\|_{L^p_x}\,ds .
\end{align*}
Applying Lemma~\ref{field01} with $\psi=\rho$, we have
\[
\|\nabla_x E(t)\|_{L^p_x}
\lesssim \|\nabla_x \rho(t)\|_{L^p_x},
\]
which yields
\begin{equation}\label{midu003}
\|\nabla_x\rho(t)\|_{L^p_x}
\le \|\nabla_x f_{\mathrm{in}}\|_{L^1_vL^p_x}
+C\int_0^t \big(1+M_0C(s)\big)\,\|\nabla_x\rho(s)\|_{L^p_x}\,ds .
\end{equation}

Combining \eqref{midu001}--\eqref{midu003} and applying Gronwall's inequality,
we obtain
\begin{align*}
\|\rho(t)\|_{L^1_x}+\|\rho(t)\|_{L^\infty_x} +\|\nabla_x \rho(t)\|_{L^1_x}+\|\nabla_x \rho(t)\|_{L^\infty_x}
\leq M_0C(t) \varepsilon_0.
\end{align*}
Finally, by Lemma~\ref{chazhi00}, we have
\[
\|\rho(t)\|_{\dot B^a_{1,\infty}}
\lesssim
\|\rho(t)\|_{L^1_x}^{1-a}\|\nabla_x\rho(t)\|_{L^1_x}^{a},
\qquad
\|\rho(t)\|_{\dot B^a_{\infty,\infty}}
\lesssim
\|\rho(t)\|_{L^\infty_x}^{1-a}\|\nabla_x\rho(t)\|_{L^\infty_x}^{a}.
\]
Hence,
\[
\|\rho(t)\|_{\dot B^a_{1,\infty}}
+\|\rho(t)\|_{\dot B^a_{\infty,\infty}}
\le M_0C(t)\varepsilon_0 .
\]
This completes the proof.
\hfill$\Box$

\section{Large-Time Analysis}\label{diwujie}

In this section, we establish the large-time decay estimates for the charge density.
The proof combines a refined analysis of the characteristic flow with the resolvent structure of the screened Poisson operator.

\subsection{Characteristic Representation of the Charge Density}
We begin with the Duhamel formula for the perturbations \(f^\pm\).
Integrating the equations in \eqref{VKG} along the characteristic flows
\((\mathcal X^\pm,\mathcal V^\pm)\) introduced in \eqref{character1}, we obtain
\[
f^{\pm}(t,x,v)=f_{\mathrm{in}}^{\pm}(\mathcal{X}^{\pm}(0), \mathcal{V}^{\pm}(0))
\mp\int_0^t E(s,\mathcal{X}^{\pm}(s))\cdot\nabla_{v}\mu(\mathcal{V}^{\pm}(s))\,ds.
\]
Integrating with respect to \(v\), one gets that
\begin{equation*}
\rho^{\pm}(t,x)=\int_{\mathbb{R}^d}f_{\mathrm{in}}^{\pm}(\mathcal{X}^{\pm}(0), \mathcal{V}^{\pm}(0))\,dv
\mp\int_0^t\int_{\mathbb{R}^d} E(s,\mathcal{X}^{\pm}(s))\cdot\nabla_{v}\mu(\mathcal{V}^{\pm}(s))\,dv\,ds.
\end{equation*}
Using $E=\nabla_x (I - \Delta_x)^{-1} \rho$, we decompose $\rho^{\pm}$ as
\begin{equation}\label{midufenjie001}
\rho^{\pm}(t,x)=I^\pm(t,x)\mp Q(t,x)\mp R^\pm(t,x),
\end{equation}
where
\begin{align*}
I^\pm(t,x):=&\int_{\mathbb{R}^d}f_{\mathrm{in}}^{\pm}(\mathcal{X}^{\pm}(0),
\mathcal{V}^{\pm}(0))\,dv,\nonumber\\
Q(t,x):=&\int_0^t\int_{\mathbb{R}^d}\left[\nabla_x(I-\Delta_x)^{-1}\rho\right](s, x-(t-s)v)\cdot\nabla_v\mu(v)\,dv\,ds,\nonumber\\
R^\pm(t,x):=&\int_0^t\int_{\mathbb{R}^d} \left[E(s,\mathcal{X}^{\pm}(s))\cdot\nabla_{v}\mu(\mathcal{V}^{\pm}(s))
- E(s,x-(t-s)v)\cdot\nabla_{v}\mu(v)\right]\,dv\,ds.
\end{align*}

We next rewrite the linear term \(Q(t,x)\) in convolution form.
Using the change of variables \(w=x-(t-s)v\), and then integrating by parts in \(v\), we obtain
\[
Q(t,x)=\int_0^t\int_{\mathbb{R}^d}\mathcal{Q}(t-s, x-w)\rho(s, w)\,dw\,ds
= (\mathcal{Q}*_{(t,x)}\rho)(t,x),
\]
where the space-time kernel is given by
\[
\mathcal{Q}(t, x)=\frac{1}{t^{d+1}}\left((1-\frac{1}{t^2}\Delta_v)^{-1} \Delta_v\mu\right)\left(\frac{x}{t}\right)\mathbf{1}_{t>0}.
\]
Consequently, since \(\rho=\rho^+-\rho^-\), we have
\[
\rho(t,x)=-2(\mathcal{Q}\ast_{(t,x)}\rho)(t,x)+\big(I^+-I^-\big)(t,x)
-\big(R^++R^-\big)(t,x).
\]

As in the linear analysis, we extend all time-dependent functions by zero to the region \(t<0\), and then take the Fourier transform in \((t,x)\).
Invoking the Penrose stability condition, we obtain
\begin{equation}\label{midu}
 \widehat{\rho}(\tau, \xi)=\frac{1}{1+2\widehat{\mathcal{Q}}(\tau, \xi)}
\Big[\Big(\widehat{I}^+-\widehat{I}^-\Big)(\tau, \xi)
-\Big(\widehat{R}^+ +\widehat{R}^-\Big)(\tau, \xi)\Big],
\end{equation}
where
\[
\widehat{\mathcal{Q}}(\tau, \xi)=\int_0^{+\infty}e^{-i\tau t}\frac{i\xi}{1+|\xi|^2}\cdot
\widehat{\nabla_v\mu}(t\xi)\,dt.
\]
We introduce the resolvent kernel
\begin{equation*}
\mathcal{G}(t,x)=\mathcal{F}_{(\tau,\xi)\rightarrow(t,x)}^{-1}\left(\frac{
2\widehat{\mathcal{Q}}(\tau, \xi)}
{1+2\widehat{\mathcal{Q}}(\tau, \xi)}\right).
\end{equation*}
Rearranging \eqref{midu} and inverting the Fourier transform, we obtain the representation
\begin{equation}\label{fangcheng000}
\rho=-\mathcal{G}\ast_{(t,x)}\left[(I^+-I^-)
-(R^+ +R^-)\right]+\left(I^+-I^-\right)-\left(R^+ +R^-\right).
\end{equation}
Moreover, since \(\widehat{\mathcal Q}(\tau,0)=0\) for every \(\tau\in\mathbb R\), the kernel \(\mathcal G\) satisfies the cancellation property
\begin{equation}\label{kernel006}
\int_{\mathbb{R}^d}\mathcal{G}(t,x)dx=0, \quad t > 0,
\end{equation}
which will be crucial for the large-time analysis.

Before deriving pointwise bounds for \(\rho\), we record the following estimate on the resolvent kernel \(\mathcal G\).
It is a direct consequence of \cite[Theorem~3.1]{Huang002}, after translating the notation and matching the weighted regularity assumptions with the present setting.

\begin{lemma}\label{kernel}
Under Assumptions~\ref{ass:regularity} and~\ref{ass:penrose}, for all $a\in(0,1)$
there exists a constant $C=C(\mu)>0$ such that
\begin{equation*}
|\nabla^j \mathcal G(t,x)|
\le
\frac{C}{(t+|x|)^{d+j-1}(1+t+|x|)^2}, \qquad j\ge0, \quad t\ge1, \quad x\in\mathbb R^d.
\end{equation*}
In particular,
\begin{equation}\label{lemma5.1-2}
\|\nabla^j \mathcal G(t)\|_{L^\infty_x}
\le \frac{C}{t^{d+j-1}(1+t)^2},
\qquad
\|\nabla^j \mathcal G(t)\|_{L^1_x}
\le \frac{C}{(1+t)^{j+1}},
\qquad j\in\{0,1\},
\end{equation}
\begin{equation}\label{lemma5.1-3}
\bigl\||x|^a\mathcal G(t)\bigr\|_{L^1_x}
\le \frac{C}{(1+t)^{1-a}},
\end{equation}
and
\begin{equation}\label{lemma5.1-4}
\bigl\||x|^a\delta_\alpha \mathcal G(t)\bigr\|_{L^1_x}
\le C\Bigl(
\frac{|\alpha|}{(1+t+|\alpha|)^{2-a}}
+\frac{|\alpha|^a\mathbf 1_{|\alpha|>t}}{1+t}
+\frac{|\alpha|^{d+a}\mathbf 1_{|\alpha|\le t}}{t^{d-1}(1+t)^2}
\Bigr).
\end{equation}
\end{lemma}

For a function $\hbar=\hbar(t,x)$, we define the space-time convolution
\begin{equation*}
(\mathcal{G}*_{(t,x)}\hbar)(t,x)=\int^t_0\int_{\mathbb{R}^d}
\mathcal{G}(t-s,y)\hbar(s,x-y)\,dy\,ds.
\end{equation*}
We next record its mapping properties in Lebesgue and Besov spaces.
\begin{lemma}\label{juanji001}
Let $a\in(0,1)$. For every $t\ge1$, one has
\begin{equation}\label{lemma5.2}
\left\{
\begin{aligned}
\|(\mathcal{G}\ast_{(t,x)}\hbar)(t)\|_{L^1_x}&\lesssim\sup_{0\leq s\leq t}\langle s\rangle^a\|\hbar(s)\|_{\dot{B}_{1,\infty}^a},\\
\langle t\rangle^d\|(\mathcal{G}\ast_{(t,x)}\hbar)(t)\|_{L^\infty_x}&\lesssim\sup_{0\leq s\leq t}(\|\hbar(s)\|_{L^1_x}
+\langle s\rangle^{d+a}\|\hbar(s)\|_{\dot{B}_{\infty,\infty}^a}),\\
\frac{\langle t\rangle}{\ln(1+t)}\|\nabla( \mathcal{G}\ast_{(t,x)}\hbar)(t)\|_{L^1_x}
&\lesssim\sup_{0\leq s\leq t}\big(\|\hbar(s)\|_{L^1_x}+\langle s\rangle\|\nabla\hbar(s)\|_{L^1_x}\big),\\
\frac{\langle t\rangle^{d+1}}{\ln (1+t)}\|\nabla(\mathcal{G}\ast_{(t,x)}\hbar)(t)\|_{L^\infty_x}&\lesssim
\sup_{0\leq s\leq t}\big(\|\hbar(s)\|_{L^1_x}+\langle s\rangle^{d+1}\|\nabla\hbar(s)\|_{L^\infty_x}\big),\\
\langle t\rangle^{a}\|(\mathcal{G}\ast_{(t,x)}\hbar)(t)\|_{\dot{B}_{1,\infty}^a}
&\lesssim \sup_{0\leq s\leq t}\langle s\rangle^{a}\|\hbar(s)\|_{\dot{B}_{1,\infty}^a},\\
\langle t\rangle^{d+a}\|(\mathcal{G}\ast_{(t,x)}\hbar)(t)\|_{\dot{B}_{\infty,\infty}^a}
&\lesssim \sup_{0\leq s\leq t}(\langle s\rangle^{a}\|\hbar(s)\|_{\dot{B}_{1,\infty}^a}+\langle s\rangle^{d+a}\|\hbar(s)\|_{\dot{B}_{\infty,\infty}^a}).
\end{aligned}
\right.
\end{equation}

\end{lemma}
\begin{proof}
By \eqref{kernel006}, we have
\[
\int_{\mathbb R^d}\mathcal{G}(t-s,y)\,dy=0.
\]
Hence
\begin{equation}\label{lemma5.2-L}
(\mathcal{G}\ast_{(t,x)}\hbar)(t,x)
=\int_0^t\int_{\R^d}\mathcal{G}(t-s,y)\,\big[\hbar(s,x-y)-\hbar(s,x)\big]\,dy\,ds .
\end{equation}

We first prove the \(L^1_x\) bound. Using \eqref{lemma5.2-L} and \eqref{lemma5.1-3}, we obtain
\begin{align*}
\|(\mathcal{G}\ast_{(t,x)}\hbar)(t)\|_{L^1_x}
\leq&
\int_0^t\int_{\mathbb R^d}
|\mathcal{G}(t-s,y)|\,\|\hbar(s,\cdot-y)-\hbar(s,\cdot)\|_{L^1_x}\,dy\,ds \nonumber\\
\leq& \int^t_0\||y|^a\mathcal{G}(t-s,y)\|_{L^1_y}
\sup\limits_y\frac{\|\delta_y\hbar(s)\|_{L_x^1}}{|y|^a}\,ds\nonumber\\
\lesssim&\ \sup\limits_{0\leq s\leq t}\langle s\rangle^a\|\hbar(s)\|_{\dot{B}_{1,\infty}^a}\int^t_0\frac{\,ds}{\langle s\rangle^a(1+t-s)^{1-a}}.
\end{align*}
Since $a\in(0,1)$, the last integral is bounded uniformly in $t\ge1$. Therefore
\begin{equation}\label{kernel02}
\|(\mathcal{G}\ast_{(t,x)}\hbar)(t)\|_{L^1_x}
\lesssim\sup\limits_{0\leq s\leq t}\langle s\rangle^a\|\hbar(s)\|_{\dot{B}_{1,\infty}^a}.
\end{equation}
We next prove the $L^\infty$ estimate.
Split $[0,t]=[0,\frac{t}{2}]\cup[\frac{t}{2},t]$.
On \([0,\frac{t}{2}]\), we use the \(L^\infty_x\)-bound on \(\mathcal{G}\); on \([\frac{t}{2},t]\), we use \eqref{lemma5.2-L}. Thus
$$
\|(\mathcal{G}\ast_{(t,x)}\hbar)(t)\|_{L^\infty_x}\leq
\int^{\frac{t}{2}}_0\|\mathcal{G}(t-s)\|_{L^\infty_x}\|\hbar(s)\|_{L^1_x}ds+
\int_{\frac{t}{2}}^t\||\cdot|^a\mathcal{G}(t-s,\cdot)\|_{L^1_x}\| \hbar(s)\|_{\dot{B}_{\infty,\infty}^a}ds.
$$
By \eqref{lemma5.1-2} and \eqref{lemma5.1-3}, it holds that
\begin{align}\label{kernel04}
\|(\mathcal{G}\ast_{(t,x)}\hbar)(t)\|_{L^\infty_x}
\lesssim&\ \sup\limits_{0\leq s\leq t}\|\hbar(s)\|_{L^1_x}
\int^\frac{t}{2}_0\frac{ds}{(t-s)^{d-1}(1+t-s)^{2}}\nonumber\\
&+\sup\limits_{0\leq s\leq t}\langle s\rangle^{d+a}\|\hbar(s)\|_{\dot{B}_{\infty,\infty}^a}
\int_{\frac{t}{2}}^t\frac{ds}{(1+t-s)^{1-a}\langle s\rangle^{d+a}}\nonumber\\
\lesssim& \ \langle t\rangle^{-d}\sup\limits_{0\leq s\leq t}(\|\hbar(s)\|_{L^1_x}
+\langle s\rangle^{d+a}\|\hbar(s)\|_{\dot{B}_{\infty,\infty}^a}).
\end{align}

We next estimate the spatial gradient. Using the same time splitting, we obtain
$$
\|\nabla( \mathcal{G}\ast_{(t,x)}\hbar)(t)\|_{L^1_x}\leq \int^{\frac{t}{2}}_0\|\nabla \mathcal{G}(t-s)\|_{L^1_x}\|\hbar(s)\|_{L^1_x}ds+\int_{\frac{t}{2}}^t\|
\mathcal{G}(t-s)\|_{L^1_x}\|\nabla\hbar(s)\|_{L^1_x}\,ds.
$$
This together with \eqref{lemma5.1-2} gives
\begin{align*}
\|\nabla( \mathcal{G}\ast_{(t,x)}\hbar)(t)\|_{L^1_x}
\lesssim&
\sup_{0\le s\le t}\|\hbar(s)\|_{L^1_x}\int_0^{\frac{t}{2}}\frac{ds}{(1+t-s)^2}
+\sup_{0\le s\le t}\langle s\rangle\|\nabla\hbar(s)\|_{L^1_x}
\int_{\frac{t}{2}}^{t}\frac{ds}{(1+t-s)\,\langle s\rangle}\\
\lesssim &\frac{\ln (1+t)}{\langle t\rangle}\sup\limits_{0\leq s\leq t}(\|\hbar(s)\|_{L^1_x}+\langle s\rangle\|\nabla\hbar(s)\|_{L^1_x}).
\end{align*}
Similarly,
$$
\|\nabla(\mathcal{G}\ast_{(t,x)}\hbar)(t)\|_{L^\infty_x}
\leq\int^{\frac{t}{2}}_0\|\nabla \mathcal{G}(t-s)\|_{L^\infty_x}\|\hbar(s)\|_{L^1_x}ds+
\int_{\frac{t}{2}}^t\|\mathcal{G}(t-s)\|_{L^1_x}\| \nabla \hbar(s)\|_{L^\infty_x}ds.
$$
By \eqref{lemma5.1-2}, we have
\begin{align*}
\|\nabla(\mathcal{G}\ast_{(t,x)}\hbar)(t)\|_{L^\infty_x}
\lesssim&\ \sup\limits_{0\leq s\leq t}\|\hbar(s)\|_{L^1_x}
\int^\frac{t}{2}_0\frac{ds}{(t-s)^{d}(1+t-s)^{2}}\nonumber\\
&+\sup\limits_{0\leq s\leq t}\langle s\rangle^{d+1}\|\nabla\hbar(s)\|_{L^\infty_x}
\int_{\frac{t}{2}}^t\frac{ds}{(1+t-s)\langle s\rangle^{d+1}}\\
\lesssim& \frac{\ln (1+t)}{\langle t\rangle^{d+1}}\sup\limits_{0\leq s\leq t}(\|\hbar(s)\|_{L^1_x}+\langle s\rangle^{d+1}\|\nabla\hbar(s)\|_{L^\infty_x}).
\end{align*}

It remains to establish the Besov estimates.
By the difference-quotient characterization of \(\dot B^a_{p,\infty}\), it suffices to estimate
\[
\sup_{\alpha}
\frac{\|\delta_\alpha(\mathcal G*_{(t,x)}\hbar)(t)\|_{L^p_x}}{|\alpha|^a},
\qquad p\in\{1,\infty\}.
\]
We decompose
\[
\delta_\alpha (\mathcal G*_{(t,x)}\hbar)(t,x)
=G_1(t,x)+G_2(t,x),
\]
where
\[
G_1(t,x):=\int_0^{\frac{t}{2}}\int_{\mathbb R^d}\mathcal G(t-s,y)\,
\delta_\alpha \hbar(s,x-y)\,dy\,ds,
\]
\[
G_2(t,x):=\int_{\frac{t}{2}}^{t}\int_{\mathbb R^d}\mathcal G(t-s,y)\,
\delta_\alpha \hbar(s,x-y)\,dy\,ds.
\]

We first estimate \(G_1\).
By Young's inequality, it holds that
$$
\frac{\|G_1(t)\|_{L^p_x}}{|\alpha|^a}
\le
\int_0^{\frac{t}{2}}\|\mathcal G(t-s)\|_{L^p_x}
\|\delta_\alpha \hbar(s)\|_{L^1_x}\frac{\,ds}{|\alpha|^a}
\lesssim
\int_0^{\frac{t}{2}}\|\mathcal G(t-s)\|_{L^p_x}
\|\hbar(s)\|_{\dot B^a_{1,\infty}}\,ds,
$$
which together with \eqref{lemma5.1-2} yields
\begin{equation}\label{eq:Besov-I1}
\frac{\|G_1(t)\|_{L^p_x}}{|\alpha|^a}\lesssim
\int_0^{\frac{t}{2}}\langle t-s\rangle^{-\frac{d(p-1)}{p}-1}
\|\hbar(s)\|_{\dot B^a_{1,\infty}}\,ds\lesssim
\langle t\rangle^{-\frac{d(p-1)}{p}-a}
\sup_{0\le s\le t}\langle s\rangle^{a}\|\hbar(s)\|_{\dot B^a_{1,\infty}}.
\end{equation}

We next estimate \(G_2\), and distinguish two cases.

\smallskip
\noindent\emph{Case 1: $|\alpha|\ge \frac{t}{2}$.}
In this regime,
$$
%\frac{\|\delta_\alpha G_2(t)\|_{L^p_x}}{|\alpha|^a}
\frac{\|G_2(t)\|_{L^p_x}}{|\alpha|^a}
\lesssim\langle t\rangle^{-a}\|G_2(t)\|_{L^p_x}.
$$
If \(p=1\), we use \eqref{kernel02}; if \(p=\infty\), we use the estimate on the interval $[\frac{t}{2},t]$ in \eqref{kernel04}. This yields
%\[
%\frac{\|\delta_\alpha G_2(t)\|_{L^p_x}}{|\alpha|^a}
%\lesssim\frac{2\|G_2(t)\|_{L^p_x}}{|\alpha|^a}
%\lesssim\frac{1}{\langle t\rangle^{a}}\sup\limits_{0\leq s\leq t}\langle s\rangle^{a}\| \hbar(s)\|_{\dot{B}_{1,\infty}^a},
%\]
%and
\begin{equation}\label{dashijian-0}
\frac{\| G_2(t)\|_{L^1_x}}{|\alpha|^a}
\lesssim\langle t\rangle^{-a}\sup\limits_{0\leq s\leq t}
\langle s\rangle^{a}\|\hbar(s)\|_{\dot{B}_{1,\infty}^a},\ \
\frac{\| G_2(t)\|_{L^\infty_x}}{|\alpha|^a}
\lesssim\langle t\rangle^{-d-a}\sup\limits_{0\leq s\leq t}\langle s\rangle^{d+a}\|\hbar(s)\|_{\dot{B}_{\infty,\infty}^a}.
\end{equation}

\smallskip
\noindent\emph{Case 2: $|\alpha|<\frac{t}{2}$.}
By expanding the difference quotient in \(\delta_\alpha\hbar\) and making the change of variables
\(z=y+\alpha\), we obtain
\[
G_2(t,x)=\int_{\frac{t}{2}}^{t}\int_{\mathbb R^d}
\delta_\alpha \mathcal G(t-s,z)\,\hbar(s,x-z)\,dz\,ds.
\]
Since \eqref{kernel006} implies
\[
\int_{\mathbb R^d}\delta_\alpha \mathcal G(t-s,z)\,dz=0,
\]
we have
\[
G_2(t,x)=\int_{\frac{t}{2}}^{t}\int_{\mathbb R^d}\delta_\alpha \mathcal G(t-s,z)
\bigl[\hbar(s,x-z)-\hbar(s,x)\bigr]\,dz\,ds.
\]

We further split the time interval into
$t-s\le |\alpha|$ and $t-s\ge |\alpha|$. Thus
$$
\|G_2(t)\|_{L_x^p}
\leq\mathcal{A}_1+\mathcal{A}_2,
$$
where
\begin{align*}
\mathcal{A}_1:= & \int_{t-|\alpha|}^t\int_{\mathbb{R}^d}
|\delta_\alpha \mathcal{G}(t-s,z)|\| \hbar(s,\cdot-z)- \hbar(s,\cdot)\|_{L_x^p}\,dz\, ds,\\
\mathcal{A}_2:= & \int_{\frac{t}{2}}^{t-|\alpha|}\int_{\mathbb{R}^d}|\delta_\alpha \mathcal{G}(t-s,z)|\| \hbar(s,\cdot-z)- \hbar(s,\cdot)\|_{L_x^p}\,dz\, ds.
\end{align*}
For $\mathcal{A}_1$, we use \(|\delta_\alpha \mathcal{G}|\le 2|\mathcal{G}|\) and \eqref{lemma5.1-3}. Since \(s\in[t-|\alpha|,t]\subset[\frac{t}{2},t]\), we have
\begin{align*}
\mathcal{A}_1\leq
&\ 2\int_{t-|\alpha|}^t
\||z|^a\mathcal{G}(t-s,z)\|_{L^1_z}\| \hbar(s)\|_{\dot{B}_{p,\infty}^a}\,ds\nonumber\\
\lesssim&\ \sup\limits_{0\leq s\leq t}\langle s\rangle^{\frac{d(p-1)}{p}+a}\| \hbar(s)\|_{\dot{B}_{p,\infty}^a}\int_{t-|\alpha|}^t
\frac{ds}{(1+t-s)^{1-a}\langle s\rangle^{\frac{d(p-1)}{p}+a}}\nonumber\\
%\lesssim&\ \sup\limits_{0\leq s\leq t}\langle s\rangle^{\frac{d(p-1)}{p}+a}\| \hbar(s)\|_{\dot{B}_{p,\infty}^a}\langle t\rangle^{\frac{-d(p-1)}{p}-a}\int^{|\alpha|}_0
%\frac{1}{(1+\tau)^{1-a}}
%\,d\tau\\
\lesssim&\ \frac{|\alpha|^a}{\langle t\rangle^{\frac{d(p-1)}{p}+a}}\sup\limits_{0\leq s\leq t}\langle s\rangle^{\frac{d(p-1)}{p}+a}\|\hbar(s)\|_{\dot{B}_{p,\infty}^a}.
\end{align*}
For $\mathcal{A}_2$, we note that
\begin{align*}
\mathcal{A}_2\leq&\int_\frac{t}{2}^{t-|\alpha|}\||z|^a\delta_\alpha \mathcal{G}(t-s,z)\|_{L^1_z}\sup\limits_z\frac{\|\delta_z \hbar(s)\|_{L^p_x}}{|z|^a}\,ds\nonumber\\
\lesssim&\ \langle t\rangle^{-\frac{d(p-1)}{p}-a}\sup\limits_{0\leq s\leq t}\langle s\rangle^{\frac{d(p-1)}{p}+a}\| \hbar(s)\|_{\dot{B}_{p,\infty}^a}
\int_\frac{t}{2}^{t-|\alpha|}
\||z|^a\delta_\alpha \mathcal{G}(t-s,z)\|_{L^1_z}\,ds.
\end{align*}
Set \(\tau=t-s\), it follows from \eqref{lemma5.1-4} that
\begin{align*}
\int_\frac{t}{2}^{t-|\alpha|}
\||z|^a\delta_\alpha \mathcal{G}(t-s,z)\|_{L^1_z}\,ds
\lesssim&\
\int_{|\alpha|}^\frac{t}{2}
\frac{|\alpha|}{(1+\tau+|\alpha|)^{2-a}}d\tau
+\int_{|\alpha|}^\frac{t}{2}\frac{|\alpha|^{d+a}\mathbf{1}_{\{|\alpha|\leq \tau\}}}{(1+\tau)^{2}\tau^{d-1}}\,d\tau\lesssim|\alpha|^{a}.
\end{align*}
Therefore,
\begin{equation*}
\mathcal{A}_2\leq \frac{|\alpha|^a}{\langle t\rangle^{\frac{d(p-1)}{p}+a}}\sup_{0\leq s\leq t}\langle s\rangle^{\frac{d(p-1)}{p}+a}\|\hbar(s)\|_{\dot{B}_{p,\infty}^a}.
\end{equation*}

Combining the bounds for $\mathcal{A}_1$ and $\mathcal{A}_2$, we conclude that, for $|\alpha|<\frac{t}{2}$,
\begin{equation}\label{lemma5.2-G2-small}
\frac{\|G_2(t)\|_{L^p_x}}{|\alpha|^a}
\lesssim
\langle t\rangle^{-\frac{d(p-1)}p-a}
\sup_{0\le s\le t}
\langle s\rangle^{\frac{d(p-1)}p+a}
\|\hbar(s)\|_{\dot B^a_{p,\infty}}.
\end{equation}

Taking $p=1$ in \eqref{eq:Besov-I1}, \eqref{dashijian-0}, and \eqref{lemma5.2-G2-small}, we obtain
$$
\sup_{\alpha}
\frac{\|\delta_\alpha(\mathcal G*_{(t,x)}\hbar)(t)\|_{L^1_x}}{|\alpha|^a}
\lesssim \langle t\rangle^{-a}\sup\limits_{0\leq s\leq t}\langle s\rangle^{a}\|\hbar(s)\|_{\dot{B}_{1,\infty}^a},
$$
which gives
$$
\|(\mathcal{G}\ast_{(t,x)}\hbar)(t)\|_{\dot{B}_{1,\infty}^a}
\lesssim \langle t\rangle^{-a}\sup\limits_{0\leq s\leq t}\langle s\rangle^{a}\|\hbar(s)\|_{\dot{B}_{1,\infty}^a}.
$$
Taking $p=\infty$, we similarly deduce
$$
\sup_{\alpha}
\frac{\|\delta_\alpha(\mathcal G*_{(t,x)}\hbar)(t)\|_{L^\infty_x}}{|\alpha|^a}
\lesssim \langle t\rangle^{-d-a}\sup\limits_{0\leq s\leq t}
(\langle s\rangle^{a}\|\hbar(s)\|_{\dot{B}_{1,\infty}^a}+\langle s\rangle^{d+a}\|\hbar(s)\|_{\dot{B}_{\infty,\infty}^a}),
$$
and hence
$$
\|( \mathcal{G}\ast_{(t,x)}\hbar)(t)\|_{\dot{B}_{\infty,\infty}^a}
\lesssim \langle t\rangle^{-d-a}\sup\limits_{0\leq s\leq t}
(\langle s\rangle^{a}\|\hbar(s)\|_{\dot{B}_{1,\infty}^a}+\langle s\rangle^{d+a}\|\hbar(s)\|_{\dot{B}_{\infty,\infty}^a}).
$$
This completes the proof.
\end{proof}

\begin{lemma}\label{lemma5.3}
Let \(\mathcal F=\mathcal F(t,x)\) be a vector field, and let \(\eta=\eta(v)\) be a velocity profile satisfying
\begin{equation}\label{shuaijian28}
|\nabla_v\eta(v)|+|\nabla_v^2\eta(v)|
\lesssim \frac{1}{\langle v\rangle^N}.
\end{equation}
Define
\begin{align}\label{zuoyongxiang-08}
\Gamma^\pm[\mathcal{F},\eta](t,x)
:=&\mp\int_0^t\int_{\mathbb{R}^d}\mathcal{F}
\left(s,\mathcal{X}^{\pm}(s))\cdot\nabla_v\eta(\mathcal{V}^{\pm}(s)\right)\,dv\,ds\nonumber\\
&\pm\int_0^t\int_{\mathbb{R}^d}\mathcal{F}(s,x-(t-s)v)
\cdot\nabla_v\eta(v)\,dv\,ds.
\end{align}
Then
\begin{equation}\label{zuoyong08}
\|\Gamma^\pm[\mathcal{F},\eta](t)\|_{L^1_x}
\lesssim \varepsilon\int_0^t\frac{\|\mathcal{F}(s)\|_{L^1_x}}{\langle s\rangle^{d-\gamma-1}}ds,
\end{equation}
and
\begin{equation}\label{zuoyong09}
\|\Gamma^\pm[\mathcal{F},\eta](t)\|_{L^\infty_x}
\lesssim
\frac{\varepsilon}{\langle t\rangle^d}
\int_0^{\frac{t}{2}}
\frac{\|\mathcal{F}(s)\|_{L^1_x}}{\langle s\rangle^{d-\gamma-1}}
\,ds
+
\varepsilon
\int_{\frac{t}{2}}^t
\frac{\|\mathcal{F}(s)\|_{L^\infty_x}}{\langle s\rangle^{d-\gamma-1}}
\,ds.
\end{equation}
\end{lemma}

\begin{proof}
Using the change of variables \(v=\Psi_{s,t}^{\pm}(x,w)\) and Lemma~\ref{proposition of psi and phi}, we rewrite
\begin{align*}
&\Gamma^\pm[\mathcal{F},\eta](t,x)\\
=&\mp\int_0^t\int_{\mathbb{R}^{d}}\mathcal{F}(s,x-(t-s)w)\,
\nabla_v\eta\bigl(\mathcal{V}^{\pm}(s;t,x,\Psi_{s,t}^{\pm}(x,w))\bigr)
\bigl[\det(\nabla_w\Psi_{s,t}^{\pm}(x,w))-1\bigr]\,dw\,ds\\
&\mp\int_0^t\int_{\mathbb{R}^{d}}\mathcal{F}(s,x-(t-s)w)\,
\bigl[
\nabla_v\eta\bigl(\mathcal{V}^{\pm}(s;t,x,\Psi_{s,t}^{\pm}(x,w))\bigr)
-\nabla_v\eta(w)
\bigr]\,dw\,ds.
\end{align*}
For convenience, we relabel \(w\) as \(v\) below.
By \eqref{shuaijian28} and $\langle \mathcal{V}^{\pm}(s;t,x,\Psi_{s,t}^{\pm}(x,v)) \rangle\sim \langle v \rangle$, we have
\[
|\nabla_v
\eta(\mathcal{V}^{\pm}(s;t,x,\Psi_{s,t}^{\pm}(x,v)))|
\lesssim\frac{1}{\langle v \rangle^N}.
\]
Moreover, since \(\nabla_v\Psi_{s,t}^\pm=I+O(\varepsilon \langle s\rangle^{-(d-\gamma-1)})\), the smoothness of the determinant near \(I\) gives
\[
\bigl|\det(\nabla_v\Psi_{s,t}^{\pm}(x,v))-1\bigr|\lesssim
\|\nabla_v\Psi_{s,t}^{\pm}(x,v)-I
\|_{L^\infty_{x,v}}
\lesssim\frac{\varepsilon}{\langle s\rangle^{d-\gamma-1}}.
\]
Next, by the mean value theorem, \eqref{shuaijian28}, \eqref{fangchonh0890}, \eqref{z4} and \eqref{tezhengW},
\begin{align*}
&|\nabla_v\eta\bigl(\mathcal{V}^{\pm}(s;t,x,\Psi_{s,t}^{\pm}(x,v))\bigr)
-\nabla_v\eta(v)|\\
%\lesssim&\frac{1}{\langle v\rangle^{N}}|\mathcal{V}^{\pm}(s;t,x,\Psi^{\pm}_{s,t}(x,v))-v|\\
\leq&\frac{1}{\langle v\rangle^{N}}\Big(|\mathcal{V}^{\pm}(s;t,x,
\Psi_{s,t}^{\pm}(x,v))-\mathcal{V}^{\pm}(s)|+|\mathcal{V}^{\pm}(s)-v|\Big)\\
\leq&\frac{1}{\langle v\rangle^{N}}\Big(\|\nabla_v\mathcal{V}^{\pm}(s)\|_{L_{x,v}^{\infty}}
|v-\Psi_{s,t}^{\pm}(x,v)|+\|W_{s,t}^{\pm}\|_{L_{x, v}^{\infty}}\Big)\\
\lesssim&\frac{1}{\langle v \rangle^N}\frac{\varepsilon}{\langle s\rangle^{d-\gamma}}.
\end{align*}
Combining the above bounds, we obtain
$$
|\Gamma^\pm[\mathcal{F},\eta](t,x)|
\lesssim\varepsilon\int_0^t\frac{1}{\langle s\rangle^{d-\gamma-1}}
\int_{\mathbb{R}^{d}}
\frac{|\mathcal{F}(s,x-(t-s)v)|}{\langle v\rangle^N}\,dv\,ds.
$$

By Fubini's theorem and the change of variables $y=x-(t-s)v$,
\begin{equation*}
\|\Gamma^\pm[\mathcal{F},\eta](t)\|_{L^1_x}
\lesssim\varepsilon\int_0^t\frac{1}{\langle s\rangle^{d-\gamma-1}}
\int_{\mathbb{R}^{d}}|\mathcal{F}(s,y)|dy\,\int_{\mathbb{R}^{d}}
\frac{\,dv\,ds}{\langle v\rangle^N}
\lesssim \varepsilon\int_0^t\frac{\|\mathcal{F}(s)\|_{L^1_x}}{\langle s\rangle^{d-\gamma-1}}ds.
\end{equation*}

To estimate the \(L^\infty_x\) norm, we split the time integral into \([0,\frac{t}{2}]\) and \([\frac{t}{2},t]\).

For \(0\le s\le \frac{t}{2}\), using the change of variables \(w=x-(t-s)v\), we get
\begin{align*}
\int_0^{\frac{t}{2}}
\frac{1}{\langle s\rangle^{d-\gamma-1}}
\int_{\mathbb{R}^{d}}
\frac{|\mathcal{F}(s,x-(t-s)v)|}{\langle v\rangle^N}\,dv\,ds
=&\int_0^{\frac{t}{2}}
\frac{1}{\langle s\rangle^{d-\gamma-1}}
\int_{\mathbb{R}^{d}}
|\mathcal{F}(s,w)|
\Bigl\langle \frac{x-w}{t-s}\Bigr\rangle^{-N}dw\,
\frac{ds}{(t-s)^d}\\
\lesssim&\frac{\varepsilon}{\langle t\rangle^d}
\int_0^{\frac{t}{2}}
\frac{\|\mathcal{F}(s)\|_{L^1_x}}{\langle s\rangle^{d-\gamma-1}}
\,ds.
\end{align*}
For \(\frac{t}{2}\le s\le t\), we use the \(L^\infty_x\) bound on \(\mathcal F\) and deduce
$$
\int_{\frac{t}{2}}^t\frac{1}{\langle s\rangle^{d-\gamma-1}}
\int_{\mathbb R^d}\frac{|\mathcal F(s,x-(t-s)v)|}{\langle v\rangle^N}\,dv\,ds
\lesssim
\int_{\frac{t}{2}}^t\frac{\|\mathcal F(s)\|_{L^\infty_x}}{\langle s\rangle^{d-\gamma-1}}
\,ds,
$$
since $\int_{\mathbb R^d}\langle v\rangle^{-N}\,dv<\infty$.

Combining the above bounds, we obtain \eqref{zuoyong09}.
This completes the proof.
\end{proof}

\begin{lemma}\label{lemma:J123-H}
Let $\mathcal H_i=\mathcal H_i(x,v)$ $(i=1,2,3)$ be measurable functions. Define
\begin{align*}
\mathcal{L}_1(t,x)&:=\int_0^1\int_{\mathbb{R}^{d}}
\big|\mathcal{H}_1\big(\mathcal{X}^+(0),
\Xi_{\mathcal{V},\theta}(0)\big)\big|\,dv\,d\theta,\\
\mathcal{L}_2(t,x)&:=\int_0^1\int_{\mathbb{R}^{d}}
\big|\mathcal{H}_2\big(
\Xi_{\mathcal{X},\theta}(0),\mathcal{V}^-(0)\big)\big|\,dv\,d\theta,\\
\mathcal{L}_3(t,x)&:=\int_{\mathbb{R}^{d}}
\big|\mathcal{H}_3
\big(\mathcal{X}^-(0),\mathcal{V}^-(0)\big)\big|\,dv,
\end{align*}
where
\begin{equation}\label{tezheng-V}
\Xi_{\mathcal{X},\theta}(0) := \theta \mathcal{X}^+(0)
+ (1-\theta)\mathcal{X}^-(0), \quad \Xi_{\mathcal{V},\theta}(0) := \theta \mathcal{V}^+(0) + (1-\theta)\mathcal{V}^-(0),   \quad \theta\in[0,1].
\end{equation}
Then for any $k>d$ and all $t>0$, the following estimates hold:

\medskip
\noindent
(i) \(L^1_x\) estimates:
\[
\|\mathcal{L}_1(t)\|_{L^1_x}
\lesssim
\|\langle v\rangle^k \mathcal H_1\|_{L^1_xL^\infty_v},
\quad
\|\mathcal{L}_2(t)\|_{L^1_x}
\lesssim
\|\langle v\rangle^k \mathcal H_2\|_{L^1_xL^\infty_v},
\quad
\|\mathcal{L}_3(t)\|_{L^1_x}
\lesssim
\|\mathcal H_3\|_{L^1_{x}L^1_{v}}.
\]

\medskip
\noindent
(ii) \(L^\infty_x\) estimates:
$$
\|\mathcal{L}_1(t)\|_{L^\infty_x}
\lesssim
t^{-d}\|\mathcal H_1\|_{L^1_xL^\infty_v},
\quad
\|\mathcal{L}_2(t)\|_{L^\infty_x}
\lesssim
t^{-d}\|\mathcal H_2\|_{L^1_xL^\infty_v},
\quad
\|\mathcal{L}_3(t)\|_{L^\infty_x}
\lesssim
t^{-d}\|\mathcal H_3\|_{L^1_xL^\infty_v}.
$$
\end{lemma}
\begin{proof}

\medskip
\noindent
\textit{$L^1_x$ estimates.}
Integrating in $x$ and using Fubini's theorem, we obtain
$$
\|\mathcal{L}_1(t)\|_{L^1_x}
\le
\int_0^1\int_{\mathbb{R}^d}\int_{\mathbb{R}^d}
\big|\mathcal H_1\big(\mathcal X^+(0),\Xi_{\mathcal V,\theta}(0)\big)\big|
\,dx\,dv\,d\theta.
$$
For $v\in \mathbb{R}^d$, the change of variables $y=\mathcal X^+(0)$ yields
$$
\|\mathcal{L}_1(t)\|_{L^1_x}
\lesssim
\int_0^1\int_{\mathbb{R}^d}\int_{\mathbb{R}^d}
\big|\mathcal H_1\big(y,\Xi_{\mathcal V,\theta}(0;t,x(y),v)\big)\big|\big|\det\big(\nabla_x \mathcal{X}^+(0)\big)\big|^{-1}
\,dy\,dv\,d\theta.
$$
Since $\langle \Xi_{\mathcal V,\theta}(0)\rangle\sim \langle v\rangle$, we have
\[
\big|\mathcal H_1(y,\Xi_{\mathcal V,\theta}(0))\big|
\le
\|\langle v\rangle^k \mathcal H_1(y,v)\|_{L^\infty_v}
\langle \Xi_{\mathcal V,\theta}(0)\rangle^{-k}.
\]
Therefore, using \(k>d\) and \(|\det(\nabla_x\mathcal X^+(0))|^{-1}\lesssim 1\), we infer
$$
\|\mathcal{L}_1(t)\|_{L^1_x}
\lesssim
\|\langle v\rangle^k \mathcal H_1\|_{L^1_xL^\infty_v}.
$$

The estimate for \(\mathcal L_2\) follows in the same way. Indeed, for each fixed \((v,\theta)\in\mathbb R^d\times[0,1]\), the map
$x\mapsto \Xi_{\mathcal X,\theta}(0)$
is a global \(\mathcal C^1\) diffeomorphism. Thus, changing variables
$z=\Xi_{\mathcal X,\theta}(0)$,
and using \(|\det(\nabla_x\Xi_{\mathcal X,\theta}(0))|^{-1}\lesssim 1\) together with \(\langle \mathcal V^-(0)\rangle\sim \langle v\rangle\), we obtain
\[
\|\mathcal L_2(t)\|_{L^1_x}
\lesssim
\|\langle v\rangle^k\mathcal H_2\|_{L^1_xL^\infty_v}.
\]

For $\mathcal{L}_3$, using the measure-preserving property of the full
characteristic flow, we obtain
\[
\|\mathcal{L}_3(t)\|_{L^1_x}
\le
\int_{\mathbb{R}^d}\int_{\mathbb{R}^d}
\big|\mathcal H_3\big(\mathcal X^-(0),\mathcal V^-(0)\big)\big|
\,dx\,dv
=
\|\mathcal H_3\|_{L^1_{x}L^1_{v}}.
\]

\textit{$L^\infty_x$ estimates.}
For each fixed $x\in\mathbb R^d$, the map
$v\mapsto \mathcal X^+(0)$
is a $\mathcal C^1$ diffeomorphism by Lemma~\ref{Lemma3.4-YKB}. Therefore, making the change of variables
$y=\mathcal X^+(0)$,
we obtain
\[
\mathcal L_1(t,x)
=\int_0^1\int_{\mathbb R^d}\big|\mathcal H_1\big(
y,\Xi_{\mathcal V,\theta}(0;t,x,v(y))\big)\big|\big|\det(\nabla_v\mathcal X^+(0))
\big|^{-1}\,dy\,d\theta.
\]
Using $\big|\det(\nabla_v\mathcal X^+(0))\big|^{-1}\lesssim t^{-d}$,
we infer
\[
\mathcal L_1(t,x)
\lesssim
t^{-d}
\int_0^1\int_{\mathbb R^d}
\big|\mathcal H_1(y,\Xi_{\mathcal V,\theta}(0;t,x,v(y)))
\big|
\,dy\,d\theta.
\]
Therefore,
\[
\|\mathcal L_1(t)\|_{L^\infty_x}
\lesssim
t^{-d}\|\mathcal H_1\|_{L^1_xL^\infty_v}.
\]

Similarly, by the change of variables of $z=\Xi_{\mathcal X,\theta}(0)$,
 one has that
\[
\mathcal L_2(t,x)
=
\int_0^1\int_{\mathbb R^d}
\big|\mathcal H_2\big(z,\mathcal V^-(0;t,x,v(z))\big)\big|
\Big|\det\big(\nabla_v\Xi_{\mathcal X,\theta}(0;t,x,v(z))\big)\Big|^{-1}
\,dz\,d\theta.
\]
Using
\[
\Big|\det\big(\nabla_v\Xi_{\mathcal X,\theta}(0;t,x,v(z))\big)\Big|^{-1}\lesssim t^{-d},
\]
we infer
\[
\|\mathcal L_2(t)\|_{L^\infty_x}
\lesssim
t^{-d}\|\mathcal H_2\|_{L^1_xL^\infty_v}.
\]

Finally, by the same argument as above, using that for each fixed $x$ the map
$v\mapsto \mathcal X^-(0)$ is a $\mathcal C^1$ diffeomorphism and
$\big|\det(\nabla_v\mathcal X^-(0))\big|^{-1}\lesssim t^{-d}$, we obtain
\[
\|\mathcal L_3(t)\|_{L^\infty_x}
\lesssim
t^{-d}\|\mathcal H_3\|_{L^1_xL^\infty_v}.
\]
Combining the above estimates, we conclude the proof.
\end{proof}

\subsection{Besov Estimates for the Charge Density}
\begin{proposition}\label{Besovestimates}
For $t\geq1$ and $a \in \big(\frac{\gamma}{d-\gamma-2}, 1 \big)$, one has
\begin{equation}\label{miduguji-04}
\|\rho(t)\|_{L^1_x}+\langle t\rangle^{d}\|\rho(t)\|_{L^\infty_x}
+\langle t\rangle^{a}\|\rho(t)\|_{\dot{B}_{1,\infty}^a}+\langle t\rangle^{d+a}\|\rho(t)\|_{\dot{B}_{\infty,\infty}^a}
\lesssim\varepsilon_0+\varepsilon.
\end{equation}
\end{proposition}

\begin{proof}
The proof is based on decay and Besov estimates for the terms
\[
I:=I^+-I^-,
\qquad
R:=R^++R^-,
\]
in the decomposition \eqref{midufenjie001}.

\medskip
\textit{Step 1. Estimates for $I$.}
Applying the mean value theorem, we may rewrite $I(t,x)$ as
\begin{align}\label{chuzhiI-00}
I(t,x)
=&\int_0^1\int_{\mathbb{R}^{d}}
\nabla_vf_{\mathrm{in}}^+\left(\mathcal{X}^+(0),
\Xi_{\mathcal{V},\theta}(0)\right)\left(\mathcal{V}^+(0)-\mathcal{V}^-(0)\right)\,dv\,d\theta\nonumber\\
&+\int_0^1\int_{\mathbb{R}^{d}}
\nabla_xf_{\mathrm{in}}^+\left(
\Xi_{\mathcal{X},\theta}(0),\mathcal{V}^-(0)\right)
\left(\mathcal{X}^+(0)-\mathcal{X}^-(0)\right)\,dv\,d\theta\nonumber\\
&+\int_{\mathbb{R}^{d}}f_{\mathrm{in}}
\left(\mathcal{X}^-(0),\mathcal{V}^-(0)\right)\,dv,
\end{align}
where $\Xi_{\mathcal{X},\theta}(0)$ and $\Xi_{\mathcal{V},\theta}(0)$ are defined in \eqref{tezheng-V}.
By \eqref{cor:estimates001-difference},
\begin{align*}
|I(t,x)|
\lesssim&\ \varepsilon\int_0^1\int_{\mathbb{R}^{d}}
\left|\nabla_vf_{\mathrm{in}}^+\left(\mathcal{X}^+(0),
\Xi_{\mathcal{V},\theta}(0)\right)\right|\,dv\,d\theta\\
&+\varepsilon\int_0^1\int_{\mathbb{R}^{d}}
\left|\nabla_xf_{\mathrm{in}}^+\left(
\Xi_{\mathcal{X},\theta}(0),\mathcal{V}^-(0)\right)\right|\,dv\,d\theta\\
&+\int_{\mathbb{R}^{d}}\left|f_{\mathrm{in}}
\left(\mathcal{X}^-(0),\mathcal{V}^-(0)\right)\right|\,dv.
\end{align*}
Applying Lemma~\ref{lemma:J123-H} with
\[
\mathcal H_1=\nabla_v f_{\mathrm{in}}^+,\qquad
\mathcal H_2=\nabla_x f_{\mathrm{in}}^+,\qquad
\mathcal H_3=f_{\mathrm{in}},
\]
we infer that
\begin{equation}\label{chuzhi0082}
\|I(t)\|_{L^1_x}\lesssim \varepsilon\big(\|\langle v\rangle^k\nabla_vf_{\mathrm{in}}^+\|_{L_{x}^{1}L_{v}^{\infty}}
+\|\langle v\rangle^k\nabla_xf_{\mathrm{in}}^+\|_{L_{x}^{1}L_{v}^{\infty}}\big)
+\|f_{\mathrm{in}}\|_{L_{x}^{1}L_{v}^{1}}\lesssim\varepsilon+\varepsilon_0,
\end{equation}
and
\begin{equation}\label{chuzhi008}
\|I(t)\|_{L^\infty_x}\lesssim \frac{1}{\langle t\rangle^{d}}\left(\varepsilon\|\nabla_vf_{\mathrm{in}}^+\|_{L_{x}^{1}L_{v}^{\infty}}
+\varepsilon\|\nabla_xf_{\mathrm{in}}^+\|_{L_{x}^{1}L_{v}^{\infty}}
+\|f_{\mathrm{in}}\|_{L_{x}^{1}L_{v}^{\infty}}\right)
\lesssim\frac{1}{\langle t\rangle^{d}}(\varepsilon+\varepsilon_0),
\end{equation}
where in the last step we used \eqref{large-data} and \eqref{qn}.

We next establish the Besov estimates of $I$.
Substituting \eqref{z4} into \eqref{chuzhiI-00} and using the change of variables \(w=x-tv\), we obtain
\begin{equation}\label{chuzhi0123}
I(t,x)=I_1(t,x)+I_2(t,x)+I_3(t,x),
\end{equation}
where
\begin{align*}
I_1(t,x)
:=&\frac{1}{t^{d}}\int_0^1\!\int_{\mathbb{R}^d}
\nabla_v f^{+}_{\mathrm{in}}
\Big(w+Y^{+}_{0,t}(w,\frac{x-w}{t}),\ \frac{x-w}{t}
+\Big(\theta W^{+}_{0,t}+(1-\theta)W^{-}_{0,t}\Big)(w,\frac{x-w}{t})\Big)\\
&\quad \cdot
\Big(W^{+}_{0,t}-W^{-}_{0,t}\Big)(w,\frac{x-w}{t})
\,  dw\,d\theta,
\end{align*}
\begin{align*}
I_2(t,x)
:=&\frac{1}{t^{d}}\int_0^1\!\int_{\mathbb{R}^d}
\nabla_x f^{+}_{\mathrm{in}}
\Big(w+\Big(\theta Y^{+}_{0,t}
+(1-\theta)Y^{-}_{0,t}\Big)(w,\frac{x-w}{t}), \  \frac{x-w}{t}
+W^{-}_{0,t}(w,\frac{x-w}{t})\Big) \ \\
&\quad \cdot
\Big(Y^{+}_{0,t}-Y^{-}_{0,t}\Big)(w,\frac{x-w}{t})
\, dw\,d\theta,
\end{align*}
$$
I_3(t,x)
=\frac{1}{t^{d}}\int_{\mathbb{R}^d}
f_{\mathrm{in}}\Big(w+Y^{-}_{0,t}(w,\frac{x-w}{t}), \
 \frac{x-w}{t}+W^{-}_{0,t}(w,\frac{x-w}{t})\Big)\,dw.
$$

\medskip
\textit{Estimate for $I_1$.}
Applying $\delta_\alpha^x$ to $I_1$, we write
$$
\delta_\a^xI_1(t,x)=I_{11}(t,x)+I_{12}(t,x),
$$
where
\begin{align*}
I_{11}(t,x):=&\frac{1}{t^d}\int_0^1\int_{\mathbb{R}^d}\delta_\a^x\nabla_v
f^+_{\mathrm{in}}\Big(w+Y^+_{0,t}(w,\frac{x-w}{t}),\ \frac{x-w}{t}
+\Big(\t W^+_{0,t}+(1-\t)W_{0,t}^-\Big)(w,\frac{x-w}{t})\Big)\\
&\quad \ \Big(W^+_{0,t}-W_{0,t}^-\Big)(w,\frac{x-w}{t})\,dw\, d\theta,
\end{align*}
\begin{align*}
I_{12}(t,x):=&\frac{1}{t^d}\int_0^1\int_{\mathbb{R}^d}\nabla_vf^+_{\mathrm{in}}
\Big(w+Y^+_{0,t}(w,\frac{x-w-\alpha}{t}),\ \frac{x-w-\alpha}{t} \\
&\quad  +\left(\t W^+_{0,t}+(1-\t)W_{0,t}^-\right)(w, \frac{x-w-\alpha}{t})\Big)\ \delta_\a^x\Big(W^+_{0,t}
-W_{0,t}^-\Big)(w,\frac{x-w}{t})\,dw\, d\theta.
\end{align*}
For $I_{11}$,
by inserting an intermediate term,
\begin{align*}
&|I_{11}(t,x)|\\
\lesssim&\int_0^1\int_{\mathbb{R}^d}\frac{\big|\delta_{z_1}
\nabla_vf^+_{\mathrm{in}}\Big(\cdot, \  \frac{x-w}{t}
+\Big(\t W^+_{0,t}+(1-\t)W_{0,t}^-\Big)(w,\frac{x-w}{t})\Big)
\big(w+Y^+_{0,t}(w,\frac{x-w}{t})\big)\big|}{|z_1|^a}\frac{|\a|^a}{t^{d+a}}\,dw\,d\theta\\
&+\int_0^1\int_{\mathbb{R}^d}\frac{\big|\delta_{z_2}\nabla_vf^+_{\mathrm{in}}\Big(w+Y^+_{0,t}(w,
\frac{x-w-\alpha}{t}),\cdot\Big)\Big(\frac{x-w}{t}
+\Big(\t W^+_{0,t}+(1-\t)W_{0,t}^-\Big)(w,\frac{x-w}{t})\Big)\big|
}{|z_2|^a} \frac{|\a|^a}{t^{d+a}}\,dw\,d\theta,
\end{align*}
where
\begin{align*}
z_1 &:= Y^+_{0,t}(w,\tfrac{x-w}{t}) - Y^+_{0,t}(w,\tfrac{x-w-\alpha}{t}), \\
z_2 &:= \theta \big[W^+_{0,t}(w,\tfrac{x-w}{t}) - W^+_{0,t}(w,\tfrac{x-w-\alpha}{t})\big] + (1-\theta)\big[W_{0,t}^-(w,\tfrac{x-w}{t}) - W_{0,t}^-(w,\tfrac{x-w-\alpha}{t})\big] + \tfrac{\alpha}{t}.
\end{align*}
From Lemma~\ref{estimates001}, one has that
\[
|z_1|+|z_2|
\lesssim\frac{|\a|}{t}.
\]
Using the change of variables $w=x-tv$ and \eqref{z4}, we get
\begin{align*}
\frac{|I_{11}(t,x)|}{|\a|^a}
\lesssim&\ \frac{1}{t^a}\int_0^1\int_{\mathbb{R}^d}\frac{\big|
\delta_{z_1}\nabla_vf^+_{\mathrm{in}}\Big(\cdot,\ \Xi_{\mathcal{V},\theta}(0)\Big)\big(\mathcal{X}^+(0)\big)\big|
}{|z_1|^a}\,dv\,d\theta\nonumber\\
&+\frac{1}{t^a}\int_0^1\int_{\mathbb{R}^d}\frac{\big|\delta_{z_2}
\nabla_vf^+_{\mathrm{in}}\Big(\mathcal{X}^+(0;t,x-\alpha,v-\frac{\alpha}{t}),\ \cdot\Big)\big(\Xi_{\mathcal{V},\theta}(0)\big)\big|}{|z_2|^a}\,dv\,d\theta.
\end{align*}
By \eqref{cor:estimates001-difference} and \eqref{Dhdef},
\begin{align}\label{chuzhi-I11}
\frac{|I_{11}(t,x)|}{|\a|^a}
\lesssim&\ \frac{\v}{\langle t\rangle^a}\int_0^1\int_{\mathbb{R}^d}\mathcal{D}_1^a\nabla_v
f^+_{\mathrm{in}}\big(\mathcal{X}^+(0),\ \Xi_{\mathcal{V},\theta}(0)\big)\,dv\, d\theta \nonumber\\
&+\frac{\v}{\langle t\rangle^a}\int_0^1\int_{\mathbb{R}^d}\mathcal{D}_2^a
\nabla_vf^+_{\mathrm{in}}\Big(\mathcal{X}^+(0;t,x-\alpha,v-\frac{\a}{t}), \
\big(\Xi_{\mathcal{V},\theta}(0)\big)\Big)\,dv \,d\theta.
\end{align}

For $I_{12}$, by the change of variables $w=x-tv$ and \eqref{z4},
$$
|I_{12}(t,x)|\lesssim\int_0^1\int_{\mathbb{R}^d}|\nabla_vf^+_{\mathrm{in}}
(\mathcal{X}^+(0),\Xi_{\mathcal{V},\theta}(0))|\frac{\Big|\delta_{\frac{\a}{t}}
\Big(\big(W^+_{0,t}-W^-_{0,t}\big)(x-tv, \cdot)(v)\Big)\Big|}{|\frac{\a}{t}|^a}|\frac{\a}{t}|^a\,dv\,d\theta.
$$
It follows from~\eqref{cor:estimates001-difference} that
\begin{equation}\label{chuzhi-I12}
\frac{|I_{12}(t,x)|}{|\a|^a}\lesssim\frac{\varepsilon}{\langle t\rangle^a}\int_0^1\int_{\mathbb{R}^d}
|\nabla_vf^+_{\mathrm{in}}\left(\mathcal{X}^+(0),\Xi_{\mathcal{V},\theta}(0)\right)|\,dv\, d\theta.
\end{equation}

Combining \eqref{chuzhi-I11} and \eqref{chuzhi-I12}, we have
\begin{align*}
\frac{|\delta_\a^xI_1(t,x)|}{|\a|^a}
\lesssim&\ \frac{\v}{\langle t\rangle^a}\int_0^1\int_{\mathbb{R}^d}\mathcal{D}_1^a\nabla_v
f^+_{\mathrm{in}}\big(\mathcal{X}^+(0),\ \Xi_{\mathcal{V},\theta}(0)\big)\,dv\, d\theta \nonumber\\
&+\frac{\v}{\langle t\rangle^a}\int_0^1\int_{\mathbb{R}^d}\mathcal{D}_2^a
\nabla_vf^+_{\mathrm{in}}\Big(\mathcal{X}^+(0;t,x-\alpha,v-\frac{\a}{t}), \
\Xi_{\mathcal{V},\theta}(0)\Big)\,dv \,d\theta\\
&+\frac{\varepsilon}{\langle t\rangle^a}\int_0^1\int_{\mathbb{R}^d}
\left|\nabla_vf^+_{\mathrm{in}}\left(\mathcal{X}^+(0),\Xi_{\mathcal{V},\theta}(0)\right)\right|\,dv\, d\theta.
\end{align*}
Taking the supremum over $\alpha$ and applying Lemma~\ref{lemma:J123-H} to the three terms above, each viewed as an instance of $\mathcal L_2$, with
\[
\mathcal H=\mathcal D_1^a\nabla_v f^+_{\mathrm{in}},\qquad
\mathcal D_2^a\nabla_v f^+_{\mathrm{in}},\qquad
\nabla_v f^+_{\mathrm{in}},
\]
we obtain
\begin{equation}\label{chuzhi-I01}
\|I_1(t)\|_{\dot{B}_{1,\infty}^a}
\lesssim\frac{\v}{\langle t\rangle^{a}}\left(\|\langle v\rangle^k\mathcal{D}_1^a\nabla_vf^+_{\mathrm{in}}\|_{L_x^1L_v^\infty}
+\|\langle v\rangle^k\mathcal{D}_2^a\nabla_vf^+_{\mathrm{in}}\|_{L_x^1L_v^\infty}
+\|\langle v\rangle^k\nabla_vf^+_{\mathrm{in}}\|_{L_x^1L_v^\infty}\right).
\end{equation}
The second term is handled by the same change of variables, justified by Lemma~\ref{Lemma3.4-YKB}, together with the comparability of the corresponding velocity variable and $v$. Similarly,
\begin{equation}\label{chuzhi-I02}
\|I_1(t)\|_{\dot{B}_{\infty,\infty}^a}
\lesssim\frac{\v}{\langle t\rangle^{d+a}}\Big(\|\mathcal{D}_1^a\nabla_vf^+_{\mathrm{in}}\|_{L_x^1L_v^\infty}
+\|\mathcal{D}_2^a\nabla_vf^+_{\mathrm{in}}\|_{L_x^1L_v^\infty}
+\|\nabla_vf^+_{\mathrm{in}}\|_{L_x^1L_v^\infty}\Big).
\end{equation}

\textit{Estimate for $I_2$.}
Applying $\delta_{\alpha}^{x}$ to $I_2$ gives
$$
\delta_\a^xI_2(t,x)=I_{21}(t,x)+I_{22}(t,x),
$$
where
\begin{align*}
I_{21}(t,x):=&\frac{1}{t^d}\int_0^1\int_{\mathbb{R}^d}\delta_\a^x\nabla_x
f^+_{\mathrm{in}}\Big(w+\big(\t Y^+_{0,t}+(1-\t)Y^-_{0,t}\big)(w,\frac{x-w}{t}),\ \frac{x-w}{t}+W_{0,t}^-(w,\frac{x-w}{t})\Big) \\
&\quad \cdot
\Big(Y_{0,t}^+-Y^-_{0,t}\Big)(w,\frac{x-w}{t})\,dw\,d\theta,
\end{align*}
\begin{align*}
I_{22}(t,x):=
&\frac{1}{t^d}\int_0^1\int_{\mathbb{R}^d}\nabla_xf^+_{\mathrm{in}}\Big(w+\big(\t Y^+_{0,t}+(1-\t)Y^-_{0,t}\big)(w,\frac{x-w}{t}),\
\frac{x-w}{t}+W_{0,t}^-(w,\frac{x-w}{t})\Big) \\
&\quad \cdot
\delta_\a^x\Big(Y_{0,t}^+-Y^-_{0,t}\Big)(w,\frac{x-w}{t})
\,dw\,d\theta.
\end{align*}

By \eqref{cor:estimates001-difference} and insert an intermediate term, we obtain
\begin{align*}
&|I_{21}(t,x)|\\
\lesssim
&\frac{|\a|^a}{t^{d+a}}\int_0^1\int_{\mathbb{R}^d}\frac{\delta_{z_3}\nabla_x
f^+_{\mathrm{in}}\Big(w+\Big(\t Y^+_{0,t}+(1-\t)Y^-_{0,t}\Big)(w,\frac{x-w}{t}),\  \cdot\Big)
\big(\frac{x-w}{t}+W_{0,t}^-(w,\frac{x-w}{t})
\big)}{|z_3|^a}
\,dw\,d\theta\\
&+\frac{|\a|^a}{t^{d+a}}\int_0^1\int_{\mathbb{R}^d}
\frac{\delta_{z_4}\nabla_xf^+_{\mathrm{in}}\Big(\cdot, \frac{x-w-\alpha}{t}+W_{0,t}^-
(w,\frac{x-w-\alpha}{t})\Big)\Big(w+\Big(\t Y^+_{0,t}+(1-\t)Y^-_{0,t}\Big)(w,\frac{x-w}{t})
\Big)
}{|z_4|^a}\,dw\,d\theta,
\end{align*}
where
\begin{align*}
z_3 &:= \tfrac{\alpha}{t} + W_{0,t}^-(w,\tfrac{x-w}{t}) - W_{0,t}^-(w,\tfrac{x-w-\alpha}{t}), \\
z_4 &:= \theta\big[Y^+_{0,t}(w,\tfrac{x-w}{t})-Y^+_{0,t}(w,\tfrac{x-w-\alpha}{t})\big] + (1-\theta)\big[Y^-_{0,t}(w,\tfrac{x-w}{t})-Y^-_{0,t}(w,\tfrac{x-w-\alpha}{t})\big].
\end{align*}
By Lemma~\ref{estimates001},
\[
|z_3|+|z_4|\lesssim\frac{|\a|}{t}.
\]
Using the change of variables $w=x-tv$ and \eqref{z4}, we have
\begin{align*}
&|I_{21}(t,x)|\\
\lesssim
&\frac{|\a|^a}{t^{a}}\int_0^1\int_{\mathbb{R}^d}\frac{\Big|\delta_{z_3}
\nabla_xf^+_{\mathrm{in}}\Big(\Xi_{\mathcal{X},\theta}(0),\
\cdot\Big)\big(\mathcal{V}^-(0)\big)\Big|}{|z_3|^a}
\,dv\,d\theta\\
&+\frac{|\a|^a}{t^{a}}\int_0^1\int_{\mathbb{R}^d}
\frac{\Big|\delta_{z_4}\nabla_xf^+_{\mathrm{in}}
\Big(\cdot, \  \mathcal{V}^-(0;t,x-\alpha,v-\frac{\a}{t})\Big)\big(\t \mathcal{X}^+(0)+(1-\t)\mathcal{X}^-(0)\big)\Big|}{|z_4|^a}\,dv\,d\theta.
\end{align*}
This together with \eqref{cor:estimates001-difference} and \eqref{Dhdef} implies
\begin{align*}
\frac{|I_{21}(t,x)|}{|\a|^a}
\lesssim&\ \frac{\v}{\langle t\rangle^a}\int_0^1\int_{\mathbb{R}^d}
\mathcal{D}_2^a
\nabla_xf^+_{\mathrm{in}}\Big(\Xi_{\mathcal{X},\theta}(0),\mathcal{V}^-(0)\Big)\,dv\, d\theta \\
&+ \frac{\v}{\langle t\rangle^a}\int_0^1\int_{\mathbb{R}^d}
\mathcal{D}_1^a\nabla_xf^+_{\mathrm{in}}\Big(\Xi_{\mathcal{X},\theta}(0),\mathcal{V}^-(0;t,x-\alpha,v-\frac{\a}{t})\Big)\,dv \,d\theta.
\end{align*}
Applying the change of variables $w=x-tv$ and \eqref{z4}, we have
$$
|I_{22}(t,x)|\lesssim\frac{|\a|^a}{t^a}\int_0^1\int_{\mathbb{R}^d}\left|\nabla_xf^+_{\mathrm{in}}
\Big(\Xi_{\mathcal{X},\theta}(0),\mathcal{V}^-(0)\Big)\right|\frac{\Big|\delta_{\frac{\a}{t}}
\left(\big(Y^+_{0,t}-Y^-_{0,t}\big)(x-tv, \cdot)(v)\right)\Big|}
{|\frac{\a}{t}|^a}\,dv\,d\theta.
$$
It follows from ~\eqref{cor:estimates001-difference} that
$$
\frac{|I_{22}(t,x)|}{|\a|^a}\lesssim\frac{\varepsilon}{\langle t\rangle^a}\int_0^1\int_{\mathbb{R}^d}\big|\nabla_xf^+_{\mathrm{in}}
\big(\Xi_{\mathcal{X},\theta}(0),\mathcal{V}^-(0)\big)\big|\,dv \,d\theta.
$$
Hence,
\begin{align*}
\frac{|\delta_\a^xI_2(t,x)|}{|\a|^a}
\lesssim&\ \frac{\v}{\langle t\rangle^a}\int_0^1\int_{\mathbb{R}^d}
\mathcal{D}_2^a
\nabla_xf^+_{\mathrm{in}}\Big(\Xi_{\mathcal{X},\theta}(0),\mathcal{V}^-(0)\Big)\,dv\, d\theta \\
&+ \frac{\v}{\langle t\rangle^a}\int_0^1\int_{\mathbb{R}^d}
\mathcal{D}_1^a\nabla_xf^+_{\mathrm{in}}\Big(\Xi_{\mathcal{X},\theta}(0),
\mathcal{V}^-(0;t,x-\alpha,v-\frac{\a}{t})\Big)\,dv \,d\theta\\
&+\frac{\varepsilon}{\langle t\rangle^a}\int_0^1\int_{\mathbb{R}^d}\big|\nabla_xf^+_{\mathrm{in}}
\big(\Xi_{\mathcal{X},\theta}(0),\mathcal{V}^-(0)\big)\big|\,dv \,d\theta.
\end{align*}
Taking the supremum over $\alpha$ and applying Lemma~\ref{lemma:J123-H} with $\mathcal L_2$
to the three terms above, with
\[
\mathcal H=\mathcal D_2^a\nabla_xf^+_{\mathrm{in}},\qquad
\mathcal D_1^a\nabla_xf^+_{\mathrm{in}},\qquad
\nabla_xf^+_{\mathrm{in}},
\]
we obtain
\begin{equation}\label{chuzhi-I03}
\|I_2(t)\|_{\dot{B}_{1,\infty}^a}
\lesssim\frac{\v}{\langle t\rangle^{a}}\Big(\|\langle v\rangle^k\mathcal{D}_1^a\nabla_xf^+_{\mathrm{in}}\|_{L_x^1L_v^\infty}
+\|\langle v\rangle^k\mathcal{D}_2^a\nabla_xf^+_{\mathrm{in}}\|_{L_x^1L_v^\infty}
+\|\langle v\rangle^k\nabla_xf^+_{\mathrm{in}}\|_{L_x^1L_v^\infty}\Big).
\end{equation}
For the second term, we use
\[
\big\langle \mathcal V^-\big(0;t,x-\alpha,v-\tfrac{\alpha}{t}\big)\big\rangle
\sim \langle v\rangle.
\]
Similarly, we obtain
\begin{equation}\label{chuzhi-I04}
\|I_2(t)\|_{\dot{B}_{\infty,\infty}^a}
\lesssim\frac{\v}{\langle t\rangle^{d+a}}\Big(\|\mathcal{D}_1^a\nabla_xf^+_{\mathrm{in}}\|_{L_x^1L_v^\infty}
+\|\mathcal{D}_2^a\nabla_xf^+_{\mathrm{in}}\|_{L_x^1L_v^\infty}
+\|\nabla_xf^+_{\mathrm{in}}\|_{L_x^1L_v^\infty}\Big).
\end{equation}

\textit{Estimate for $I_3$.}
Applying $\delta_{\alpha}^{x}$ to $I_3$, we decompose the resulting expression as
\begin{align*}
\delta_\a^xI_3(t,x)
=&\frac{1}{t^d}\int_{\mathbb{R}^d}\delta_{z_5}f_{\mathrm{in}}
\Big(\cdot,\  \frac{x-w}{t}
+W_{0,t}^-(w,\frac{x-w}{t})\Big)\Big(w+ Y^-_{0,t}(w,\frac{x-w}{t})\Big)\,dw\\
&+\frac{1}{t^d}\int_{\mathbb{R}^d}\delta_{z_6}f_{\mathrm{in}}\Big(w+ Y^-_{0,t}(w,\frac{x-w-\alpha}{t}),\ \cdot\Big)\Big(\frac{x-w}{t}
+W_{0,t}^-(w,\frac{x-w}{t})\Big)\,dw,
\end{align*}
where
\[
z_5 := Y^-_{0,t}(w,\tfrac{x-w}{t}) - Y^-_{0,t}(w,\tfrac{x-w-\alpha}{t}), \  \  \
z_6 := \tfrac{\alpha}{t} + W_{0,t}^-(w,\tfrac{x-w}{t}) - W_{0,t}^-(w,\tfrac{x-w-\alpha}{t}).
\]
By Lemma~\ref{estimates001},
$$
|z_5|+|z_6|\lesssim\frac{|\alpha|}{t}.
$$
Therefore, one gets that
\begin{align*}
|\delta_\a^xI_3(t,x)|
\leq&\frac{1}{t^{d+a}}\int_{\mathbb{R}^d}\frac{\Big|\delta_{z_5}f_{\mathrm{in}}
\Big(\cdot,\frac{x-w}{t}
+W_{0,t}^-(w,\frac{x-w}{t})\Big)\Big(w+ Y^-_{0,t}(w,\frac{x-w}{t})\Big)\Big|}{|z_5|^a}|\a|^a\,dw\\
&+\frac{1}{t^{d+a}}\int_{\mathbb{R}^d}\frac{\Big|\delta_{z_6}f_{\mathrm{in}}\Big(w+ Y^-_{0,t}(w,\frac{x-w-\alpha}{t}),\cdot\Big)\Big(\frac{x-w}{t}
+W_{0,t}^-(w,\frac{x-w}{t})\Big)
\Big|}{|z_6|^a}|\a|^a\,dw.
\end{align*}
Changing variables $w=x-tv$ and using \eqref{z4}, we have
\begin{align*}
\frac{|\delta_\a^xI_3(t,x)|}{|\alpha|^a}
\lesssim&\frac{1}{\langle t\rangle^a}\int_{\mathbb{R}^d}\mathcal{D}_1^af_{\mathrm{in}}
\Big(\mathcal{X}^-(0),\mathcal{V}^-(0)\Big)
\,dv\\
&+\frac{1}{\langle t\rangle^a}\int_{\mathbb{R}^d}\mathcal{D}_2^a
f_{\mathrm{in}}\Big(\mathcal{X}^-(0;t,x-\alpha,v-\frac{\a}{t}),\
\mathcal{V}^-(0)\Big)\,dv.
\end{align*}
Taking the supremum over \(\alpha\) and applying Lemma~\ref{lemma:J123-H} with \(\mathcal L_3\) to the two terms above,  with
\[
\mathcal H=\mathcal D_1^a f_{\mathrm{in}},\qquad
\mathcal D_2^a f_{\mathrm{in}},
\]
we infer that
\begin{equation}\label{chuzhi-I05}
\|I_3(t)\|_{\dot B^a_{1,\infty}}
\lesssim
\frac{1}{\langle t\rangle^{a}}
\Bigl(
\|\mathcal D_1^a f_{\mathrm{in}}\|_{L_x^1L_v^1}
+\|\mathcal D_2^a f_{\mathrm{in}}\|_{L_x^1L_v^1}
\Bigr).
\end{equation}
For the second term, we also use \eqref{yabi01} to justify the same change of variables for
\[
\mathcal X^-\bigl(0;t,x-\alpha,v-\tfrac{\alpha}{t}\bigr).
\]
Likewise,
\begin{equation}\label{chuzhi-I06}
\|I_3(t)\|_{\dot B^a_{\infty,\infty}}
\lesssim
\frac{1}{\langle t\rangle^{d+a}}
\Bigl(
\|\mathcal D_1^a f_{\mathrm{in}}\|_{L_x^1L_v^\infty}
+\|\mathcal D_2^a f_{\mathrm{in}}\|_{L_x^1L_v^\infty}
\Bigr).
\end{equation}

Combining \eqref{chuzhi-I01}--\eqref{chuzhi-I06} with \eqref{large-data} and \eqref{qn}, and using the triangle inequality, we obtain
\begin{equation}\label{chuzhi076}
\langle t\rangle^{a}\|I(t)\|_{\dot{B}_{1,\infty}^a}+
 \langle t\rangle^{d+a}\|I(t)\|_{\dot{B}_{\infty,\infty}^a}\lesssim \varepsilon+\varepsilon_0.
\end{equation}

\textit{Step 2. Estimates for $R^{\pm}$.}
Taking \(\mathcal F=E\) and $\eta=\mu$  in \eqref{zuoyongxiang-08}, we have
$$
R^{\pm}(t,x)=\Gamma^\pm[E,\mu](t,x).
$$
Therefore, by \eqref{zuoyong08}, \eqref{zuoyong09} and  \eqref{ziyouliu01},
\begin{equation}\label{zuoyong028}
\|R^{\pm}(t)\|_{L^1_x}\lesssim\varepsilon^2,
\qquad
\|R^{\pm}(t)\|_{L^\infty_x}\lesssim\frac{\varepsilon^2}{\langle t\rangle^d}.
\end{equation}

It remains to estimate the Besov seminorm.
Fix $p\in\{1,\infty\}$.

\noindent\emph{Case 1: $|\alpha|\ge \frac{t}{2}$.}
Using
$\|\delta_\alpha^x R^\pm(t)\|_{L^p_x}\le 2\|R^\pm(t)\|_{L^p_x}$ and \eqref{zuoyong028},
we get
\begin{equation}\label{zuoyong028p}
\frac{\|\delta^x_\alpha R^\pm(t)\|_{L^p_x}}{|\alpha|^a}
\lesssim\frac{\varepsilon^2}{\langle t\rangle^{\frac{d(p-1)}{p}+a}}.
\end{equation}

\noindent\emph{Case 2: $|\alpha|<\frac{t}{2}$.}
By \eqref{z4} and the change of variables $w=x-tv$, we may write
\begin{align}\label{zuoyong01}
R^{\pm}(t,x)
=&\mp\frac{1}{t^d}\int_{0}^{t}\int_{\mathbb{R}^{d}}
E\big(s,w+\frac{x-w}{t}s+Y_{s,t}^{\pm}(w,\frac{x-w}{t})\big)
\cdot\nabla_{v}\mu\big(\frac{x-w}{t}+W_{s,t}^\pm(w,\frac{x-w}{t})\big)
\,dw\,ds\nonumber\\
&\pm\frac{1}{t^d}\int_{0}^{t}\int_{\mathbb{R}^{d}}E(s,w+\frac{x-w}{t}s) \cdot\nabla_{v}\mu(\frac{x-w}{t})\,dw\,ds.
\end{align}
Applying $\delta^x_\alpha$ to \eqref{zuoyong01},
changing variables back to $v$ and \eqref{z4}, we decompose
\[
\delta^x_\alpha R^\pm(t,x)=\mathcal{D}_1(t, x)+\mathcal{D}_2(t, x)+\mathcal{D}_3(t, x)+\mathcal{D}_4(t, x),
\]
where
\begin{align*}
 \mathcal{D}_1(t, x) := & \mp\int_0^t\int_{\mathbb{R}^d}\big[\delta_{\frac{s\alpha}{t}}
E(s,\cdot)(\mathcal{X}^\pm(s))\cdot\nabla_v\mu(\mathcal{V}^\pm(s))-\delta_{\frac{s\alpha}{t}}
E(s,\cdot)(x-(t-s)v)\cdot\nabla_v\mu(v)\big]\,dv\,ds,\\
\mathcal{D}_2(t, x):= & \mp\int_0^t\int_{\mathbb{R}^d}\big[E\big(s,\mathcal X^\pm(s;t,\,x-\alpha,\,v-\frac{\alpha}{t})\big)
\cdot\big(\delta_{\frac{\alpha}{t}}\nabla_v\mu\big)
\big(\mathcal{V}^\pm(s)\big)\\
&\quad-E\big(s,x-(t-s)v-\frac{s\alpha}{t}\big)
\cdot\big(\delta_{\frac{\alpha}{t}}\nabla_v\mu\big)(v)\big]\,dv\,ds,\\
\mathcal{D}_3(t, x):= & \mp\int_0^t\int_{\mathbb{R}^d}
\big[E\big(s,\mathcal{X}^\pm(s)-\frac{s\alpha}{t}\big)-
E\big(s,\mathcal X^\pm(s;t,\,x-\alpha,\,v-\frac{\alpha}{t})\big)\big]
\cdot\nabla_v\mu(\mathcal{V}^\pm(s))\,dv\,ds,\\
\mathcal{D}_4(t, x):= & \pm\int_0^t\int_{\mathbb{R}^d}E\big(s,\mathcal X^\pm(s;t,\,x-\alpha,\,v-\frac{\alpha}{t})\big)\\
&\quad\cdot\big[\nabla_v\mu\big(\mathcal{V}^\pm(s)-\frac{\alpha}{t}\big)
-\nabla_v\mu\big(\mathcal V^\pm(s;t,\,x-\alpha,\,v-\frac{\alpha}{t})\big)\big] \,dv\,ds.
\end{align*}

\noindent\textit{Estimate for $\mathcal{D}_1$.}
By definition,
$$
\mathcal{D}_1(t,x)=\Gamma^\pm[\delta_{\frac{s\alpha}{t}}E(s,\cdot),\mu](t,x).
$$
It follows from \eqref{zuoyong08} that
$$
\|\mathcal{D}_1(t)\|_{L^1_x}
\lesssim \varepsilon\int_0^t
\frac{\|\delta_{\frac{s\alpha}{t}}E(s)\|_{L^1_x}}
{{|\frac{s\alpha}{t}|}^a}{\frac{|s\alpha|^a}{t^a}}\frac{ds}{
\langle s\rangle^{d-\gamma-1}}.
$$
Using \eqref{Besov} and \eqref{ziyouliu01}, we obtain
\[
\|D_1(t)\|_{L^1_x}
\lesssim
\frac{\varepsilon |\alpha|^a}{t^a}
\int_0^t
\frac{\|E(s)\|_{\dot B^a_{1,\infty}}}{\langle s\rangle^{d-a-\gamma-1}}\,ds
\lesssim
\frac{\varepsilon^2 |\alpha|^a}{\langle t\rangle^a}.
\]
Similarly, by \eqref{zuoyong09},
$$
\|\mathcal{D}_1(t)\|_{L^\infty_x}\lesssim \frac{\varepsilon}{\langle t\rangle^d}
\int_0^{\frac{t}{2}}
\frac{1}{\langle s\rangle^{d-\gamma-1}}
\frac{\|\delta_{\frac{s\alpha}{t}}E(s)\|_{L^1_x}}
{{|\frac{s\alpha}{t}|}^a}{\frac{|s\alpha|^a}{t^a}} \,ds
+
\varepsilon
\int_{\frac{t}{2}}^t
\frac{1}{\langle s\rangle^{d-\gamma-1}}
\frac{\|\delta_{\frac{s\alpha}{t}}E(s)\|_{L^\infty_x}}
{{|\frac{s\alpha}{t}|}^a}{\frac{|s\alpha|^a}{t^a}} \,ds.
$$
Again using \eqref{Besov} and \eqref{ziyouliu01}, we infer
\begin{align*}
\|\mathcal{D}_1(t)\|_{L^\infty_x}
%=&\ \|\Gamma^\pm[\delta_{\frac{s\alpha}{t}}E,\mu]
%(t)\|_{L^\infty_x}\\
\lesssim&\ \frac{\varepsilon}{t^{d+a}}\int_0^\frac{t}{2}
\frac{|s\alpha|^a}{\langle s\rangle^{d-\gamma-1}}\|E(s)\|_{\dot{B}^a_{1,\infty}}ds+\frac{\varepsilon}{t^a}
\int_\frac{t}{2}^t
\frac{|s\alpha|^a}{\langle s\rangle^{d-\gamma-1}}\|E(s)\|_{\dot{B}^a_{\infty,\infty}}ds\\
\lesssim& \frac{\varepsilon^2|\alpha|^a}{\langle t\rangle^{d+a}}.
\end{align*}

\noindent\textit{Estimate for $\mathcal{D}_2$.}
We decompose
$$
\mathcal{D}_2(t,x)=\mathcal{D}_{21}(t,x)+\mathcal{D}_{22}(t,x),
$$
where
\begin{align*}
\mathcal{D}_{21}(t,x):=&\mp \int_0^t \int_{\mathbb{R}^d} \big[ E\big(s,\mathcal X^\pm(s;t,\,x-\alpha,\,v-\frac{\alpha}{t})\big)-E\big(s, x-(t-s)v-\tfrac{s\alpha}{t}\big) \big]\\
&\qquad\cdot \big(\delta_{\frac{\alpha}{t}} \nabla_v \mu\big)(\mathcal{V}^\pm(s)) \, dv \, ds,
\end{align*}
$$
\mathcal{D}_{22}(t,x):=\mp \int_0^t \int_{\mathbb{R}^d} E\big(s, x-(t-s)v-\tfrac{s\alpha}{t}\big) \cdot \big[ \big(\delta_{\frac{\alpha}{t}} \nabla_v \mu\big)(\mathcal{V}^\pm(s))-\big(\delta_{\frac{\alpha}{t}} \nabla_v \mu\big)(v) \big] \, dv \, ds.
$$

For $\mathcal{D}_{21}$. By the mean value theorem and Assumption~\ref{ass:regularity},
\begin{equation}\label{guji329}
|(\delta_{\frac{\alpha}{t}} \nabla_v \mu )(\mathcal{V}^\pm(s))| \lesssim \frac{ |\alpha|}{t}\frac1{\langle v\rangle^N},
\end{equation}
since
$\langle \mathcal{V}^{\pm}(s)-(1+\theta)\frac{\alpha}{t} \rangle\sim \langle v \rangle, \theta\in[0,1]$.
Set $\tau_1 := Y_{s,t}^\pm(x - tv, v - \tfrac{\alpha}{t})$.
By Lemma \ref{estimates001},
\[
|\tau_1|\le \|Y_{s,t}^\pm\|_{L^\infty_{x,v}}\lesssim \varepsilon\langle s\rangle^{-(d-\gamma-1)}, \quad |\tau_1|\ll 1.
\]
Since
\[
\mathcal X^\pm\bigl(s;t,\,x-\alpha,\,v-\frac{\alpha}{t}\bigr)
=
x-(t-s)v-\frac{s\alpha}{t}+\tau_1,
\]
it follows from \eqref{guji329} that
%$$
%|\mathcal{D}_{21}(t,x)|
%\lesssim \varepsilon^a\frac{ |\alpha|}{t}\int_0^t
%\int_{\mathbb{R}^d}
%\frac{|\delta_{\tau_1}E(s,\cdot)
%(\mathcal{X}^\pm(s;t,\,x-\alpha,\,v-\frac{\alpha}{t}))|}{|\tau_1|^a}
%\frac{dv \, ds}{\langle v\rangle^N\langle s\rangle^{(d-\gamma-1)a}}.
%$$
%Using $|\tau_1|\le 1$ and $|\alpha|<\frac t2$, hence $\frac{|\alpha|}{t}\lesssim \frac{|\alpha|^a}{\langle t\rangle^a}$, we obtain
\begin{align*}
|\mathcal{D}_{21}(t,x)|\lesssim&\ \frac{\varepsilon^a|\alpha|^a}{\langle t\rangle^a}\int_0^t
\int_{\mathbb{R}^d}
\sup\limits_{|\tau_1|\leq1}\frac{|\delta_{\tau_1}E(s,\cdot)
(\mathcal{X}^\pm(s;t,\,x-\alpha,\,v-\frac{\alpha}{t}))|}{|\tau_1|^a}
\frac{dv \, ds}{\langle v\rangle^N\langle s\rangle^{(d-\gamma-1)a}}.
\end{align*}

Integrating in $x$ and applying the change of variables
$\mathbf{X}=\mathcal{X}^\pm(s;t,\,x-\alpha,\,v-\frac{\alpha}{t})$,
together with \eqref{yabi01}, we get
\begin{align}\label{D2estimate}
\|\mathcal{D}_{21}(t)\|_{L^1}
\lesssim&\frac{\varepsilon^a|\alpha|^a}{\langle t\rangle^a}\int_0^t
\int_{\mathbb{R}^d}\int_{\mathbb{R}^d}
\sup\limits_{|\tau_1|\leq1}\frac{|\delta_{\tau_1}E(s,\cdot)
(\mathbf{X})|}{|\tau_1|^a}\big|\det\big(\nabla_x\mathcal X^\pm(s;t,\,x-\alpha,\,v-\frac{\alpha}{t})\big)\big|^{-1}d\mathbf{X}\nonumber\\
&\qquad\cdot\frac{dv \, ds}{\langle v\rangle^N\langle s\rangle^{(d-\gamma-1)a}}\nonumber\\
\lesssim&
\frac{\varepsilon^a|\alpha|^a}{\langle t\rangle^a}
\int_0^t
\Bigl\|
\sup_{|\tau_1|\le 1}\frac{|\delta_{\tau_1}E(s)|}{|\tau_1|^a}
\Bigr\|_{L^1_x}
\frac{ds}{\langle s\rangle^{(d-\gamma-1)a}}.
\end{align}
By \eqref{eq:dq-1-local} and \eqref{F03},
\begin{equation}\label{DL-21}
\Big\|\sup_{|\tau_1|\le 1}\frac{|\delta_{\tau_1}E(s)|}{|\tau_1|^a}\Big\|_{L^1_x}
\lesssim \|E(s)\|_{\dot F^a_{1,\infty}}\lesssim\|E(s)\|_{L^1_x}^{1-a}\|\nabla_x E(s)\|_{L^1_x}^{a}.
\end{equation}
Substituting this into \eqref{D2estimate} and using \eqref{ziyouliu01}, we get
$$
\|\mathcal{D}_{21}(t)\|_{L^1_x}
\lesssim
\frac{\varepsilon^{a+1}|\alpha|^a}{\langle t\rangle^a}
\int_0^t
\frac{ds}{\langle s\rangle^{(d-\gamma-1)a+1-\gamma}}\lesssim
\frac{\varepsilon^{a+1}|\alpha|^a}{\langle t\rangle^a}.
$$

Split the time integral into $[0,\frac{t}{2}]$ and $[\frac{t}{2},t]$. For $0\le s\le t/2$, we use the change of variables
$\mathbf{X}:=X^\pm(s;t,\,x-\alpha,\,v-\frac{\alpha}{t})$,
together with \eqref{yabi02}, and infer that
\begin{align*}
|\mathcal{D}_{21}(t,x)|
%\lesssim&\ \frac{\varepsilon^a|\alpha|^a}{\langle t\rangle^a}\int_0^{\frac{t}{2}}
%\int_{\mathbb{R}^d}\sup\limits_{|\tau_1|\leq1}\frac{|\delta_{\tau_1}E(s,\cdot)
%(\mathbf{X})|}{|\tau_1|^a}d\mathbf{X}
%\frac{ ds}{(t-s)^d\langle s\rangle^{(d-\gamma-1)a}}\\
%&+\frac{\varepsilon^a|\alpha|^a}{\langle t\rangle^a}\int_{\frac{t}{2}}^t
%\|\sup\limits_{|\tau_1|\leq1}\frac{|\delta_{\tau_1}E(s)|}{|\tau_1|^a}\|_{L^\infty_x}\int_{\mathbb{R}^d}\langle v\rangle^{-N}dv\frac{ ds}{\langle s\rangle^{(d-\gamma-1)a}}\\
\lesssim&\ \frac{\varepsilon^a|\alpha|^a}{\langle t\rangle^{d+a}}\int_0^{\frac{t}{2}}
\Big\|\sup\limits_{|\tau_1|\leq1}\frac{|\delta_{\tau_1}E(s)|}{|\tau_1|^a}\Big\|_{L^1_x}
\frac{ ds}{\langle s\rangle^{(d-\gamma-1)a}}\\
&+\frac{\varepsilon^a|\alpha|^a}{\langle t\rangle^a}\int_{\frac{t}{2}}^t
\Big\|\sup\limits_{|\tau_1|\leq1}\frac{|\delta_{\tau_1}E(s)|}{|\tau_1|^a}\Big\|_{L^\infty_x}
\frac{ ds}{\langle s\rangle^{(d-\gamma-1)a}}.
\end{align*}
Using \eqref{DL-21} for the first term, and \eqref{eq:dq-infty} together with \eqref{FTL0} for the second term, we obtain
\begin{align*}
|\mathcal{D}_{21}(t,x)|
\lesssim& \frac{\varepsilon^a|\alpha|^a}{\langle t\rangle^{d+a}}\int_0^{\frac{t}{2}}
|E(s)\|_{L^1_x}^{1-a}\|\nabla_x E(s)\|_{L^1_x}^{a}
\frac{ ds}{\langle s\rangle^{(d-\gamma-1)a}}\\
&\quad+\frac{\varepsilon^a|\alpha|^a}{\langle t\rangle^a}\int_{\frac{t}{2}}^t
\|E(s)\|_{\dot{B}^a_{\infty,\infty}}
\frac{ ds}{\langle s\rangle^{(d-\gamma-1)a}}.
\end{align*}
Therefore, by \eqref{ziyouliu01},
$$
\|\mathcal{D}_{21}(t)\|_{L^\infty_x}
\lesssim\ \frac{\varepsilon^{a+1}|\alpha|^a}{\langle t\rangle^{d+a}}\int_0^{\frac{t}{2}}\frac{ds}{\langle s\rangle^{(d-\gamma-1)a+1-\gamma}}
+\frac{\varepsilon^{a+1}|\alpha|^a}{\langle t\rangle^{a}}\int^t_{\frac{t}{2}}\frac{ds}{\langle s\rangle^{(d-\gamma-1)a+d+1-\gamma}}\lesssim \frac{\varepsilon^{a+1}|\alpha|^a}{\langle t\rangle^{d+a}}.
$$

For $\mathcal{D}_{22}$.
Using the mean value theorem twice and \eqref{tezhengW}, we have
% and $\langle \theta(\mathcal{V}^\pm(s)-v)+v-(1-\beta_1)\frac{\alpha}{t} \rangle \sim
%\langle v \rangle$,
\begin{align*}
&\big|(\delta_{\frac{\alpha}{t}} \nabla_v \mu) (\mathcal{V}^\pm(s))
-(\delta_{\frac{\alpha}{t}} \nabla_v \mu)(v)\big|\nonumber\\
\leq&\big|\nabla_v\big(\delta_{\frac{\alpha}{t}}\nabla_v \mu\big) \big(\theta\mathcal{V}^\pm(s)+(1-\theta)v\big)\big||\mathcal{V}^\pm(s)-v|\nonumber\\
%\leq&\big|\nabla_v^3 \mu \big(\beta_1\big(\theta\mathcal{V}^\pm(s)+(1-\theta)v\big)
%+(1-\beta_1)\big(\theta\mathcal{V}^\pm(s)+(1-\theta)v-\frac{\alpha}{t}\big)\big)\big| \frac{|\alpha|}{ t } \|W_{s,t}^\pm\|_{L^\infty_{x,v}}\nonumber\\
%\lesssim&  \frac{|\alpha|}{ t } \sup_{\theta\in [0,1]} \sup_{\beta_1\in [0,1]} \frac{\varepsilon }{\langle \theta(\mathcal{V}^\pm(s)-v)+v-(1-\beta_1)\frac{\alpha}{t} \rangle^N  \langle s \rangle^{d-\gamma} }\nonumber\\
\lesssim&  \frac{|\alpha|}{ t }  \frac{\varepsilon }{\langle v\rangle^N  \langle s \rangle^{d-\gamma} }.
\end{align*}
Since $|\alpha|<\frac t2$, it follows that
$$
|\mathcal{D}_{22}(t,x)| \lesssim\frac{\varepsilon|\alpha|^a}{\langle t\rangle^a}\int_0^t
\int_{\mathbb{R}^d}
|E(s,x-(t-s)v-\frac{s\alpha}{t})|
\frac{dv \, ds}{\langle v\rangle^N\langle s\rangle^{d-\gamma}}.
$$
Integrating in $x$ and using the change of variables $y=x-(t-s)v-\frac{s\alpha}{t}$,
we have
$$
\|\mathcal{D}_{22}(t)\|_{L^1_x}\lesssim\frac{\varepsilon|\alpha|^a}{\langle t\rangle^a}\int_0^t\|E(s)\|_{L^1_x}\int_{\mathbb{R}^d}\langle v\rangle^{-N}dv\frac{ds}{\langle s\rangle^{d-\gamma}}.
$$
Hence, by \eqref{ziyouliu01},
\[
\|\mathcal{D}_{22}(t)\|_{L^1_x}
\lesssim
\frac{\varepsilon^{2} |\alpha|^a}{\langle t\rangle^a}
\int_0^t
\frac{ds}{\langle s\rangle^{d-2\gamma}}
\lesssim
\frac{\varepsilon^{a+1}|\alpha|^a}{\langle t\rangle^a}.
\]

To estimate the $L^\infty_x$ norm, we split again the time interval into
$[0,\frac t2]$ and $[\frac t2,t]$. By the change of variables
\[
y=x-(t-s)v-\frac{s\alpha}{t},
\qquad0\le s\le \frac t2,
\]
%By the change of variables $v\mapsto y=x-(t-s)v-\frac{s\alpha}{t}$ in $[0,\frac{t}{2}]$ ,
we obtain
\begin{align*}
|\mathcal{D}_{22}(t,x)|
\lesssim&
\frac{\varepsilon|\alpha|^a}{\langle t\rangle^a}\int_0^{\frac{t}{2}}
\|E(s)\|_{L^1_x}\frac{1}{(t-s)^d}
\frac{ds}{\langle s\rangle^{d-\gamma}}\\
&\quad+\frac{\varepsilon|\alpha|^a}{\langle t\rangle^a}\int_{\frac{t}{2}}^t
\|E(s)\|_{L^\infty}\int_{\mathbb{R}^d}\langle v\rangle^{-N}dv
\frac{ds}{\langle s\rangle^{d-\gamma}}.
\end{align*}
By \eqref{ziyouliu01}, it holds that
$$
\|\mathcal{D}_{22}(t)\|_{L^\infty_x}
\lesssim
\frac{\varepsilon^2|\alpha|^a}{\langle t\rangle^{d+a}}\int_0^{\frac{t}{2}}\frac{ds}{\langle s\rangle^{d-2\gamma+1}}
+\frac{\varepsilon^2|\alpha|^a}{\langle t\rangle^{a}}\int^t_{\frac{t}{2}}\frac{ds}{\langle s\rangle^{2d-2\gamma+1}}\lesssim \frac{\varepsilon^{a+1}|\alpha|^a}{\langle t\rangle^{d+a}}.
$$

Combining the estimates for $\mathcal{D}_{21}$ and $\mathcal{D}_{22}$, we obtain
\[
\|\mathcal{D}_{2}(t)\|_{L^1_x}
\lesssim
\frac{\varepsilon^{a+1}|\alpha|^a}{\langle t\rangle^a}, \quad
\|\mathcal{D}_{2}(t)\|_{L^\infty_x}
\lesssim
\frac{\varepsilon^{a+1}|\alpha|^a}{\langle t\rangle^{d+a}},
\]

\noindent\textit{Estimate for $\mathcal{D}_3$.}
By \eqref{z4},
\[
\mathcal X^\pm(s)-\frac{s\alpha}{t}
-\mathcal X^\pm\bigl(s;t,x-\alpha,v-\frac{\alpha}{t}\bigr)
=\tau_2,
\]
where
\[
\tau_2:=Y^\pm_{s,t}(x-tv,v)
-Y^\pm_{s,t}\bigl(x-tv,v-\frac{\alpha}{t}\bigr).
\]
Therefore,
$$
|\mathcal{D}_3(t,x)|
\lesssim \int_0^t\int_{\R^d}
|\delta_{\tau_2}E(s,\,\cdot
)(\mathcal{X}^\pm(s)-\frac{s\alpha}{t})|
\,|\nabla_v\mu(\mathcal{V}^\pm(s))|
\,dv\,ds.
$$
Moreover,
\begin{equation*}
|\tau_2|\lesssim \frac{\varepsilon}{\langle s \rangle^{d-\gamma-2}} \big| \frac{\alpha}{t} \big|, \qquad
|\tau_2| \lesssim\varepsilon\,\langle s\rangle^{-(d-\gamma-1)}
\ll 1,
\end{equation*}
and
$|\nabla_v\mu(\mathcal V^\pm(s))|
\lesssim
\langle v\rangle^{-N}$.
Therefore,
\begin{equation*}
|\mathcal D_3(t,x)|
%&\lesssim
%\int_0^t\int_{\mathbb R^d}\frac{\bigl|
%\delta_{\tau_2}E(s,\cdot)\bigl(\mathcal X^\pm(s)-\frac{s\alpha}{t}\bigr)
%\bigr|}{|\tau_2|^a}\,|\tau_2|^a\,
%\frac{dv\,ds}{\langle v\rangle^N} \\
\lesssim
\frac{\varepsilon^a|\alpha|^a}{\langle t\rangle^a}
\int_0^t\int_{\mathbb R^d}
\sup_{|\tau_2|\le1}\frac{\bigl|
\delta_{\tau_2}E(s,\cdot)\bigl(\mathcal X^\pm(s)-\frac{s\alpha}{t}\bigr)\bigr|}{|\tau_2|^a}
\frac{dv\,ds}{\langle v\rangle^N\langle s\rangle^{(d-\gamma-2)a}}.
\end{equation*}
Integrating in $x$ and applying the change of variables
$X=\mathcal X^\pm(s)-\frac{s\alpha}{t}$,
together with Lemma \ref{Lemma3.4-YKB}, we obtain
\begin{align*}
\|\mathcal D_3(t)\|_{L^1_x}
\lesssim&
\frac{\varepsilon^a|\alpha|^a}{\langle t\rangle^a}
\int_0^t\int_{\mathbb R^d}\int_{\mathbb R^d}
\sup_{|\tau_2|\le1}
\frac{|\delta_{\tau_2}E(s,\cdot)(X)|}{|\tau_2|^a}
\bigl|\det(\nabla_x\mathcal X^\pm(s))\bigr|^{-1}
\,dX\,
\frac{dv\,ds}{\langle v\rangle^N\langle s\rangle^{(d-\gamma-2)a}} \\
\lesssim&
\frac{\varepsilon^a|\alpha|^a}{\langle t\rangle^a}
\int_0^t
\Bigl\|
\sup_{|\tau_2|\le1}\frac{|\delta_{\tau_2}E(s)|}{|\tau_2|^a}
\Bigr\|_{L^1_x}
\frac{ds}{\langle s\rangle^{(d-\gamma-2)a}}.
\end{align*}
By \eqref{eq:dq-1-local} and \eqref{F03},
\begin{equation}\label{DL-2122}
\Bigl\|
\sup_{|\tau_2|\le1}\frac{|\delta_{\tau_2}E(s)|}{|\tau_2|^a}
\Bigr\|_{L^1_x}
\lesssim
\|E(s)\|_{\dot F^a_{1,\infty}}
\lesssim
\|E(s)\|_{L^1_x}^{1-a}\|\nabla_x E(s)\|_{L^1_x}^{a}.
\end{equation}
Hence, using \eqref{ziyouliu01},
$$
\|\mathcal{D}_3(t)\|_{L^1_x}
\lesssim\frac{\varepsilon^a|\alpha|^a}{\langle t\rangle^a}
\int_0^t\frac{\|E(s)\|_{\dot{F}^a_{1,\infty}}}{\langle s\rangle^{(d-\gamma-2)a}}ds\lesssim\frac{\varepsilon^{a+1}|\alpha|^a}{\langle t\rangle^a}
\int_0^t
\frac{ds}{\langle s\rangle^{(d-\gamma-2)a+1-\gamma}}
\lesssim\frac{\varepsilon^{a+1}
|\alpha|^a}{\langle t\rangle^a}.
$$

Similarly, splitting into $[0,\frac{t}{2}]$ and $[\frac{t}{2},t]$, using the change of variables
\[
X=\mathcal X^\pm(s)-\frac{s\alpha}{t},
\qquad 0\le s\le \frac t2,
\]
together with Lemma \ref{Lemma3.4-YKB}, we obtain
\begin{align*}
|\mathcal{D}_3(t, x)|
\lesssim&\ \frac{\varepsilon^{a}|\alpha|^{a}}{\langle t\rangle^{a}}
\int_0^{\frac{t}{2}}\frac{1}{(t-s)^d}
\Big\|\sup_{|\tau_2|\le 1}\frac{|\delta_{\tau_2}E(s)|}{|\tau_2|^{a}}\Big\|_{L^1_x}
\frac{ds}{\langle s\rangle^{(d-\gamma-2)a}}\\
&\ +\frac{\varepsilon^{a}|\alpha|^{a}}{\langle t\rangle^{a}}
\int_{\frac{t}{2}}^{t}
\Big\|\sup_{|\tau_2|\le 1}\frac{|\delta_{\tau_2}E(s)|}{|\tau_2|^{a}}\Big\|_{L^\infty_x}
\frac{ds}{\langle s\rangle^{(d-\gamma-2)a}}.
\end{align*}
By Lemma \ref{lem:difference-F},
\[
\Bigl\|
\sup_{|\tau_2|\le1}\frac{|\delta_{\tau_2}E(s)|}{|\tau_2|^a}
\Bigr\|_{L^\infty_x}
\lesssim
\|E(s)\|_{\dot F^a_{\infty,\infty}}
\sim
\|E(s)\|_{\dot B^a_{\infty,\infty}},
\]
using \eqref{DL-2122} and \eqref{ziyouliu01}, we obtain
\begin{align*}
\|\mathcal D_3(t)\|_{L^\infty_x}
\lesssim&
\frac{\varepsilon^{a+1}|\alpha|^a}{\langle t\rangle^{d+a}}
\int_0^{\frac{t}{2}}
\frac{ds}{\langle s\rangle^{(d-\gamma-2)a+1-\gamma}}
+
\frac{\varepsilon^{a+1}|\alpha|^a}{\langle t\rangle^a}
\int_{\frac{t}{2}}^{t}
\frac{ds}{\langle s\rangle^{(d-\gamma-2)a+d+1-\gamma}} \\
\lesssim&
\frac{\varepsilon^{a+1}|\alpha|^a}{\langle t\rangle^{d+a}},
\end{align*}
where the last step uses $a \in \big(\frac{\gamma}{d-\gamma-2}, 1 \big)$.

\noindent\emph{Estimate for $\mathcal{D}_4$.}
By \eqref{z4},
\[
\mathcal V^\pm\bigl(s;t,x-\alpha,v-\frac{\alpha}{t}\bigr)
=
v-\frac{\alpha}{t}
+
W^\pm_{s,t}\bigl(x-tv,v-\frac{\alpha}{t}\bigr).
\]
Hence, by the mean value theorem, Assumption \ref{ass:regularity}, and Lemma~\ref{estimates001},
\begin{align*}
&|\nabla_v\mu(\mathcal{V}^\pm(s)-\frac{\alpha}{t})
-\nabla_v\mu(\mathcal V^\pm(s;t,\,x-\alpha,\,v-\frac{\alpha}{t}))|\\
&\qquad\lesssim
\frac{1}{\langle \xi\rangle^N}
\bigl|
W^\pm_{s,t}(x-tv,v)
-
W^\pm_{s,t}\bigl(x-tv,v-\frac{\alpha}{t}\bigr)
\bigr|,
\end{align*}
where \(\xi\) is a point on the line segment joining the two arguments above. Since
$$
\langle \theta(W_{s,t}^\pm(x-tv,v)
-W_{s,t}^\pm(x-tv,v-\frac{\alpha}{t}))+v-\frac{\alpha}{t}
+W_{s,t}^\pm(x-tv,v-\frac{\alpha}{t})\rangle \sim\langle v\rangle,
$$
it follows that \(\langle \xi\rangle\sim \langle v\rangle\), and therefore
$$
|\nabla_v\mu(\mathcal{V}^\pm(s)-\frac{\alpha}{t})
-\nabla_v\mu(\mathcal V^\pm(s;t,\,x-\alpha,\,v-\frac{\alpha}{t}))|
\lesssim
\frac{1}{\langle v\rangle^N}
\frac{\varepsilon}{\langle s\rangle^{d-\gamma-1}}
\frac{|\alpha|^a}{t^a}.
$$
Hence,
\[
|\mathcal{D}_4(t, x)|
\lesssim\frac{\varepsilon|\alpha|^a}{\langle t\rangle^a}\int_0^t\int_{\mathbb{R}^d}
\frac{|E(s,\mathcal X^\pm(s;t,\,x-\alpha,\,v-\frac{\alpha}{t}))|}{\langle v\rangle^N\langle s\rangle^{d-\gamma-1}}dv\, ds.
\]
Integrating in $x$, applying the change of variables
$X=\mathcal X^\pm\bigl(s;t,x-\alpha,v-\frac{\alpha}{t}\bigr)$,
and using \eqref{yabi01}, we obtain
\begin{align*}
\|\mathcal{D}_4(t)\|_{L^1_x}
\lesssim&
\frac{\varepsilon |\alpha|^a}{\langle t\rangle^a}
\int_0^t\int_{\mathbb R^d}\int_{\mathbb R^d}
|E(s,\cdot)(X)|
\big|
\det\bigl(
\nabla_x\mathcal X^\pm\bigl(s;t,x-\alpha,v-\frac{\alpha}{t}\bigr)
\bigr)
\big|^{-1}
\,dX\,
\frac{dv\,ds}{\langle v\rangle^N\langle s\rangle^{d-\gamma-1}} \\
\lesssim&
\frac{\varepsilon |\alpha|^a}{\langle t\rangle^a}
\int_0^t
\frac{\|E(s)\|_{L^1_x}}{\langle s\rangle^{d-\gamma-1}}\,ds.
\end{align*}
Using \eqref{ziyouliu01}, we deduce
\[
\|\mathcal{D}_4(t)\|_{L^1_x}
\lesssim
\frac{\varepsilon^2 |\alpha|^a}{\langle t\rangle^a}
\int_0^t
\frac{ds}{\langle s\rangle^{d-2\gamma}}
\lesssim
\frac{\varepsilon^2 |\alpha|^a}{\langle t\rangle^a}.
\]

For the $L^\infty_x$ norm, arguing as above, we get
\begin{align*}
\|\mathcal{D}_4(t)\|_{L^\infty_x}
\lesssim&
\frac{\varepsilon |\alpha|^a}{\langle t\rangle^a}
\int_0^{\frac{t}{2}}
\frac{1}{(t-s)^d}
\frac{\|E(s)\|_{L^1_x}}{\langle s\rangle^{d-\gamma-1}}\,ds +
\frac{\varepsilon |\alpha|^a}{\langle t\rangle^a}
\int_{\frac{t}{2}}^{t}
\frac{\|E(s)\|_{L^\infty_x}}{\langle s\rangle^{d-\gamma-1}}\,ds.
\end{align*}
Invoking \eqref{ziyouliu01} once more, we conclude that
\[
\|\mathcal{D}_4(t)\|_{L^\infty_x}
\lesssim
\frac{\varepsilon^2 |\alpha|^a}{\langle t\rangle^{d+a}}
\int_0^{\frac{t}{2}}\frac{ds}{\langle s\rangle^{d-2\gamma}}
+\frac{\varepsilon^2 |\alpha|^a}{\langle t\rangle^{a}}
\int_{\frac{t}{2}}^{t}\frac{ds}{\langle s\rangle^{2d-2\gamma}}
\lesssim
\frac{\varepsilon^2 |\alpha|^a}{\langle t\rangle^{d+a}}.
\]

Combining the bounds for $\mathcal{D}_1$--$\mathcal{D}_4$, and using that \(0<\varepsilon<1\) and \(0<a<1\), so that
$\varepsilon^2\le \varepsilon^{a+1}$,
we infer that, for \(|\alpha|<\frac{t}{2}\),
\begin{equation}\label{zuoyong028-pp}
\frac{\|\delta^x_\alpha R^\pm(t)\|_{L^p_x}}{|\alpha|^a}
\lesssim
\frac{\varepsilon^{a+1}}{\langle t\rangle^{\frac{d(p-1)}{p}+a}},
\qquad p\in\{1,\infty\}.
\end{equation}
By \eqref{zuoyong028p}, \eqref{zuoyong028-pp}, and \eqref{Besov}, we obtain
\begin{equation}\label{zuoyongx009}
\langle t\rangle^a\|R^\pm(t)\|_{\dot B^a_{1,\infty}}
+
\langle t\rangle^{d+a}\|R^\pm(t)\|_{\dot B^a_{\infty,\infty}}
\lesssim
\varepsilon^{a+1}.
\end{equation}

\medskip
\textit{Step 3. Norm estimates for the charge density.}
To conclude the proof, we establish the temporal decay of $\rho(t)$ in the spaces: $\dot{B}_{\infty,\infty}^a$, $\dot{B}_{1,\infty}^a$, $L^\infty$, and $L^1$.

Let
$$
\mathcal{Z}(t,x):=(I^+-I^-)(t,x)-(R^++R^-)(t,x).
$$
Then \eqref{fangcheng000} can be rewritten as
\begin{equation}\label{jian04}
\rho=-\mathcal G*_{(t,x)}\mathcal{Z}+\mathcal{Z}.
\end{equation}

We first estimate $\mathcal{Z}$.
By \eqref{chuzhi0082} and \eqref{zuoyong028},
\begin{equation}\label{midu-hbar-L1}
\|\mathcal{Z}(t)\|_{L_x^1}\lesssim \varepsilon_0+\varepsilon+\varepsilon^2.
\end{equation}

By \eqref{chuzhi008} and \eqref{zuoyong028},
%\[
%\langle t\rangle^d\|I(t)\|_{L_x^\infty}\lesssim \varepsilon+\varepsilon_0,
%\qquad
%\langle t\rangle^d\|R^\pm(t)\|_{L_x^\infty}\lesssim \varepsilon^2,
%\]
%and therefore
\begin{equation}\label{midu-hbar-Linf}
\langle t\rangle^d\|\mathcal{Z}(t)\|_{L_x^\infty}
\lesssim
\varepsilon_0+\varepsilon+\varepsilon^2.
\end{equation}

Moreover, by \eqref{chuzhi076} and \eqref{zuoyongx009},
%\[
%\langle t\rangle^a\|I(t)\|_{\dot B_{1,\infty}^a}
%+\langle t\rangle^{d+a}\|I(t)\|_{\dot B_{\infty,\infty}^a}
%\lesssim \varepsilon+\varepsilon_0,
%\]
%and
%\[
%\langle t\rangle^a\|R^\pm(t)\|_{\dot B_{1,\infty}^a}
%+\langle t\rangle^{d+a}\|R^\pm(t)\|_{\dot B_{\infty,\infty}^a}
%\lesssim \varepsilon^{a+1}.
%\]
%Thus
\begin{equation}\label{midu-hbar-Besov}
\langle t\rangle^a\|\mathcal{Z}(t)\|_{\dot B_{1,\infty}^a}
+\langle t\rangle^{d+a}\|\mathcal{Z}(t)\|_{\dot B_{\infty,\infty}^a}
\lesssim
\varepsilon_0+\varepsilon+\varepsilon^{a+1}.
\end{equation}

We now estimate the convolution term.
Applying \eqref{lemma5.2} to $\mathcal{Z}$, together with
\eqref{midu-hbar-L1}, \eqref{midu-hbar-Linf}, and \eqref{midu-hbar-Besov}, yields
\[
\|(\mathcal G*_{(t,x)}\mathcal{Z})(t)\|_{L_x^1}
\lesssim
\sup_{0\le s\le t}\langle s\rangle^a\|\mathcal{Z}(s)\|_{\dot B_{1,\infty}^a}
\lesssim
\varepsilon_0+\varepsilon+\varepsilon^{a+1},
\]
\[
\langle t\rangle^d\|(\mathcal G*_{(t,x)}\mathcal{Z})(t)\|_{L_x^\infty}
\lesssim
\sup_{0\le s\le t}
\Big(
\|\mathcal{Z}(s)\|_{L_x^1}
+\langle s\rangle^{d+a}\|\mathcal{Z}(s)\|_{\dot B_{\infty,\infty}^a}
\Big)
\lesssim
\varepsilon_0+\varepsilon+\varepsilon^2+\varepsilon^{a+1},
\]
\[
\langle t\rangle^a\|(\mathcal G*_{(t,x)}\mathcal{Z})(t)\|_{\dot B_{1,\infty}^a}
\lesssim
\sup_{0\le s\le t}\langle s\rangle^a\|\mathcal{Z}(s)\|_{\dot B_{1,\infty}^a}
\lesssim
\varepsilon_0+\varepsilon+\varepsilon^{a+1},
\]
and
\[
\langle t\rangle^{d+a}\|(\mathcal G*_{(t,x)}\mathcal{Z})(t)\|_{\dot B_{\infty,\infty}^a}
\lesssim
\sup_{0\le s\le t}
\Big(
\langle s\rangle^a\|\mathcal{Z}(s)\|_{\dot B_{1,\infty}^a}
+\langle s\rangle^{d+a}\|\mathcal{Z}(s)\|_{\dot B_{\infty,\infty}^a}
\Big)
\lesssim
\varepsilon_0+\varepsilon+\varepsilon^{a+1}.
\]

Combining the above estimates with \eqref{jian04},
and using the triangle inequality, we obtain
\[
\|\rho(t)\|_{L_x^1}
+\langle t\rangle^d\|\rho(t)\|_{L_x^\infty}
+\langle t\rangle^a\|\rho(t)\|_{\dot B_{1,\infty}^a}
+\langle t\rangle^{d+a}\|\rho(t)\|_{\dot B_{\infty,\infty}^a}
\lesssim
\varepsilon_0+\varepsilon+\varepsilon^2+\varepsilon^{a+1}.
\]
Since $0<\varepsilon<1$ and $a \in \big(0, 1 \big)$ , we have
\[
\varepsilon^2+\varepsilon^{a+1}\le 2\varepsilon.
\]
Hence \eqref{miduguji-04} follows. This completes the proof.
\end{proof}

\subsection{Estimates for the Charge Density and its Gradient}
\begin{proposition}\label{daoshu009}
For $t\geq1$, one has
\begin{equation}\label{Pro5.6}
\langle t\rangle\|\nabla_x\rho(t)\|_{L^1_x}+\langle t\rangle^{d+1}\|\nabla_x\rho(t)\|_{L^\infty_x}\lesssim(\varepsilon_0+\varepsilon) \ln(1+t).
\end{equation}
\end{proposition}

\begin{proof}
\textit{Step 1. Estimates for  $\nabla_x I$.}
Differentiating \eqref{chuzhi0123} with respect to $x$, we write
$$
\nabla_x I(t,x) = \nabla_x I_1(t,x) + \nabla_x I_2(t,x) + \nabla_x I_3(t,x).
$$

\textit{For $\nabla_x I_1$.}
Set
$$
\mathcal{U}^+(w):=w+Y^+_{0,t}\Big(w,\frac{x-w}{t}\Big),
\quad
\mathcal{N}_\theta(w):=\frac{x-w}{t}
+\theta W^+_{0,t}\Big(w,\frac{x-w}{t}\Big)
+(1-\theta)W^-_{0,t}\Big(w,\frac{x-w}{t}\Big).
$$
Then
$$
I_1(t,x)
=\frac1{t^d}\int_0^1\int_{\mathbb R^d}
\nabla_v f_{\mathrm{in}}^+\big(\mathcal{U}^+(w),\mathcal{N}_\theta(w)\big)
\cdot
\Big(W_{0,t}^+-W_{0,t}^-\Big)\Big(w,\frac{x-w}{t}\Big)
\,dw\,d\theta.
$$
By the chain rule, for each $i\in\{1,\dots,d\}$,
\begin{align*}
|\partial_{x_i} I_1(t,x)|
\lesssim{}&
\frac1{t^d}\int_0^1\int_{\mathbb R^d}
\big|\nabla_x\nabla_v f_{\mathrm{in}}^+\big(\mathcal{U}^+(w),\mathcal{N}_\theta(w)\big)\big|
\,\big|\partial_{x_i}\mathcal{U}^+(w)\big|
\,\big|W_{0,t}^+-W_{0,t}^-\big|\Big(w,\frac{x-w}{t}\Big)
\,dw\,d\theta \\
&+
\frac1{t^d}\int_0^1\int_{\mathbb R^d}
\big|\nabla_v^2 f_{\mathrm{in}}^+\big(\mathcal{U}^+(w),\mathcal{N}_\theta(w)\big)\big|
\,\big|\partial_{x_i}\mathcal{N}_\theta(w)\big|
\,\big|W_{0,t}^+-W_{0,t}^-\big|\Big(w,\frac{x-w}{t}\Big)
\,dw\,d\theta \\
&+
\frac1{t^d}\int_0^1\int_{\mathbb R^d}
\big|\nabla_v f_{\mathrm{in}}^+\big(\mathcal{U}^+(w),\mathcal{N}_\theta(w)\big)\big|
\,\big|\partial_{x_i}(W_{0,t}^+-W_{0,t}^-)\big|\Big(w,\frac{x-w}{t}\Big)
\,dw\,d\theta.
\end{align*}
%A direct application of the chain rule yields
%\begin{align*}
%&\partial_{x_i}I_1(t,x)\\
%=&\frac{1 }{t^d}\int_0^1\int_{\mathbb{R}^d}\nabla_{x}\nabla_vf^+_{\mathrm{in}}
%\Big(w+Y^+_{0,t}(w,\frac{x-w}{t})
%\ ,\frac{x-w}{t}+\Big(\t W^+_{0,t}
%+(1-\t)W_{0,t}^-\Big)(w,\frac{x-w}{t})\Big)\\
%&\quad \cdot\partial_{x_i}\Big(w+Y^+_{0,t}(w,\frac{x-w}{t})
%\Big) \  \Big(W^+_{0,t}-W_{0,t}^-\Big)(w,\frac{x-w}{t})
%\,dw \,d\theta\\
%&+\frac{1}{t^d}\int_0^1\int_{\mathbb{R}^d}\nabla^2_vf^+_{\mathrm{in}}
%\Big(w+Y^+_{0,t}(w,\frac{x-w}{t})
%\ ,\frac{x-w}{t}+\Big(\t W^+_{0,t}
%+(1-\t)W_{0,t}^-\Big)(w,\frac{x-w}{t})\Big)\\
%&\quad \cdot\partial_{x_i}\Big(\frac{x-w}{t}+\t W^+_{0,t}(w,\frac{x-w}{t})+(1-\t)W_{0,t}^-(w,\frac{x-w}{t})\Big)
% \Big(W^+_{0,t}-W_{0,t}^-\Big)(w,\frac{x-w}{t})
%\,dw\, d\theta\\
%&+\frac{1}{t^d}\int_0^1\int_{\mathbb{R}^d}\nabla_vf^+_{\mathrm{in}}
%\Big(w+Y^+_{0,t}(w,\frac{x-w}{t})
%\ ,\frac{x-w}{t}+\Big(\t W^+_{0,t}
%+(1-\t)W_{0,t}^-\Big)(w,\frac{x-w}{t})\Big)\\
%&\quad \cdot\partial_{x_i}\Big(W^+_{0,t}
%-W_{0,t}^-\Big)(w,\frac{x-w}{t})\,dw \,d\theta.
%\end{align*}
Since $w$ is independent of $x$, we have
\begin{equation}\label{tezheng-045}
\partial_{x_i}\Big[Y^\pm_{0,t}\Big(w,\frac{x-w}{t}\Big)\Big]
=
\frac1t\,\partial_{v_i}Y^\pm_{0,t}\Big(w,\frac{x-w}{t}\Big),
\partial_{x_i}\Big[W^\pm_{0,t}\Big(w,\frac{x-w}{t}\Big)\Big]
=
\frac1t\,\partial_{v_i}W^\pm_{0,t}\Big(w,\frac{x-w}{t}\Big).
\end{equation}
It then follows from \eqref{cor:estimates001-difference} and Lemma \ref{estimates001} that
$$
|W_{0,t}^+-W_{0,t}^-|\lesssim \varepsilon,
\quad
|\partial_{x_i}\mathcal{U}^+(w)|\lesssim \frac{\varepsilon}{t},
\quad
|\partial_{x_i}\mathcal{N}^+(w)|\lesssim \frac1t,
\quad
\big|\partial_{x_i}(W_{0,t}^+-W_{0,t}^-)\big|\lesssim \frac{\varepsilon}{t}.
$$
Hence
\begin{align*}
|\nabla_x I_1(t,x)|
\lesssim&
\frac{\varepsilon^2}{t^{d+1}}
\int_0^1\int_{\mathbb R^d}
\big|\nabla_x\nabla_v f_{\mathrm{in}}^+\big(\mathcal{U}^+(w),\mathcal{N}_\theta(w)\big)\big|
\,dw\,d\theta \\
&+\frac{\varepsilon}{t^{d+1}}
\int_0^1\int_{\mathbb R^d}
\big|\nabla_v^2 f_{\mathrm{in}}^+\big(\mathcal{U}^+(w),\mathcal{N}_\theta(w)\big)\big|
\,dw\,d\theta \\
&+\frac{\varepsilon}{t^{d+1}}
\int_0^1\int_{\mathbb R^d}
\big|\nabla_v f_{\mathrm{in}}^+\big(\mathcal{U}^+(w),\mathcal{N}_\theta(w)\big)\big|
\,dw\,d\theta.
\end{align*}
Using the change of variables $w=x-tv$, together with \eqref{z4} and the definition of $\Xi_{\mathcal V,\theta}(0)$ in \eqref{tezheng-V}, we obtain
\begin{align*}
|\nabla_xI_1(t,x)|
\lesssim&\ \frac{\varepsilon^2 }{t}\int_0^1\int_{\mathbb{R}^d}
\left|\nabla_{x}\nabla_vf^+_{\mathrm{in}}\left(\mathcal{X}^+(0),
\Xi_{\mathcal{V},\theta}(0)\right)\right| \,dv\, d\theta \nonumber\\
&+\frac{\varepsilon }{t}\int_0^1\int_{\mathbb{R}^d}
\left|\nabla^2_vf^+_{\mathrm{in}}\left(\mathcal{X}^+(0),\Xi_{\mathcal{V},\theta}(0)
\right)\right|\,dv \,d\theta \nonumber\\
&+\frac{\varepsilon}{t}\int_0^1\int_{\mathbb{R}^d}
\left|\nabla_vf^+_{\mathrm{in}}\left(\mathcal{X}^+(0),\Xi_{\mathcal{V},\theta}(0)\right)\right|\,dv\, d\theta.
%:=&\ \mathcal{I}_{11}(t,x)+\mathcal{I}_{12}(t,x)+\mathcal{I}_{13}(t,x),
\end{align*}
Applying Lemma~\ref{lemma:J123-H} with $\mathcal L_1$ to the three terms above, with
\[
\mathcal{H}=\nabla_x\nabla_v f_{\mathrm{in}}^+,\qquad
\nabla_v^2 f_{\mathrm{in}}^+,\qquad
\nabla_v f_{\mathrm{in}}^+,
\]
respectively. Since \(0<\varepsilon<1\), the factor \(\varepsilon^2\) may be absorbed into \(\varepsilon\). Therefore,
\[
\|\nabla_x I_1(t)\|_{L^1_x}
\lesssim\frac{\v}{\langle t\rangle}\left(\|\langle v\rangle^k\nabla_x\nabla_vf^+_{\mathrm{in}}\|_{L_x^1L_v^\infty}
+\|\langle v\rangle^k\nabla^2_vf^+_{\mathrm{in}}\|_{L_x^1L_v^\infty}
+\|\langle v\rangle^k\nabla_vf^+_{\mathrm{in}}\|_{L_x^1L_v^\infty}\right),
\]
and
\[
\|\nabla_x I_1(t)\|_{L^\infty_x}
\lesssim\frac{\v}{\langle t\rangle^{d+1}}\Big(\|\nabla_x\nabla_vf^+_{\mathrm{in}}\|_{L_x^1L_v^\infty}
+\|\nabla^2_vf^+_{\mathrm{in}}\|_{L_x^1L_v^\infty}
+\|\nabla_vf^+_{\mathrm{in}}\|_{L_x^1L_v^\infty}\Big).
\]

\textit{For $\nabla_x I_2$.}
Set
\[
\mathcal{U}_\theta(w):=
w+\theta Y^+_{0,t}\Big(w,\frac{x-w}{t}\Big)
+(1-\theta)Y^-_{0,t}\Big(w,\frac{x-w}{t}\Big),
\quad
\mathcal{N}^-(w):=
\frac{x-w}{t}+W^-_{0,t}\Big(w,\frac{x-w}{t}\Big).
\]
Then
\[
I_2(t,x)
=
\frac1{t^d}\int_0^1\int_{\mathbb R^d}
\nabla_x f_{\mathrm{in}}^+\big(\mathcal{U}_\theta(w),\mathcal{N}^-(w)\big)
\cdot
\big(Y_{0,t}^+-Y_{0,t}^-\big)\Big(w,\frac{x-w}{t}\Big)
\,dw\,d\theta.
\]
Similarly, by the chain rule,
\begin{align*}
|\partial_{x_i}I_2(t,x)|
\lesssim{}&
\frac1{t^d}\int_0^1\int_{\mathbb R^d}
\big|\nabla_x^2 f_{\mathrm{in}}^+\big(\mathcal{U}_\theta(w),\mathcal{N}^-(w)\big)\big|
\,\big|\partial_{x_i}\mathcal{U}_\theta(w)\big|
\,\big|Y_{0,t}^+-Y_{0,t}^-\big|\Big(w,\frac{x-w}{t}\Big)
\,dw\,d\theta \\
&+
\frac1{t^d}\int_0^1\int_{\mathbb R^d}
\big|\nabla_x\nabla_v f_{\mathrm{in}}^+\big(\mathcal{U}_\theta(w),\mathcal{N}^-(w)\big)\big|
\,\big|\partial_{x_i}\mathcal{N}^-(w)\big|
\,\big|Y_{0,t}^+-Y_{0,t}^-\big|\Big(w,\frac{x-w}{t}\Big)
\,dw\,d\theta \\
&+
\frac1{t^d}\int_0^1\int_{\mathbb R^d}
\big|\nabla_x f_{\mathrm{in}}^+\big(\mathcal{U}_\theta(w),\mathcal{N}^-(w)\big)\big|
\,\big|\partial_{x_i}(Y_{0,t}^+-Y_{0,t}^-)\big|\Big(w,\frac{x-w}{t}\Big)
\,dw\,d\theta.
\end{align*}
By \eqref{cor:estimates001-difference} and Lemma \ref{estimates001},
\[
\big|Y_{0,t}^+-Y_{0,t}^-\big|\lesssim \varepsilon, \quad
\big|\partial_{x_i}\mathcal U_\theta(w)\big|
\lesssim \frac{\varepsilon}{t},\quad
\big|\partial_{x_i}\mathcal N^-(w)\big|
\lesssim \frac{1}{t},\quad
\big|\partial_{x_i}(Y_{0,t}^+-Y_{0,t}^-)\big|
\lesssim \frac{\varepsilon}{t}.
\]
Consequently,
\begin{align*}
|\nabla_x I_2(t,x)|
\lesssim{}&
\frac{\varepsilon}{t^{d+1}}
\int_0^1\int_{\mathbb R^d}
\big|\nabla_x^2 f_{\mathrm{in}}^+\big(\mathcal{U}_\theta(w),\mathcal{N}^-(w)\big)\big|
\,dw\,d\theta \\
&+
\frac{\varepsilon}{t^{d+1}}
\int_0^1\int_{\mathbb R^d}
\big|\nabla_x\nabla_v f_{\mathrm{in}}^+\big(\mathcal{U}_\theta(w),\mathcal{N}^-(w)\big)\big|
\,dw\,d\theta \\
&+
\frac{\varepsilon}{t^{d+1}}
\int_0^1\int_{\mathbb R^d}
\big|\nabla_x f_{\mathrm{in}}^+\big(\mathcal{U}_\theta(w),\mathcal{N}^-(w)\big)\big|
\,dw\,d\theta.
\end{align*}
Using the change of variables $w=x-tv$, together with \eqref{z4} and the definition of $\Xi_{\mathcal X,\theta}(0)$, we obtain
\begin{align*}
|\nabla_{x}I_2(t,x)|
\lesssim&\ \frac{\varepsilon}{t}\int_0^1\int_{\mathbb{R}^d}
\left|\nabla^2_{x}f^+_{\mathrm{in}}\left(\Xi_{\mathcal{X},\theta}(0), \mathcal{V}^-(0)\right)\right|\,dv \,d\theta \nonumber\\
&+\frac{\varepsilon}{t}\int_0^1\int_{\mathbb{R}^d}
\left|\nabla_{x}\nabla_vf^+_{\mathrm{in}}\left(\Xi_{\mathcal{X},\theta}(0), \mathcal{V}^-(0)\right)\right| \,dv\, d\theta \nonumber\\
&+\frac{\varepsilon}{t}\int_0^1\int_{\mathbb{R}^d}
\left|\nabla_xf^+_{\mathrm{in}}\left(\Xi_{\mathcal{X},\theta}(0), \mathcal{V}^-(0)\right)\right|\,dv\, d\theta.
%:=&\ \mathcal{I}_{21}(t,x)+\mathcal{I}_{22}(t,x)+\mathcal{I}_{23}(t,x),
\end{align*}
Applying Lemma~\ref{lemma:J123-H} with $\mathcal L_2$ to the three terms above, with
\[
\mathcal H=
\nabla_x^2 f^+_{\mathrm{in}},\quad
\nabla_x\nabla_v f^+_{\mathrm{in}},\quad
\nabla_x f^+_{\mathrm{in}},
\]
respectively, and using $t\sim\langle t\rangle$ for $t\ge1$, we obtain
\[
\|\nabla_xI_{2}(t)\|_{L^1_x}\lesssim\frac{\v}{\langle t\rangle}\Big(\|\langle v\rangle^k\nabla_x^2f^+_{\mathrm{in}}\|_{L_x^1L_v^\infty}
+\|\langle v\rangle^k\nabla_x\nabla_vf^+_{\mathrm{in}}\|_{L_x^1L_v^\infty}
+\|\langle v\rangle^k\nabla_xf^+_{\mathrm{in}}\|_{L_x^1L_v^\infty}\Big),
\]
and
\[
\|\nabla_x I_2(t)\|_{L^\infty_x}
\lesssim\frac{\v}{\langle t\rangle^{d+1}}\Big(\|\nabla_x^2f^+_{\mathrm{in}}\|_{L_x^1L_v^\infty}
+\|\nabla_x\nabla_vf^+_{\mathrm{in}}\|_{L_x^1L_v^\infty}
+\|\nabla_xf^+_{\mathrm{in}}\|_{L_x^1L_v^\infty}\Big).
\]

\textit{For $\nabla_x I_3$.}
Differentiating with respect to $x_i$ and applying the chain rule, we obtain
\begin{align*}
|\partial_{x_i}I_3(t,x)|
\lesssim{}&
\frac1{t^d}\int_{\mathbb R^d}
\big|\nabla_x f_{\mathrm{in}}\big(\mathcal{U}^-(w),\mathcal{N}^-(w)\big)\big|
\,\big|\partial_{x_i}\mathcal{U}^-(w)\big|
\,dw \\
&+\frac1{t^d}\int_{\mathbb R^d}
\big|\nabla_v f_{\mathrm{in}}\big(\mathcal{U}^-(w),\mathcal{N}^-(w)\big)\big|
\,\big|\partial_{x_i}\mathcal{N}^-(w)\big|
\,dw,
\end{align*}
where
$$
\mathcal{U}^-(w):=w+Y^-_{0,t}\Big(w,\frac{x-w}{t}\Big),
\quad
\mathcal{N}^-(w):=\frac{x-w}{t}+W^-_{0,t}\Big(w,\frac{x-w}{t}\Big).
$$
%Using again
%$$
%\partial_{x_i}\Big[Y^-_{0,t}\Big(w,\frac{x-w}{t}\Big)\Big]
%=\frac1t\,\partial_{v_i}Y^-_{0,t}\Big(w,\frac{x-w}{t}\Big),
%\quad
%\partial_{x_i}\Big[W^-_{0,t}\Big(w,\frac{x-w}{t}\Big)\Big]
%=\frac1t\,\partial_{v_i}W^-_{0,t}\Big(w,\frac{x-w}{t}\Big),
%$$
By \eqref{tezheng-045} and \eqref{cor:estimates001-difference},
\[
|\nabla_x I_3(t,x)|
\lesssim
\frac1{t^{d+1}}\int_{\mathbb R^d}
\big|\nabla_x f_{\mathrm{in}}\big(\mathcal{U}^-(w),\mathcal{N}^-(w)\big)\big|\,dw
+
\frac1{t^{d+1}}\int_{\mathbb R^d}
\big|\nabla_v f_{\mathrm{in}}\big(\mathcal{U}^-(w),\mathcal{N}^-(w)\big)\big|\,dw.
\]
Changing variables $w=x-tv$ and using \eqref{z4}, we get
\[
|\nabla_x I_3(t,x)|
\lesssim
\frac1t\int_{\mathbb R^d}
\big|\nabla_x f_{\mathrm{in}}\big(\mathcal X^-(0),\mathcal V^-(0)\big)\big|\,dv
+
\frac1t\int_{\mathbb R^d}
\big|\nabla_v f_{\mathrm{in}}\big(\mathcal X^-(0),\mathcal V^-(0)\big)\big|\,dv.
\]
Therefore, applying Lemma~\ref{lemma:J123-H} with $\mathcal L_3$ to the two terms above, with
$$
\mathcal H_3=\nabla_x f_{\mathrm{in}}, \quad \nabla_v f_{\mathrm{in}},
$$
respectively, we obtain
\[
\|\nabla_x I_3(t)\|_{L_x^1}
\lesssim
\frac1{\langle t\rangle}
\Big(
\|\nabla_x f_{\mathrm{in}}\|_{L_x^1L_v^1}
+
\|\nabla_v f_{\mathrm{in}}\|_{L_x^1L_v^1}
\Big),
\]
and
\[
\|\nabla_x I_3(t)\|_{L_x^\infty}
\lesssim
\frac1{\langle t\rangle^{d+1}}
\Big(
\|\nabla_x f_{\mathrm{in}}\|_{L_x^1L_v^\infty}
+
\|\nabla_v f_{\mathrm{in}}\|_{L_x^1L_v^\infty}
\Big).
\]
Combining the estimates for $\nabla_x I_1$, $\nabla_x I_2$, and $\nabla_x I_3$ with \eqref{large-data} and \eqref{qn}, and using the triangle inequality, we obtain
\begin{equation}\label{chizhi008}
\langle t\rangle\|\nabla_x I(t)\|_{L_x^1}
+
\langle t\rangle^{d+1}\|\nabla_x I(t)\|_{L_x^\infty}
\lesssim
\varepsilon_0+\varepsilon.
\end{equation}

\medskip
\textit{Step 2.  Estimates for $\nabla_x R^{\pm}$.}
%We now derive $L^p$ bounds for $\nabla_x R^{\pm}(t)$ with $p \in \{1, \infty\}$.
Differentiate \eqref{zuoyong01} with respect to $x_i$ and then using the change of variables $w=x-tv$, we obtain, in view of \eqref{z4},
\begin{align*}
\partial_{x_i}R^{\pm}(t,x)
=&\mathcal{R}_1(t,x)+\mathcal{R}_2(t,x)+\mathcal{R}_3(t,x)+\mathcal{R}_4(t,x),
\end{align*}
where
\begin{align*}
\mathcal{R}_1(t,x)&:=\mp\frac{1}{t}\int_{0}^{t}\int_{\mathbb{R}^{d}}
s\Big[\partial_{x_i}E\big(s,\mathcal{X}^{\pm}(s)\big)\cdot\nabla_{v}
\mu(\mathcal{V}^{\pm}(s))-
\partial_{x_i}E\big(s,x-(t-s)v\big)\cdot\nabla_v\mu(v)\Big]\,dv\,ds,\\
\mathcal{R}_2(t,x)&:=\mp \frac{1}{t}\int_{0}^{t}\int_{\mathbb{R}^{d}}
\Big[E(s,\mathcal{X}^{\pm}(s))\cdot\partial_{v_i}\nabla_{v}\mu(\mathcal{V}^{\pm}(s))
- E(s,x-(t-s)v)\cdot\partial_{v_i}\nabla_{v}\mu(v)\Big]\,dv\,ds,\\
\mathcal{R}_3(t,x)&:=\mp\frac{1}{t}\int_{0}^{t}\int_{\mathbb{R}^{d}}
\Big(\partial_{v_i}Y_{s,t}^{\pm}(x-tv,v)\cdot\nabla_{x}E(s,\mathcal{X}^{\pm}(s))\Big)
\cdot\nabla_{v}\mu(\mathcal{V}^{\pm}(s))\,dv\,ds,\\
\mathcal{R}_4(t,x)&:=\mp\frac{1}{t}\int_{0}^{t}\int_{\mathbb{R}^{d}}
E(s,\mathcal{X}^{\pm}(s))\cdot\Big(\partial_{v_i}W_{s,t}^{\pm}(x-tv,v)\cdot\nabla_{v}^{2}\mu(\mathcal{V}^{\pm}(s))
\Big)\,dv\,ds.
\end{align*}

\textit{For $\mathcal{R}_1$ and $\mathcal{R}_2$.}
We may rewrite $\mathcal{R}_1$ and $\mathcal{R}_2$ as
\[
\mathcal{R}_1(t,x)=\frac{1}{t}\Gamma^\pm[s\partial_{x_i}E,\mu](t,x),
\qquad
\mathcal{R}_2(t,x)=\frac1t \Gamma^\pm[E,\partial_{v_i}\mu](t,x).
\]
Therefore, by \eqref{zuoyong08},
$$
\|\mathcal{R}_1(t)\|_{L^1_x}+\|\mathcal{R}_2(t)\|_{L^1_x}
\lesssim\frac{\varepsilon}{t}\int_0^t
\frac{\| s\nabla_xE(s)\|_{L^1_x}}{\langle s\rangle^{d-\gamma-1}}\,ds+\frac{\varepsilon}{t}\int_0^t
\frac{\|E(s)\|_{L^1_x}}{\langle s\rangle^{d-\gamma-1}}\,ds.
$$
Using \eqref{ziyouliu01}, we obtain
\begin{align*}
\|\mathcal{R}_1(t)\|_{L^1_x}+\|\mathcal{R}_2(t)\|_{L^1_x}
\lesssim \frac{\varepsilon^2}{t}
\int_0^t\frac{ds}{\langle s\rangle^{d-2\gamma-1}}
\lesssim \frac{\varepsilon^2}{\langle t\rangle}.
\end{align*}
For the $L^\infty_x$ estimate, we split $[0,t]=[0,\frac{t}{2}]\cup[\frac{t}{2},t]$. By \eqref{zuoyong08},
$$
\|\mathcal{R}_1(t)\|_{L^\infty_x}
%&\lesssim \frac{\varepsilon}{t}\int_0^{\frac{t}{2}}
%\frac{\|s\nabla_xE(s)\|_{L^1_x}}{\langle s\rangle^{d-\gamma-1}}
%\frac{ds}{(t-s)^d}
%+\frac{\varepsilon}{t}\int_{\frac{t}{2}}^{t}
%\frac{\|s\nabla_xE(s)\|_{L^\infty_x}}{\langle s\rangle^{d-\gamma-1}}\,ds\\
\lesssim \frac{\varepsilon}{t^{d+1}}\int_0^{\frac{t}{2}}
\frac{s\,\|\nabla_xE(s)\|_{L^1_x}}{\langle s\rangle^{d-\gamma-1}}\,ds
+\frac{\varepsilon}{t}\int_{\frac{t}{2}}^{t}
\frac{s\|\nabla_xE(s)\|_{L^\infty_x}}{\langle s\rangle^{d-\gamma-1}}\,ds,
$$
$$
\|\mathcal{R}_2(t)\|_{L^\infty_x}\lesssim \frac{\varepsilon}{t^{d+1}}\int_0^{\frac{t}{2}}
\frac{\|E(s)\|_{L^1_x}}{\langle s\rangle^{d-\gamma-1}}\,ds
+\frac{\varepsilon}{t}\int_{\frac{t}{2}}^{t}
\frac{\|E(s)\|_{L^\infty_x}}{\langle s\rangle^{d-\gamma-1}}\,ds.
$$
Hence, by \eqref{ziyouliu01} and the smallness of $\gamma$,
$$
\|\mathcal{R}_1(t)\|_{L^\infty_x}+\|\mathcal{R}_2(t)\|_{L^\infty_x}
\lesssim\frac{\varepsilon^2}{\langle t\rangle^{d+1}}.
$$

\textit{For $\mathcal{R}_3$ and $\mathcal{R}_4$.}
By \eqref{tezhengY},
\begin{equation}\label{zuoyou-R3}
|\mathcal{R}_3(t,x)|\lesssim
\frac{\varepsilon}{t}\int_{0}^{t}\int_{\mathbb{R}^{d}}
\frac{|\nabla_{x}E(s,\mathcal{X}^{\pm}(s))|}{\langle s\rangle^{d-\gamma-2}}
|\nabla_{v}\mu(\mathcal{V}^{\pm}(s))|\,dv\,ds.
\end{equation}
Integrating in $x$, using the change of variables
$y=\mathcal{X}^\pm(s)$ and $w=\mathcal{V}^\pm(s)$,
and using Lemma \ref{estimates001}
together with the decay of $\nabla_v\mu$, we obtain
$$
\|\mathcal{R}_3(t)\|_{L^1_x}\lesssim\frac{\varepsilon}{t}\int_{0}^{t}
\frac{\|\nabla_{x}E(s)\|_{L^1_x}}{\langle s\rangle^{d-\gamma-2}}ds
\lesssim\frac{\varepsilon^2}{\langle t\rangle}.
$$
To estimate the $L^\infty_x$ norm, we use the change of variables furnished by Lemma~\ref{proposition of psi and phi}, together with the decay of $\nabla_v\mu$, to obtain
$$
|\mathcal{R}_3(t,x)|\lesssim \frac{\varepsilon}{t}\int_{0}^{t}\int_{\mathbb{R}^{d}}\frac{|\nabla_{x}E(s,{w})
|}{\langle s\rangle^{d-\gamma-2}}\Big(1+\Big|\frac{x-w}{t-s}\Big|\Big)^{-N}
\frac{dw\,ds}{(t-s)^{d}}.
$$
Moreover, since
$$
\int_{\mathbb{R}^{d}} |\nabla_{v}\mu(\mathcal{V}^{\pm}(s))|dv
\lesssim\int_{\mathbb{R}^{d}}\|\langle v\rangle^{N}\nabla_{v}\mu\|_{L_{v}^{\infty}}
\langle \mathcal{V}^{\pm}(s)\rangle^{-N}dv\lesssim1,
$$
it follows from \eqref{zuoyou-R3} that
$$
|\mathcal{R}_3(t,x)|\lesssim \frac{\varepsilon}{t}\int_{0}^{t}
\frac{\|\nabla_{x}E(s)\|_{L^\infty_x}}{\langle s\rangle^{d-\gamma-2}}\,ds.
$$
%\begin{align*}
%|\mathcal{R}_3(t,x)|=&\ \frac{1}{t}\Big|\int_{0}^{t}\int_{\mathbb{R}^{d}}\nabla_ xE\Big(s,\mathcal{X}^{\pm}(s;t,x,\Psi_{s,t}^{\pm}(x,v))\Big)
%\cdot\nabla_{v}Y_{s,t}^{\pm}\Big(x-t\Psi_{s,t}^{\pm}(x,v),
%{\Psi}_{s,t}^{\pm}(x,v)\Big)\nonumber\\
%&\nabla_v \mu\Big(\mathcal{V}^{\pm}(s;t,x,\Psi^\pm_{s,t}(x,v))\Big)
%\det(\nabla_{v}\Psi_{s,t}^{\pm}(x,v))\,dv\,ds\Big|\nonumber\\
%\lesssim&\ \frac{1}{t}\int_{0}^{t}\int_{\mathbb{R}^d}|\nabla_ xE(s,x-(t-s)v)|\|\nabla_ vY_{s,t}^{\pm}\|_{L_{x,v}^{\infty}}\frac{dv\,ds}{\langle v \rangle^{N}}
%\end{align*}
To recover the optimal decay, we split the time integral into
$[0,\frac{t}{2}]\cup[\frac{t}{2},t]$ and infer that
\begin{align*}
|\mathcal{R}_3(t,x)|
\lesssim&\ \frac{\varepsilon}{t}
\int_{0}^{\frac{t}{2}}\frac{\|\nabla_{x}E(s)\|_{L^1_x}}{\langle s \rangle ^{d-\gamma-2}}\,\frac{ds}{(t-s)^d}+\frac{\varepsilon}{t}\int_{\frac{t}{2}}^{t}
\frac{\|\nabla_{x}E(s)\|_{L^\infty_x}}{\langle s \rangle^{d-\gamma-2}}\,ds\\
\lesssim&\ \frac{\varepsilon^2}{t}
\int_{0}^{\frac{t}{2}}\frac{1}{\langle s \rangle ^{d-2\gamma-1}}\,\frac{ds}{(t-s)^d}+\frac{\varepsilon^2}{t}\int_{\frac{t}{2}}^{t}
\frac{1}{\langle s \rangle^{2d-2\gamma-1}}\,ds\\
\lesssim&\frac{\varepsilon^2}{\langle t\rangle^{d+1}}.
\end{align*}

The term $\mathcal{R}_4$ is estimated in exactly the same way, using the bound for
$\partial_{v_i}W_{s,t}^{\pm}$ from \eqref{tezhengY} and the decay of $\nabla_v^2\mu$.

Collecting the above estimates for $\mathcal{R}_1$--$\mathcal{R}_4$, we conclude that
\begin{equation}\label{zuoyong073}
\langle t\rangle\|\nabla_x R^{\pm}(t)\|_{L^1_x}
+\langle t\rangle^{d+1}\|\nabla_x R^{\pm}(t)\|_{L^\infty_x}
\lesssim \varepsilon^{2}.
\end{equation}

\medskip
\textit{Step 3. Estimates for the gradient of the charge density.}
Differentiating \eqref{jian04} with respect to $x$, we obtain
$$
\nabla_x\rho(t,x)
=-\nabla_x(\mathcal G*_{(t,x)}\mathcal{Z})(t,x)+\nabla_x\mathcal{Z}(t,x).
$$
Moreover, by \eqref{chizhi008} and \eqref{zuoyong073},
\begin{equation}\label{rho-hbar-grad}
\|\nabla_x\mathcal{Z}(t)\|_{L_x^1}
\lesssim
\frac{\varepsilon_0+\varepsilon+\varepsilon^2}{\langle t\rangle},
\quad
\|\nabla_x\mathcal{Z}(t)\|_{L_x^\infty}
\lesssim
\frac{\varepsilon_0+\varepsilon+\varepsilon^2}{\langle t\rangle^{d+1}}.
\end{equation}

We now apply \eqref{lemma5.2}  to $\mathcal{Z}$.
By \eqref{midu-hbar-L1} and \eqref{rho-hbar-grad},
\[
\sup_{0\le s\le t}
\Big(
\|\mathcal{Z}(s)\|_{L_x^1}
+\langle s\rangle \|\nabla_x\mathcal{Z}(s)\|_{L_x^1}+\langle s\rangle^{d+1}\|\nabla_x\mathcal{Z}(s)\|_{L_x^\infty}
\Big)
\lesssim
\varepsilon_0+\varepsilon+\varepsilon^2.
\]
Hence, by \eqref{lemma5.2}, it holds that
\[
\|\nabla_x(\mathcal G*_{(t,x)}\mathcal{Z})(t)\|_{L_x^1}
\lesssim
\frac{\ln(1+t)}{\langle t\rangle}\,
(\varepsilon_0+\varepsilon+\varepsilon^2),
\]
and
\[
\|\nabla_x(\mathcal G*_{(t,x)}\mathcal{Z})(t)\|_{L_x^\infty}
\lesssim
\frac{\ln(1+t)}{\langle t\rangle^{d+1}}\,
(\varepsilon_0+\varepsilon+\varepsilon^2).
\]

Combining  the preceding estimates with \eqref{rho-hbar-grad}, we obtain
\[
\|\nabla_x\rho(t)\|_{L_x^1}
\lesssim
\frac{\ln(1+t)}{\langle t\rangle}\,
(\varepsilon_0+\varepsilon+\varepsilon^2),\quad
\|\nabla_x\rho(t)\|_{L_x^\infty}
\lesssim
\frac{\ln(1+t)}{\langle t\rangle^{d+1}}\,
(\varepsilon_0+\varepsilon+\varepsilon^2).
\]
Since $0<\varepsilon<1$, we have $\varepsilon^2\le \varepsilon$, and thus
\eqref{Pro5.6} follows.
\end{proof}

%%%%%%%%%%%%%%%%%%%%%%%%%%%%%%%%%%%%%%%%%%%%%%%%

\section{Closure of the Bootstrap and Proof of Theorem~\ref{Theorem1}}\label{diliujie}

\noindent\textbf{Proof of Theorem~\ref{Theorem1}.}
We use the continuation framework introduced after Proposition~\ref{Local solutions}.
Thus, it remains to prove that the bootstrap assumptions \eqref{ziyouliu02}
can be strictly improved on their maximal interval of validity.

Possibly after decreasing the constant
\[
\varepsilon_0=\varepsilon_0(M_0,d,k,\mu)>0,
\]
from Proposition~\ref{Local solutions},
with additional dependence on $a$ and $\gamma$,
we fix $A\ge1$ and set
\begin{equation}\label{eq:S6-epsilon-choice}
\varepsilon:=A\varepsilon_0.
\end{equation}
Then
\begin{equation}\label{eq:S6-eps-ratio}
\varepsilon_0+\varepsilon
=(A+1)\varepsilon_0
=\Bigl(1+\frac1A\Bigr)\varepsilon
\le 2\varepsilon.
\end{equation}
We shall prove that, for $\varepsilon_0>0$ sufficiently small depending only on
$M_0$, $a$, $d$, $k$, and $\gamma$, such that
\begin{equation}\label{eq:improved-bootstrap}
\left\{
\begin{aligned}
\|\rho(t)\|_{L^\infty_x}
&\le \tfrac12\varepsilon \langle t\rangle^{-d+\gamma},\\
\|\rho(t)\|_{\dot B^a_{\infty,\infty}}
&\le \tfrac12\varepsilon \langle t\rangle^{-(d+a)+\gamma},\\
\|\rho(t)\|_{\dot B^a_{1,\infty}}
&\le \tfrac12\varepsilon \langle t\rangle^{-a+\gamma},\\
\|\nabla_x \rho(t)\|_{L^\infty_x}
&\le \tfrac12\varepsilon \langle t\rangle^{-(d+1)+\gamma},\\
\|\nabla_x \rho(t)\|_{L^1_x}
&\le \tfrac12\varepsilon \langle t\rangle^{-1+\gamma},
\end{aligned}
\right.
\end{equation}
for all $t\in[0,T)$.
Once \eqref{eq:improved-bootstrap} is established, a standard continuity argument contradicts the maximality of $T$, and therefore $T=\infty$.

We prove \eqref{eq:improved-bootstrap} separately on $[0,T_*]$ and $[T_*,T)$, where $T_*\ge1$ will be fixed below.

\medskip
\noindent
\textit{Step 1. The finite-time regime.}
By Proposition~\ref{xiaoshijian008}, for $t\in[0,T_*]$,
$$
\|\rho(t)\|_{L^\infty_x}
+\|\rho(t)\|_{\dot B^a_{\infty,\infty}}
+\|\rho(t)\|_{\dot B^a_{1,\infty}}
+\|\nabla_x\rho(t)\|_{L^\infty_x}
+\|\nabla_x\rho(t)\|_{L^1_x}
\le C(T_*)\varepsilon_0.
$$
Since $\langle t\rangle\sim 1$ on $[0,T_*]$, this implies
\[
\left\{
\begin{aligned}
\|\rho(t)\|_{L^\infty_x}
&\le C(T_*)\varepsilon_0 \langle t\rangle^{-d+\gamma},\\
\|\rho(t)\|_{\dot B^a_{\infty,\infty}}
&\le C(T_*)\varepsilon_0 \langle t\rangle^{-(d+a)+\gamma},\\
\|\rho(t)\|_{\dot B^a_{1,\infty}}
&\le C(T_*)\varepsilon_0 \langle t\rangle^{-a+\gamma},\\
\|\nabla_x\rho(t)\|_{L^\infty_x}
&\le C(T_*)\varepsilon_0 \langle t\rangle^{-(d+1)+\gamma},\\
\|\nabla_x\rho(t)\|_{L^1_x}
&\le C(T_*)\varepsilon_0 \langle t\rangle^{-1+\gamma}.
\end{aligned}
\right.
\]
Choosing
$$
A\ge 2C(T_*),
$$
we obtain \eqref{eq:improved-bootstrap} on $t\in[0,T_*]$.

\medskip
\noindent
\textit{Step 2. The large-time regime.}
By Propositions \ref{Besovestimates} and \ref{daoshu009}, for all $t\ge1$,
\begin{equation}\label{eq:S6-large}
\left\{
\begin{aligned}
\|\rho(t)\|_{L^\infty_x}
&\le C(\varepsilon_0+\varepsilon)\langle t\rangle^{-d},\\
\|\rho(t)\|_{\dot B^a_{\infty,\infty}}
&\le C(\varepsilon_0+\varepsilon)\langle t\rangle^{-(d+a)},\\
\|\rho(t)\|_{\dot B^a_{1,\infty}}
&\le C(\varepsilon_0+\varepsilon)\langle t\rangle^{-a},\\
\|\nabla_x\rho(t)\|_{L^\infty_x}
&\le C(\varepsilon_0+\varepsilon)\langle t\rangle^{-(d+1)}\ln(1+t),\\
\|\nabla_x\rho(t)\|_{L^1_x}
&\le C(\varepsilon_0+\varepsilon)\langle t\rangle^{-1}\ln(1+t).
\end{aligned}
\right.
\end{equation}
Using \eqref{eq:S6-eps-ratio}, we obtain
\[
\left\{
\begin{aligned}
\|\rho(t)\|_{L^\infty_x}
&\le 2C\varepsilon\,\langle t\rangle^{-\gamma}\langle t\rangle^{-d+\gamma},\\
\|\rho(t)\|_{\dot B^a_{\infty,\infty}}
&\le 2C\varepsilon\,\langle t\rangle^{-\gamma}\langle t\rangle^{-(d+a)+\gamma},\\
\|\rho(t)\|_{\dot B^a_{1,\infty}}
&\le 2C\varepsilon\,\langle t\rangle^{-\gamma}\langle t\rangle^{-a+\gamma},\\
\|\nabla_x\rho(t)\|_{L^\infty_x}
&\le 2C\varepsilon\,\tfrac{\ln(1+t)}{\langle t\rangle^\gamma}
   \langle t\rangle^{-(d+1)+\gamma},\\
\|\nabla_x\rho(t)\|_{L^1_x}
&\le 2C\varepsilon\,\tfrac{\ln(1+t)}{\langle t\rangle^\gamma}
   \langle t\rangle^{-1+\gamma}.
\end{aligned}
\right.
\]
Since $\gamma>0$, one can choose $T_*\ge1$, depending only on $C$ and $\gamma$, such that
$$
2C\langle t\rangle^{-\gamma}\le \tfrac12,
\qquad
2C\tfrac{\ln(1+t)}{\langle t\rangle^\gamma}\le \tfrac12,
\qquad \forall\, t\ge T_*.
$$
Therefore, \eqref{eq:improved-bootstrap} also holds for all $t\in[T_*,T)$.

Combining the two steps, we obtain \eqref{eq:improved-bootstrap} for every $t\in[0,T)$. This closes the bootstrap and completes the proof.
\hfill$\Box$

\vspace{2mm}

\noindent\textbf{Proof of Corollary \ref{tuilun}.}
The argument is analogous to the scattering argument in
\cite[Corollary 1.1]{HanKwanD2021}, but we provide the details for the present two-species system.

Set
\[
\widetilde Y_t^\pm(x,v):=Y^\pm_{0,t}(x+tv,v),\qquad
\widetilde W_t^\pm(x,v):=W^\pm_{0,t}(x+tv,v).
\]
By the characteristic formula,
\[
f^\pm(t,x+tv,v)
=f^\pm_{\mathrm{in}}\big(x+\widetilde Y_t^\pm(x,v),\,v+\widetilde W_t^\pm(x,v)\big)
\mp\big[\mu\bigl(v+\widetilde W_t^\pm(x,v)\bigr)-\mu(v)\big].
\]
Moreover, by the definition of the characteristic flow, for any $0\le t^*_1<t^*_2$,
\[
|\widetilde W_{t^*_2}^\pm-\widetilde W_{t^*_1}^\pm|
\le \int_{t^*_1}^{t^*_2}\|E(s)\|_{L^\infty_x}\,ds,
\qquad
|\widetilde Y_{t^*_2}^\pm-\widetilde Y_{t^*_1}^\pm|
\le \int_{t^*_1}^{t^*_2}s\|E(s)\|_{L^\infty_x}\,ds.
\]
Using Theorem~\ref{Theorem1}, we have
\[
\|E(s)\|_{L^\infty_x}\lesssim \varepsilon_0\frac{\ln(1+s)}{\langle s\rangle^{d+1}},
\quad s\geq0,
\]
and hence \(\widetilde Y_t^\pm,\widetilde W_t^\pm\) are Cauchy in \(L^\infty_{x,v}\).
This defines \(Y_\infty^\pm,W_\infty^\pm\), and moreover
\[
\|\widetilde Y_t^\pm-Y_\infty^\pm\|_{L^\infty_{x,v}}
+
\|\widetilde W_t^\pm-W_\infty^\pm\|_{L^\infty_{x,v}}
\lesssim
\varepsilon_0\frac{\ln(1+t)}{\langle t\rangle^{d-1}}.
\]
The stated scattering estimate then follows from the mean value theorem,
the boundedness of $\nabla_{x,v}f^\pm_{\mathrm{in}}$, and the regularity of $\mu$.
\hfill$\Box$

\vspace{2mm}

\noindent\textbf{Proof of \eqref{dispersive-e}.}
By \eqref{midufenjie001} and the triangle inequality,
\begin{equation}\label{eq:rho-pm-decomp}
\|\rho^\pm(t)\|_{L^\infty_x}
\le \|Q(t)\|_{L^\infty_x}+\|I^\pm(t)\|_{L^\infty_x}+\|R^\pm(t)\|_{L^\infty_x}.
\end{equation}
Moreover, the definition of $Q$ in \eqref{midufenjie001} gives
\begin{equation}\label{eq:Q-bound-short}
\|Q(t)\|_{L^\infty_x} \leq \frac{1}{2} \left(\|I^+(t) - I^-(t)\|_{L^\infty_x} + \|R^+(t)\|_{L^\infty_x} + \|R^-(t)\|_{L^\infty_x} + \|\rho(t)\|_{L^\infty_x} \right).
\end{equation}

We next estimate \(I^\pm\).
For \(t\ge1\) and \(x\in\mathbb R^d\). Using the change of variables
\(y=\mathcal X^\pm(0)\) and Lemma~\ref{Lemma3.4-YKB}, we obtain
\[
|I^\pm(t,x)|
\lesssim
\int_{\mathbb R^d}
\sup_{v\in\mathbb R^d}|f^\pm_{\mathrm{in}}(y,v)|
\,\big|\det(\nabla_v\mathcal X^\pm(0))\big|^{-1}\,dy
\lesssim
\langle t\rangle^{-d}\int_{\mathbb R^d}\sup_{\eta}|f^\pm_{\mathrm{in}}(y,\eta)|\,dy
\lesssim \langle t\rangle^{-d}.
\]
If \(0\le t\le1\), then the same change of variables
 together with Lemma~\ref{Lemma3.4-YKB} yields
\[
|I^\pm(t,x)|
\lesssim
\int_{\mathbb R^d}\sup_{\eta}|f^\pm_{\mathrm{in}}(y,\eta)|\,dy
\lesssim 1
\lesssim \langle t\rangle^{-d}.
\]
Therefore,
\begin{equation}\label{eq:Ipm-Linfty}
\|I^\pm(t)\|_{L^\infty_x}\lesssim \langle t\rangle^{-d},
\qquad \forall\, t\ge0.
\end{equation}

On the other hand, by \eqref{zuoyong028},
\begin{equation}\label{eq:Rpm-Linfty}
\|R^\pm(t)\|_{L^\infty_x}\lesssim \varepsilon^2\langle t\rangle^{-d},
\end{equation}
while \eqref{eq:improved-bootstrap} gives
\begin{equation}\label{eq:rho-Linfty}
\|\rho(t)\|_{L^\infty_x}\lesssim \langle t\rangle^{-d}.
\end{equation}
Since
\[
\|I^+(t)-I^-(t)\|_{L^\infty_x}
\le \|I^+(t)\|_{L^\infty_x}+\|I^-(t)\|_{L^\infty_x},
\]
it follows from \eqref{eq:Q-bound-short}, \eqref{eq:Ipm-Linfty}, \eqref{eq:Rpm-Linfty}, and \eqref{eq:rho-Linfty} that
\[
\|Q(t)\|_{L^\infty_x}\lesssim \langle t\rangle^{-d}.
\]
Substituting this bound, together with \eqref{eq:Ipm-Linfty} and \eqref{eq:Rpm-Linfty} into \eqref{eq:rho-pm-decomp}, we obtain the estimate \eqref{dispersive-e}.
\hfill$\Box$

\section{Appendix}\label{diqijie}
For reference, we gather in this appendix the proofs of several auxiliary statements used throughout the paper.

\noindent\textbf{Proof of Lemma \ref{field01}.}
Let $\mathcal{J}_2:=(I-\Delta)^{-1}$ denote the Bessel potential operator. Its kernel is
\[
G_2(x)=(4\pi)^{-d/2}\int_0^\infty e^{-\tau} e^{-\frac{|x|^2}{4\tau}} \tau^{\frac{2-d}{2}} \, \frac{d\tau}{\tau}.
\]
Since $T\psi=\nabla \mathcal{J}_2\psi$, we have
\[
T\psi=(\nabla G_2*)\psi=G_2*(\nabla\psi),
\qquad
\nabla T\psi=(\nabla G_2)*(\nabla\psi).
\]
Hence, by Young's inequality and $\|G_2\|_{L^1}=1$ (see \cite[Proposition 1.2.5]{Hao13}), for $p\in\{1,\infty\}$,
\[
\|T\psi\|_{L^p_x}\le \|\nabla\psi\|_{L^p_x},
\qquad
\|\nabla T\psi\|_{L^p_x}\le \|\nabla G_2\|_{L^1_x}\,\|\nabla\psi\|_{L^p_x}.
\]
Moreover, if $\nabla^2\psi\in L^\infty(\mathbb{R}^d)$, then
\[
\|\nabla^2T\psi\|_{L^\infty_x}
\le \|\nabla G_2\|_{L^1_x}\,\|\nabla^2\psi\|_{L^\infty_x}.
\]
It remains to verify that $\nabla G_2\in L^1(\mathbb R^d)$.

A direct differentiation gives
\[
\partial_{x_i}G_2(x)
=-\frac{x_i}{2}(4\pi)^{-d/2}
\int_0^\infty e^{-\tau}e^{-\frac{|x|^2}{4\tau}}
\tau^{-d/2-1}\,d\tau.
\]

\medskip
If $|x|\ge2$, then
\[
\tau+\frac{|x|^2}{4\tau}=\frac{\tau}{2}+\left(\frac{\tau}{2}+\frac{|x|^2}{8\tau}\right)+\frac{|x|^2}{8\tau}
 \geq \frac{\tau}{2}+\frac{1}{2\tau}+\frac{|x|}{2},
\]
and therefore
\[
|\nabla G_2(x)| \leq C(d) |x| e^{-\frac{|x|}{2}}, \qquad  |x| \geq 2,
\]
which is integrable on $\{|x|\ge2\}$.

\medskip
If $|x|<2$, we split
\[
|\nabla G_2(x)| \leq \frac{|x|}{2} (4\pi)^{-d/2}\Big( \int_0^{|x|^2} + \int_{|x|^2}^4 + \int_4^\infty \Big) e^{-\tau} e^{-\frac{|x|^2}{4\tau}} \tau^{-\frac{d}{2}-1} \, d\tau := A_1(x) + A_2(x) + A_3(x).
\]
For $A_1$, using the change of variables $y=\frac{|x|^2}{4\tau}$,
\[
A_1(x)
\lesssim |x|
\int_0^{|x|^2} e^{-\frac{|x|^2}{4\tau}}\tau^{-\frac{d}{2}-1}\,d\tau
= C |x|^{1-d}
\int_{1/4}^{\infty} e^{-y} y^{\frac{d}{2}-1}\,dy\lesssim |x|^{1-d}.
\]
For $A_2$, since $e^{-\tau-\frac{|x|^2}{4\tau}}\le1$,
\[
A_2(x) \leq C |x| \int_{|x|^2}^4 \tau^{-\frac{d}{2}-1} \, d\tau\leq \frac{2C}{d} |x|^{1-d}.
\]
For $A_3$, since $e^{-\frac{|x|^2}{4\tau}}\le1$ for $\tau\ge4$,
\[
A_3(x) \leq C |x| \int_4^\infty e^{-\tau} \tau^{-\frac{d}{2}-1}\,  d\tau \leq C |x|.
\]
Hence
\[
|\nabla G_2(x)| \lesssim |x|^{1-d} + |x|, \  \  |x| < 2.
\]
Since $|x|^{1-d}$ is integrable near the origin for $d\ge2$, we conclude that $\nabla G_2\in L^1(\mathbb{R}^d)$.
This completes the proof.
\hfill$\Box$
%%%%%%%%%%%%%%%%%%%%%%%%%%%%%%%%%%%%%%%%%
\vspace{2mm}

\noindent{\bf Proof of Lemma~\ref{field02}.}
Let
\[
m(\xi):=\frac{i\xi}{1+|\xi|^2},
\qquad
m_2(\xi):=\frac{\xi\otimes \xi}{1+|\xi|^2},
\]
so that
\[
\widehat{T\psi}(\xi)=m(\xi)\widehat{\psi}(\xi),
\qquad
\widehat{\nabla_x T\psi}(\xi)=m_2(\xi)\widehat{\psi}(\xi).
\]
For each \(j\in\mathbb Z\), define
\[
K_j:=\mathcal F^{-1}\!\big(m(\xi)\varphi(2^{-j}\xi)\big),
\qquad
L_j:=\mathcal F^{-1}\!\big(m_2(\xi)\varphi(2^{-j}\xi)\big).
\]

Since \(m\) and \(m_2\) are smooth on the support of \(\varphi(2^{-j}\cdot)\), we have
\begin{equation}\label{eq:dyadicL1}
\sup_{j\in\mathbb Z}\|K_j\|_{L^1}+\sup_{j\in\mathbb Z}\|L_j\|_{L^1}\le C.
\end{equation}
Moreover,
\begin{equation}\label{eq:dyadicconv}
\dot\Delta_j(T\psi)=K_j*\dot\Delta_j \psi,
\qquad
\dot\Delta_j(\nabla T\psi)=L_j*\dot\Delta_j \psi.
\end{equation}

\medskip
\textit{(i)} Let $p\in\{1,\infty\}$. Combining \eqref{eq:dyadicconv} with Young's inequality and
\eqref{eq:dyadicL1}, we obtain for all $j$,
$$
\|\dot\Delta_j(T\psi)\|_{L^p}\le \|K_j\|_{L^1}\|\dot\Delta_j\psi\|_{L^p}\le C\|\dot\Delta_j\psi\|_{L^p}.
$$
Multiplying by $2^{ja}$ and taking the supremum over $j$ gives
\[
\|T\psi\|_{\dot B^{a}_{p,\infty}}\le C\|\psi\|_{\dot B^{a}_{p,\infty}}.
\]

\medskip
\textit{(ii)} By \eqref{FTL0} and \eqref{field02:Besov1} with $p=\infty$, we get
\[
\|T\psi\|_{\dot F^a_{\infty,\infty}}
=\|T\psi\|_{\dot B^a_{\infty,\infty}}
\le C\|\psi\|_{\dot B^a_{\infty,\infty}}.
\]

\medskip
\textit{(iii)} We first prove that, for every \(\phi\in W^{1,1}(\mathbb R^d)\cap L^1(\mathbb R^d)\),
\begin{equation}\label{eq:L22-general-F11}
\|\phi\|_{\dot F^a_{1,\infty}}
\le C\|\phi\|_{L^1_x}^{1-a}\|\nabla_x\phi\|_{L^1_x}^{a}.
\end{equation}
Indeed,
\[
\|\phi\|_{\dot F^a_{1,\infty}}
\le
\sum_{j\in\mathbb Z}2^{ja}\|\dot\Delta_j\phi\|_{L^1_x}.
\]
Fix \(J\in\mathbb Z\), to be chosen later, and split
\[
\sum_{j\in\mathbb Z}2^{ja}\|\dot\Delta_j\phi\|_{L^1}
=
\sum_{j\le J}2^{ja}\|\dot\Delta_j\phi\|_{L^1}
+
\sum_{j>J}2^{ja}\|\dot\Delta_j\phi\|_{L^1}.
\]

For the low frequencies, since each $\dot\Delta_j$ is an $L^1$-bounded convolution operator, we have
\[
\|\dot\Delta_j\phi\|_{L^1}\lesssim \|\phi\|_{L^1},
\]
and hence
\[
\sum_{j\le J}2^{ja}\|\dot\Delta_j\phi\|_{L^1}
\lesssim
\sum_{j\le J}2^{ja}\|\phi\|_{L^1}
\lesssim
2^{Ja}\|\phi\|_{L^1}.
\]

For the high frequencies, by Bernstein's inequality,
\[
\|\dot\Delta_j\phi\|_{L^1}
\lesssim
2^{-j}\|\nabla \dot\Delta_j\phi\|_{L^1},
\]
and therefore
\[
\sum_{j>J}2^{ja}\|\dot\Delta_j\phi\|_{L^1}
\lesssim
\sum_{j>J}2^{j(a-1)}\|\nabla\phi\|_{L^1}
\lesssim
2^{J(a-1)}\|\nabla\phi\|_{L^1}, \quad a\in(0,1).
\]
Thus
\[
\|\phi\|_{\dot F^{a}_{1,\infty}}
\lesssim
2^{Ja}\|\phi\|_{L^1}
+
2^{J(a-1)}\|\nabla\phi\|_{L^1}.
\]
Choosing \(2^J\sim \|\nabla_x\phi\|_{L^1}/\|\phi\|_{L^1}\), we obtain \eqref{eq:L22-general-F11}.

Finally, using \eqref{eq:L22-general-F11} with \(\phi=T\psi\), we obtain \eqref{F03}
This completes the proof.
\hfill$\Box$

%%%%%%%%%%%%%%%%%%%%%%%%%%%%%%%%%%%%%%%55
\vspace{2mm}

\noindent{\bf Proof of Lemma \ref{lem:difference-F}.}
For \(N\in\mathbb N\), let
\[
\varphi^\ast_N:=\sum_{|j|\le N}\dot\Delta_j\varphi^\ast.
\]
Since \(\varphi^\ast_N\) is smooth and the estimates below are uniform in \(N\), it suffices to prove the lemma for smooth functions with finitely many nonzero dyadic blocks.
 Set
\[
G(x):=\sup_{j\in\mathbb Z}2^{ja}|\dot\Delta_j\varphi^\ast(x)|,
\qquad
\|\varphi^\ast\|_{\dot F^{a}_{p,\infty}}=\|G\|_{L^p},  \qquad p\in\{1,\infty\}.
\]

\medskip
\noindent\textit{(i)}
Fix $\alpha\neq0$ and choose $k\in\mathbb Z$ such that $2^{-k-1}<|\alpha|\le2^{-k}$.
Decompose
\[
\delta_\alpha\varphi^\ast
=\sum_{j\le k}\delta_\alpha\dot\Delta_j\varphi^\ast
+\sum_{j>k}\delta_\alpha\dot\Delta_j\varphi^\ast
:=I_{\mathrm{low}}+I_{\mathrm{high}}.
\]

For $j>k$, we use
\[
|\delta_\alpha\dot\Delta_j\varphi^\ast(x)|
\le |\dot\Delta_j\varphi^\ast(x)|+|\dot\Delta_j\varphi^\ast(x-\alpha)|,
\]
which yields
\[
\frac{|I_{\mathrm{high}}(x)|}{|\alpha|^a}
\lesssim
2^{ka}\sum_{j>k}\big(|\dot\Delta_j\varphi^\ast(x)|+|\dot\Delta_j\varphi^\ast(x-\alpha)|\big)
\lesssim G(x)+G(x-\alpha).
\]

For $j\le k$, the mean value formula and Bernstein's inequality give
\[
|\delta_\alpha\dot\Delta_j\varphi^\ast(x)|
\lesssim |\alpha|\,2^j\|\dot\Delta_j\varphi^\ast\|_{L^\infty},
\]
and hence
\[
\frac{|I_{\mathrm{low}}(x)|}{|\alpha|^a}
\lesssim |\alpha|^{1-a}\sum_{j\le k}2^j\|\dot\Delta_j\varphi^\ast\|_{L^\infty}
\lesssim \|G\|_{L^\infty}=
\|\varphi^\ast\|_{\dot F^a_{\infty,\infty}}.
\]
Taking the supremum over $x$ and $\alpha$ yields \eqref{eq:dq-infty}.

\medskip
\noindent\textit{(ii)}
Fix $|\alpha|\le1$ and choose $k\ge0$ such that $2^{-k-1}<|\alpha|\le2^{-k}$.
Write $\varphi^\ast=S_k\varphi^\ast+(\varphi^\ast-S_k\varphi^\ast)$ with $S_k\varphi^\ast:=\sum_{j\le k}\dot\Delta_j\varphi^\ast$.

For the high-frequency part, arguing as above,
\[
\Big\|\sup_{2^{-k-1}<|\alpha|\le2^{-k}}
\frac{|\delta_\alpha(\varphi^\ast-S_k\varphi^\ast)|}{|\alpha|^a}\Big\|_{L^1}
\lesssim \|G\|_{L^1}.
\]

For the low-frequency part, by the mean value formula,
\[
\frac{|\delta_\alpha(S_k\varphi^\ast)(x)|}{|\alpha|^a}
\le |\alpha|^{1-a}\int_0^1|\nabla S_k\varphi^\ast(x-\theta\alpha)|\,d\theta,
\]
and hence
\[
\Big\|\frac{|\delta_\alpha(S_k\varphi^\ast)|}{|\alpha|^a}\Big\|_{L^1}
\le |\alpha|^{1-a}\|\nabla S_k\varphi^\ast\|_{L^1}.
\]
By Bernstein's inequality,
\[
\|\nabla S_k\varphi^\ast\|_{L^1}
\lesssim \sum_{j\le k}2^{j(1-a)}\|G\|_{L^1}
\lesssim 2^{k(1-a)}\|G\|_{L^1}.
\]
Since $|\alpha|^{1-a}\sim 2^{-k(1-a)}$, we conclude
\[
\Big\|\sup_{2^{-k-1}<|\alpha|\le2^{-k}}
\frac{|\delta_\alpha(S_k\varphi^\ast)|}{|\alpha|^a}\Big\|_{L^1}
\lesssim \|G\|_{L^1}=
\|\varphi^\ast\|_{\dot F^a_{1,\infty}}.
\]
Combining the low/high contributions and taking the supremum over $|\alpha|\le 1$ yields \eqref{eq:dq-1-local}.
\hfill$\Box$

%%%%%%%%%%%%%%%%%%%%%%%%%%%%%%%%%55

\vspace{2mm}

\noindent{\bf Proof of Proposition \ref{Local solutions}.}
We divide the proof into three steps.

\noindent\textit{Step 1. Construction of approximate solutions and uniform estimates.}
Set
\[
f_0^\pm(t,x,v):=f_{\rm in}^\pm(x,v),
\qquad
\rho_0^\pm(t,x):=\int_{\mathbb R^d}f_{\rm in}^\pm(x,v)\,dv,
\qquad
\rho_0:=\rho_0^+-\rho_0^-,
\]
and let
\[
E_0:=\nabla_x(I-\Delta_x)^{-1}\rho_0.
\]
Assume inductively that $f_n^\pm$ is defined on $[0,t_1]$, and define
\[
\rho_n^\pm(t,x):=\int_{\mathbb R^d} f_n^\pm(t,x,v)\,dv,
\qquad
\rho_n:=\rho_n^+-\rho_n^-,
\qquad
E_n:=\nabla_x(I-\Delta_x)^{-1}\rho_n.
\]
We then define $f_{n+1}^\pm$ as the solution to
\begin{equation}\label{VKG01}
\left\{
\begin{aligned}
&\partial_t f_{n+1}^+ + v\cdot\nabla_x f_{n+1}^+ + E_n\cdot\nabla_v f_{n+1}^+
   = -E_n\cdot\nabla_v\mu,\\
&\partial_t f_{n+1}^- + v\cdot\nabla_x f_{n+1}^- - E_n\cdot\nabla_v f_{n+1}^-
   = E_n\cdot\nabla_v\mu,\\
&f_{n+1}^\pm(0,x,v)=f_{\rm in}^\pm(x,v).
\end{aligned}
\right.
\end{equation}

Assume that, for some $n\ge0$,
$$
\sup_{0\le t\le t_1}
\sum_{\pm}
\Bigl(
\|f_n^\pm(t)\|_{W^{1,1}(\mathbb R^{2d})}
+\|f_n^\pm(t)\|_{W^{2,\infty}_k(\mathbb R^{2d})}
\Bigr)
\le 2(M_0+1).
$$
Since $k>d$, this implies
\[
\|\rho_n(t)\|_{W^{1,1}_x}+\|\rho_n(t)\|_{W^{1,\infty}_x}\le C(M_0+1).
\]
Hence Lemma~\ref{field01} yields
\begin{equation}\label{eq:P31-En-final}
\|E_n(t)\|_{L^1_x}+\|E_n(t)\|_{L^\infty_x}
+\|\nabla_xE_n(t)\|_{L^1_x}+\|\nabla_xE_n(t)\|_{L^\infty_x}
\le C_0(M_0+1).
\end{equation}
Moreover, since $k>d$,
\[
\|\nabla_x^2\rho_n^\pm(t)\|_{L^\infty_x}
\le \int_{\mathbb R^d} |\nabla_x^2 f_n^\pm(t,x,v)|\,dv
\le C\|\langle v\rangle^k\nabla_x^2 f_n^\pm(t)\|_{L^\infty_{x,v}},
\]
and therefore Lemma~\ref{field01} yields
\begin{equation}\label{estimate9}
\|\nabla^2_x E_n(t)\|_{L^\infty_x} \leq C \|\nabla^2_x \rho_n(t)\|_{L^\infty_x}
\leq C (M_0 + 1).
\end{equation}
Therefore the vector fields $(v,\pm E_n)$ are globally Lipschitz on $[0,t_1]\times\mathbb R^{2d}$, and the corresponding characteristic flow,
for $0\le s\le t\le t_{1}$,
\begin{equation*}
\left\{
  \begin{array}{ll}
    \dot{\mathcal{X}}_n^{\pm}(s;t,x,v)=\mathcal{V}_n^{\pm}(s;t,x,v), & \mathcal{X}_n^{\pm}(t;t,x,v)=x,  \\  [2mm]
    \dot{\mathcal{V}}_n^{\pm}(s;t,x,v)=\pm E_n(s,\mathcal{X}_n^{\pm}(s;t,x,v)), & \mathcal{V}_n^{\pm}(t;t,x,v)=v,
  \end{array}
\right.
\end{equation*}
is globally well-defined.
Moreover,
\[
\nabla_{x,v}\cdot (\mathcal{V}_n^\pm,\pm E_n)=0,
\]
and hence the flow is measure-preserving in $(x,v)$.
For brevity, we write
\[
(\mathcal X_n^\pm(s),\mathcal V_n^\pm(s))
:=
(\mathcal X_n^\pm(s;t,x,v),\mathcal V_n^\pm(s;t,x,v)).
\]
Along the characteristics,
\begin{equation}\label{vlasov051}
f_{n+1}^\pm(t,x,v)
=
f_{\rm in}^\pm\bigl(\mathcal{X}_n^\pm(0),\mathcal{V}_n^\pm(0)\bigr)
\mp
\int_0^t
E_n\bigl(s,\mathcal{X}_n^\pm(s)\bigr)\cdot
\nabla_v\mu\bigl(\mathcal{V}_n^\pm(s)\bigr)\,ds.
\end{equation}

We next derive uniform bounds for $f_{n+1}^\pm$.
Using \eqref{vlasov051}, the measure-preserving property, \eqref{eq:P31-En-final}, and
$\nabla_v\mu\in L^1_v$, we obtain
\[
\|f_{n+1}^\pm(t)\|_{L_v^1 L_x^1}
\le M_0 + C\,C_0(M_0+1)t.
\]
Differentiating \eqref{VKG01} in $x$ and $v$, integrating the resulting equations along characteristics, and using \eqref{eq:P31-En-final}, we get
\begin{align*}
\|\nabla_x f_{n+1}^\pm(t)\|_{L_v^1 L_x^1}
&\le
\|\nabla_x f_{\rm in}^\pm\|_{L_v^1 L_x^1}
+
C\int_0^t
\|\nabla_xE_n(s)\|_{L^1_x}
\Bigl(
\|\nabla_v f_{n+1}^\pm(s)\|_{L^\infty_{x,v}}+1
\Bigr)\,ds,\\
\|\nabla_v f_{n+1}^\pm(t)\|_{L_v^1 L_x^1}
&\le
\|\nabla_v f_{\rm in}^\pm\|_{L_v^1 L_x^1}
+
\int_0^t\|\nabla_x f_{n+1}^\pm(s)\|_{L_v^1 L_x^1}\,ds
+
C\int_0^t\|E_n(s)\|_{L^1_x}\,ds.
\end{align*}

On the other hand, along the characteristics,
\[
\frac{d}{ds}\langle \mathcal{V}_n^\pm(s)\rangle^k
\le C |E_n(s,\mathcal{X}_n^\pm(s))|\langle \mathcal{V}_n^\pm(s)\rangle^{k-1},
\]
so that
\begin{equation}\label{eq:P31-V-short}
\langle \mathcal{V}_n^\pm(s)\rangle\le C\langle v\rangle,
\qquad 0\le s\le t\le t_1.
\end{equation}
Using \eqref{vlasov051}, \eqref{eq:P31-V-short}, and the decay of $\mu$, we infer
\[
\|\langle v\rangle^k f_{n+1}^\pm(t)\|_{L^\infty_{x,v}}
\le M_0 + C\,C_0(M_0+1)t.
\]
Similarly, differentiating \eqref{VKG01} once and twice, multiplying by
$\langle v\rangle^k$, and integrating along characteristics, we obtain
\begin{align*}
\|\langle v\rangle^k\nabla_{x,v}f_{n+1}^\pm(t)\|_{L^\infty_{x,v}}
&\le
CM_0
+C\int_0^t\bigl(1+C_0(M_0+1)\bigr)
\Bigl(
1+\|\langle v\rangle^k\nabla_{x,v}f_{n+1}^\pm(s)\|_{L^\infty_{x,v}}
\Bigr)\,ds,\\
\|\langle v\rangle^k\nabla^2_{x,v}f_{n+1}^\pm(t)\|_{L^\infty_{x,v}}
&\le
CM_0
+C\int_0^t\bigl(1+C_0(M_0+1)\bigr)
\Bigl(
1+\|f_{n+1}^\pm(s)\|_{W^{2,\infty}_k}
\Bigr)\,ds.
\end{align*}
Here we used that, after commuting \eqref{VKG01} with all second-order derivatives,
the resulting source terms involve only $\nabla_xE_n$, $\nabla_x^2E_n$, first- and second-order
derivatives of $f_{n+1}^\pm$, and derivatives of $\mu$ up to order three; hence
\eqref{eq:P31-En-final}-\eqref{estimate9} suffice to close the $W^{2,\infty}_k$ estimate.
Combining the preceding estimates and summing over the two species, we have
\[
\sum_\pm
\Bigl(
\|f_{n+1}^\pm(t)\|_{W^{1,1}}
+\|f_{n+1}^\pm(t)\|_{W^{2,\infty}_k}
\Bigr)
\le
C(M_0+1)
+
C\bigl(1+C_0(M_0+1)\bigr)
\int_0^t
\Bigl(
1+\sum_\pm \|f_{n+1}^\pm(s)\|_{W^{2,\infty}_k}
\Bigr)\,ds.
\]
Hence, by Gronwall's inequality,
\[
\sup_{0\le t\le t_1}
\sum_\pm
\Bigl(
\|f_{n+1}^\pm(t)\|_{W^{1,1}}
+\|f_{n+1}^\pm(t)\|_{W^{2,\infty}_k}
\Bigr)
\le
C(M_0+1)\exp\!\bigl(C(1+C_0(M_0+1))t_1\bigr).
\]
Recalling that
\[
t_1=\min\Bigl\{1,\frac{c_*}{C_0(M_0+1)}\Bigr\},
\]
with $c_*>0$ sufficiently small depending only on
$d$, $k$, and $\mu$, we infer that
\[
\sup_{0\le t\le t_1}
\sum_\pm
\Bigl(
\|f_{n+1}^\pm(t)\|_{W^{1,1}}
+\|f_{n+1}^\pm(t)\|_{W^{2,\infty}_k}
\Bigr)
\le 2(M_0+1).
\]
This closes the induction, and therefore
\begin{equation}\label{eq:P31-uniform-short}
\sup_{n\ge0}\sup_{0\le t\le t_1}
\sum_\pm
\Bigl(
\|f_n^\pm(t)\|_{W^{1,1}}
+\|f_n^\pm(t)\|_{W^{2,\infty}_k}
\Bigr)
\le 2(M_0+1).
\end{equation}

Finally, \eqref{VKG01}, \eqref{eq:P31-En-final}, and \eqref{eq:P31-uniform-short} imply
\[
\sup_{n\ge0}\sup_{0\le t\le t_1}
\sum_\pm
\Bigl(
\|\partial_t f_n^\pm(t)\|_{W^{1,1}}
+\|\langle v\rangle^{k-1}\partial_t f_n^\pm(t)\|_{W^{1,\infty}}
\Bigr)
\le C.
\]
Hence
\[
f_n^\pm\in
C\bigl([0,t_1];W^{1,1}(\mathbb R^{2d})\cap W^{1,\infty}_{k-1}(\mathbb R^{2d})\bigr)
\cap
L^\infty\bigl([0,t_1];W^{2,\infty}_k(\mathbb R^{2d})\bigr),
\]
uniformly in $n$.

\medskip
\noindent\textit{Step 2. Convergence and uniqueness.}
Set
\[
d_n(t):=
\sup_{0\le s\le t}
\Bigl(
\|\rho_{n+1}^+(s)-\rho_n^+(s)\|_{L^\infty_x}
+
\|\rho_{n+1}^-(s)-\rho_n^-(s)\|_{L^\infty_x}
\Bigr).
\]
By Lemma~\ref{field01},
\[
\|E_n(s)-E_{n-1}(s)\|_{L^\infty_x}\le C\, d_{n-1}(s).
\]
Subtracting the characteristic systems associated with $E_n$ and $E_{n-1}$, and using the uniform Lipschitz bound on $\nabla_xE_n$, we obtain
\[
\sup_{0\le \tau\le t}
\Bigl(
|\mathcal{X}_n^\pm(\tau)-\mathcal{X}_{n-1}^\pm(\tau)|
+
|\mathcal{V}_n^\pm(\tau)-\mathcal{V}_{n-1}^\pm(\tau)|
\Bigr)
\le C t\, d_{n-1}(t).
\]
Subtracting the representation formulas \eqref{vlasov051}, and using the uniform $W^{1,\infty}_k$ bounds from Step~1,
\[
|f_{n+1}^\pm(t,x,v)-f_n^\pm(t,x,v)|
\le C t\, d_{n-1}(t)\langle v\rangle^{-k}.
\]
Integrating in $v$ and using $k>d$, we get
\[
d_n(t)\le C t\, d_{n-1}(t),\qquad 0\le t\le t_1.
\]
Taking $c_*>0$ smaller if necessary so that $Ct_1\le \frac12$, we deduce
\[
d_n(t_1)\le \frac12 d_{n-1}(t_1).
\]
Hence $\{\rho_n^\pm\}$ is Cauchy in $C([0,t_1];L^\infty_x(\mathbb R^{d}))$, and so is $\{E_n\}$ in
$C([0,t_1];L^\infty_x(\mathbb R^{d}))$.
By the uniform Lipschitz bounds, the characteristic flows converge uniformly on compact time intervals to the flow generated by the limit field $E$, and passing to the limit in \eqref{vlasov051} gives a solution $f^\pm$ of \eqref{VKG}.
Moreover, \eqref{eq:P31-uniform-short} and the uniform convergence of the flows imply that
\[
f^\pm\in
C\bigl([0,t_1];W^{1,1}(\mathbb R^{2d})\cap W^{1,\infty}_{k-1}(\mathbb R^{2d})\bigr)
\cap
L^\infty\bigl([0,t_1];W^{2,\infty}_k(\mathbb R^{2d})\bigr).
\]
Uniqueness follows from the same difference estimate applied to two solutions with the same initial
data, since both solutions belong to
\[
C\bigl([0,t_1];W^{1,1}(\mathbb R^d\times\mathbb R^d)\cap W^{1,\infty}_{k-1}(\mathbb R^d\times\mathbb R^d)\bigr)
\cap
L^\infty\bigl([0,t_1];W^{2,\infty}_k(\mathbb R^d\times\mathbb R^d)\bigr),
\]
and therefore satisfy the same uniform bounds as the approximate sequence.

\medskip
\noindent\textit{Step 3. Propagation of the smallness of the net charge density.}
Since $f:=f^+-f^-$,
subtracting the two equations in \eqref{VKG}, we obtain
\[
\partial_t f+v\cdot\nabla_x f
=
-E\cdot\nabla_v(f^++f^-)-2E\cdot\nabla_v\mu.
\]
Applying the free-transport Duhamel formula and integrating in $v$, we get
\begin{equation}\label{eq:P31-rho-final}
\begin{aligned}
\rho(t,x)
={}&
\int_{\mathbb R^d}f_{\rm in}(x-tv,v)\,dv
-2\int_0^t\int_{\mathbb R^d}
E(s,x-v(t-s))\cdot\nabla_v\mu(v)\,dv\,ds\\
&\quad
-\int_0^t\int_{\mathbb R^d}
E(s,x-v(t-s))\cdot
(\nabla_vf^++\nabla_vf^-)(s,x-v(t-s),v)\,dv\,ds.
\end{aligned}
\end{equation}
Using Minkowski's inequality, the translation-invariance of the $L^p_x$ norm, $p\in\{1,\infty\}$, Lemma~\ref{field01}, the bound $k>d$, and the estimates obtained in Steps~1--2, we obtain
\begin{equation}\label{eq:P31-rho-est-final}
\|\rho(t)\|_{L^p_x}
\le
\|f_{\rm in}\|_{L^1_vL^p_x}
+
C(M_0+1)\int_0^t \|\rho(s)\|_{L^p_x}\,ds.
\end{equation}
Differentiating \eqref{eq:P31-rho-final} in $x$, the same argument gives
\begin{equation}\label{eq:P31-gradrho-est-final}
\|\nabla_x\rho(t)\|_{L^p_x}
\le
\|\nabla_xf_{\rm in}\|_{L^1_vL^p_x}
+
C(M_0+1)\int_0^t
\Bigl(
\|\rho(s)\|_{L^p_x}+\|\nabla_x\rho(s)\|_{L^p_x}
\Bigr)\,ds.
\end{equation}

Define
\[
\mathcal W(t):=
\|\rho(t)\|_{L^1_x}
+\|\rho(t)\|_{L^\infty_x}
+\|\nabla_x\rho(t)\|_{L^1_x}
+\|\nabla_x\rho(t)\|_{L^\infty_x}.
\]
By \eqref{eq:P31-rho-est-final} and \eqref{eq:P31-gradrho-est-final},
\[
\mathcal W(t)
\le
\mathcal W(0)
+
C(M_0+1)\int_0^t \mathcal W(s)\,ds.
\]
Under the additional assumption of the proposition,
\[
\mathcal W(0)
=
\|f_{\rm in}\|_{L^1_vL^1_x}
+\|f_{\rm in}\|_{L^1_vL^\infty_x}
+\|\nabla_xf_{\rm in}\|_{L^1_vL^1_x}
+\|\nabla_xf_{\rm in}\|_{L^1_vL^\infty_x}
\le \varepsilon_0.
\]
Therefore,
\[
\mathcal W(t)\le \varepsilon_0 e^{C(M_0+1)t},
\qquad 0\le t\le t_1.
\]
By the definition of $t_1$, after decreasing $c_*$ if necessary, we may assume
\[
e^{C(M_0+1)t_1}\le 2,
\]
and hence
\[
\sup_{0\le t\le t_1}
\Bigl(
\|\rho(t)\|_{L^1_x}
+\|\rho(t)\|_{L^\infty_x}
+\|\nabla_x\rho(t)\|_{L^1_x}
+\|\nabla_x\rho(t)\|_{L^\infty_x}
\Bigr)
\le 2\varepsilon_0.
\]
This completes the proof.
\hfill$\Box$

\bigskip

\noindent {\bf Acknowledgments:} %The authors would like to thank anonymous referees for their helpful comments and valuable suggestions concerning the presentation of this paper.
The work of Yi Wang is partially supported by NSFC grants (Grant No. 12421001 and 12288201) and CAS Project for Young Scientists in Basic Research, Grant No. YSBR-031. The work of Hang Xiong  is partially supported by China Postdoctoral Science Foundation 2021TQ0351 and 2021M693336.
Hang Xiong thanks Professor Quoc-Hung Nguyen for the fruitful discussions and kind helps.
\bigskip

\noindent {\bf Availability of data and material:}
Data sharing not applicable to this article as no datasets were generated or
analysed during the current study.
\bigskip

\noindent {\bf Declarations}
\bigskip

\noindent {\bf Conflict of interest:} The authors declared that they have no Conflict of interest to this work.
\bigskip

%\newpage

\end{document}